\documentclass[12pt,reqno]{amsart}
\usepackage[letterpaper]{geometry}

\usepackage[utf8]{inputenc}
\usepackage{amsmath}
\usepackage{amsthm}
\usepackage{amssymb}
\usepackage{amsfonts}
\usepackage{amsaddr}
\usepackage{tabu}
\usepackage{float}
\usepackage{graphicx}
\usepackage{accents}
\usepackage{listings}
\usepackage{epstopdf}
\usepackage{color}
\usepackage{subcaption}
\usepackage{hyperref}
\usepackage{cancel}
\usepackage{multicol}
\usepackage{mathtools}
\usepackage{pgfplots}
\pgfplotsset{compat=newest}
\usepackage{tikz}
\usetikzlibrary{decorations.markings}
\usetikzlibrary{snakes}

\usepackage{caption}
\DeclareCaptionType[placement={!ht}]{listing}[Listing][Code Listings]

\usetikzlibrary{shapes.misc}

\tikzset{cross/.style={cross out, draw=black, minimum size=2*(#1-\pgflinewidth), inner sep=0pt, outer sep=0pt},
	cross/.default={1pt}}

\DeclareMathOperator*{\esssup}{ess\,sup}

\newtheorem{thm}{Theorem}[section]

\newtheorem{lem}[thm]{Lemma}
\newtheorem{prop}[thm]{Proposition}
\newtheorem{prob}[thm]{Problem}
\theoremstyle{definition}
\newtheorem{defn}[thm]{Definition}
\theoremstyle{remark}
\newtheorem{rem}[thm]{Remark}

\numberwithin{equation}{section}

\begin{document}
	
\title[Stabilization of higher order Schrödinger equations]{Stabilization of higher order  Schrödinger equations on a finite interval: Part II}%

\author{Türker Özsarı$^{\MakeLowercase{a},*}$ and Kemal Cem Yılmaz$^{\MakeLowercase{b}}$}
\address{$^a$Department of Mathematics, Bilkent University\\ Çankaya, Ankara, 06800 Turkey}
\address{$^b$Department of Mathematics, Izmir Institute of Technology\\ Urla, Izmir, 35430 Turkey}
\thanks{*Corresponding author: Türker Özsarı, turker.ozsari@bilkent.edu.tr}

\begin{abstract}
Backstepping based controller and observer models were designed for higher order linear and nonlinear Schrödinger equations on a finite interval in \cite{batal2020} where the controller was assumed to be acting from the left endpoint of the medium. In this companion paper, we further the analysis by considering boundary controller(s) acting at the right endpoint of the domain.  It turns out that the problem is more challenging in this scenario as the associated boundary value problem for the backstepping kernel becomes overdetermined and lacks a smooth solution.  The latter is essential to switch  back and forth between the original plant and the so called target system. To overcome this difficulty we rely on the strategy of using an imperfect kernel, namely one of the boundary conditions in kernel PDE model is disregarded. The drawback is that one loses rapid stabilization in comparison with the left endpoint controllability. Nevertheless,  the exponential decay of the $L^2$-norm with a certain rate still holds. The observer design is associated with new challenges from the point of view of wellposedness and one has to prove smoothing properties for an associated initial boundary value problem with inhomogeneous boundary data.  This problem is solved by using Laplace transform in time.  However, the Bromwich integral that inverts the transformed solution is associated with certain analyticity issues which are treated through a subtle analysis. Numerical algorithms and simulations verifying the theoretical results are given.

\end{abstract}

\maketitle
\newpage
\tableofcontents

\section{Introduction}
\subsection{Statements of problems and main results}
Backstepping based controller and observer models were designed for higher order linear and nonlinear Schrödinger equations on a finite interval in \cite{batal2020} where the controller was assumed to be acting from the left endpoint of the medium. In this companion paper, we further the
analysis by considering boundary controller(s) acting at the right endpoint of the domain.  We will consider only the linear higher order Schrödinger (HLS) equation in this  paper:
	\begin{eqnarray} \label{plant_lin}
		\begin{cases}
			iu_t + i\beta u_{xxx} +\alpha u_{xx} +i\delta u_x = 0, x\in (0,L), t\in (0,T),\\
			u(0,t)=0, u(L,t)=h_0(t), u_x(L,t)=h_1(t),\\
			u(x,0)=u_0(x),
		\end{cases}
	\end{eqnarray}
	where $\beta > 0$, $\alpha, \delta \in \mathbb{R}$, $h_0(t) = h_0(u(\cdot,t))$ and $h_1(t) = h_1(u(\cdot,t))$ are feedbacks acting at the right endpoint of the domain. The control design results of this paper can be extended to associated higher-order nonlinear  Schrödinger equations
$$iu_t + i\beta u_{xxx} +\alpha u_{xx} +i\delta u_x + f(u) = 0$$  as in Part I (see \cite{batal2020}) with additional assumptions on the coefficients, but this topic is omitted here considering the volume of current text and postponed to a future paper.  In addition, it is also possible to consider other sets of boundary conditions here as in Part I that involves second order traces such as $$u(0,t)=0, u_x(L,t)=h_0(t), u_{xx}(L,t)=h_1(t),$$ but this will also be discussed in another place.
	
	The higher-order nonlinear Schrödinger equation was originally given by
	\begin{equation} \label{hnls_or}
	i u_t + \frac{1}{2} u_{xx} + |u|^2u + \epsilon i \left(\beta_1 u_{xxx} + \beta_2 (|u|^2u)_x + \beta_3 u |u|^2_x\right) = 0,
	\end{equation}
	which has been used to describe the evolution of femtosecond pulse propagation in a nonlinear optical fiber \cite{koda85,koda87}. In this equation the first term represents the evolution, second term is the group velocity dispersion, third term is self-phase modulation, fourth term is the higher order linear dispersive term, fifth term is related to self-steepening and sixth term is related to self-frequency shift due to the stimulated Raman scattering. In the absence of the last three terms, the model becomes classical nonlinear Schrödinger equation (NLS) which describes slowly varying wave envelopes in a dispersive medium. It has applications in several fields of physics such as plasma physics, solid-state physics, nonlinear optics. It also describes the propagation of picosecond optical pulse in a mono-mode fiber \cite{xu2002soliton}. However, for the pulses in the femtosecond regime, the NLS equation becomes inadequate and higher order nonlinear and dispersive terms become crucial. See \cite{agrawal} for a detailed discussion of the higher order effects upon the propagation of an optical pulse.
	
	Higher order linear and nonlinear Schrödinger equations were studied from the point of many different aspects.   Regarding the wellposedness of solutions, we refer the reader to \cite{Carvajal03,Carvajal04,Carvajal06,staffilani2,Laurey97,staffilani1,Taka00}. A numerical study of this problem was given in \cite{Caval19}. From the controllability and stabilization perspective, we refer the reader to \cite{Ceballos05} for exact boundary controllability, \cite{Bis07} and \cite{chen2018} for internal feedback stabilization and \cite{batal2020} for boundary feedback stabilization.

	From a practical point of view, stabilization of solutions is necessary in order to prevent the transmission of an undesirable pulse propagation. Our study offers a practical solution to this issue because: (i) the stabilization is fast, i.e. the absorption effect is exponential and (ii) the control acts only from the boundary which is desirable when access to medium is limited.  In the absence of feedback controllers, $L^2-$norm of the solution satisfies
	\begin{equation} \label{wdecayL2}
		\frac{d}{dt} \|u(\cdot,t)\|_{L^2(0,L)}^2 = -\beta |u_x(0,t)|^2 \leq 0.
	\end{equation}
	This can be shown by taking $L^2-$inner product of the main equation in \eqref{plant_lin} with $u$ and applying integration by parts. From the estimate \eqref{wdecayL2}, we infer that $L^2-$norm of the solution does not increase in time. Furthermore, it was shown that if $\alpha^2 + 3\beta\delta > 0$ and
	\begin{equation} \label{critic_len}
		L \in \mathcal{N} \doteq \left\{2\pi \beta \sqrt{\frac{k^2 + kl + l^2}{3\beta \delta + \alpha^2}}: k,l \in \mathbb{Z}^+  \right\},
	\end{equation}
	then the $L^2-$norm of the solution does not necessarily decay to zero. Here $\mathcal{N}$ is the set of critical lengths in the context of exact boundary controllability for the HLS (see \cite{Ceballos05, dasilva} for the derivation of this set of critical lengths). For instance choosing the coefficients $\beta = 1$, $\alpha = 2$ and $\delta = 8$ with $k = 1$ and $l = 2$, we obtain $L = \pi \in \mathcal{N}$. Moreover, choosing the initial state as
	\begin{equation*}
		u_0(x) = 3 - e^{4ix} - 2e^{-2ix},
	\end{equation*}
	we see that $u(x,t) = u_0(x)$ solves \eqref{plant_lin}. Therefore, we find a time--independent solution with a constant energy if no control acts on the system.
	
	In this paper, we are interested in constructing suitable feedback controllers to make sure that we can steer all solutions to zero with an exponential rate of decay on domains of both critical and uncritical lengths. More precisely, we consider the problem below:
	
	\begin{prob} \label{prob1}
		Given $L>0$, find $\lambda > 0$ and feedback control laws $h_0(t) = h_0(u(\cdot,t))$ and $h_1(t) = h_1(u(\cdot,t))$ such that the solution of \eqref{plant_lin} satisfies $\|u(\cdot,t)\|_{L^2(0,L)} = O(e^{-\lambda t})$ for some $t > 0$.
	\end{prob}
	
	In order to solve this problem, we use backstepping method (see \cite{KrsBook} for a general discussion on the backstepping method), which is a well studied method for the second order evolutionary partial differential equations \cite{smys4,Liu03,smys1,smys2,smys3}. In recent years, researchers studied backstepping stabilization of several higher order evolutionary equations that include third order dispersion term \cite{batal2020,cerpacoron2013,Marx18,eda,Tang2013,Tang2015}. In these studies on KdV type equations, a single boundary feedback control is located at one endpoint and the number of boundary conditions located at the opposite endpoint are two. In particular, Part I of this study \cite{batal2020} assumes a control input acting from the left endpoint, and there are two homogeneous boundary conditions that are imposed from the right endpoint. Conversely, if there are two boundary controllers acting from the right endpoint and a single homogeneous boundary condition imposed at the left endpoint, which is the subject of the present paper, the situation becomes mathematically very different as we explain below.
	
	To this end, we want to transform the original plant via the \emph{backstepping} transformation
	\begin{equation}\label{bt}
		w(x,t) = [(I - \Upsilon_k) u](x,t) \doteq u(x,t) - \int_0^x k(x,y) u(y,t)dy
	\end{equation}
	to a target system which already has the desired exponential stability. The classical approach is to take the linearly damped version of the same type of pde with homogeneous boundary conditions:
	\begin{eqnarray} \label{tar_lin}
	\begin{cases}
	iw_t + i\beta w_{xxx} +\alpha w_{xx} +i\delta w_x  + ir w= 0, x\in (0,L), t\in (0,T),\\
	w(0,t)= w(L,t)= w_x(L,t)=0,\\
	w(x,0)=w_0(x) \doteq u_0(x) - \int_0^x k(x,y) u_0(y)dy.
	\end{cases}
	\end{eqnarray}
	The key point here is to be able to show the existence of a sufficiently smooth kernel and that the transformation $I - \Upsilon_k$ has a bounded inverse on a suitable space. Then, the wellposedness and stability properties for the target model will also be true for the original plant. Figure \ref{fig:bs_figure} below summarizes the standard algorithm of the backstepping method.
	\begin{figure}[H]
		\begin{tikzpicture}
		\node[draw] (plant) at (0,0)
		{\begin{tabular}{l l}
			\textbf{Linear Plant}  &\\
			\multicolumn{2}{l}{State variable $u(x,t)$}
			\end{tabular}};
		
		\node[draw] (target) at (7.5,0)
		{\begin{tabular}{l l}
			\textbf{Target}  &\\
			\multicolumn{2}{l}{State variable $w(x,t)$}
			\end{tabular}};
		
		\draw [->,bend right=-18] (plant) to node [above,midway] {$w(x,t) = u(x,t) - \int_0^x k(x,y) u(y,t)dy$} (target);
		
		\draw [->,bend right=-18] (target) to node [below,midway] {Inverse transformation} (plant);
		\end{tikzpicture}
		\caption{Backstepping} \label{fig:bs_figure}
	\end{figure}
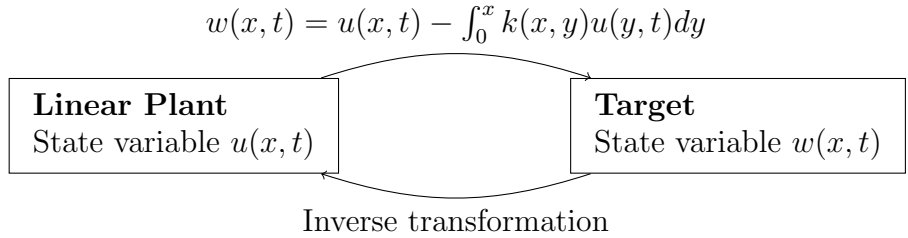
	
	To prove the existence of a kernel, we differentiate \eqref{bt} and use the original plant together with the target system to see what conditions $k$ must satisfy. After some calculations (see Appendix \ref{app1} for details), we deduce that $k$ must solve the following boundary value problem
	\begin{eqnarray}\label{kernel_k}
	\begin{cases}
	\beta(k_{xxx}+k_{yyy})-i\alpha(k_{xx}-k_{yy})+\delta(k_x+k_y)+rk=0, \\
	k(x,x)= k_y(x,0) = k(x,0)=0,\\
	k_x(x,x)=\frac{rx}{3\beta},
	\end{cases}
	\end{eqnarray}
	where $(x,y)$ belongs to the triangular region $\Delta_{x,y}\doteq\{(x,y)\in \mathbb{R}^2\,|\,y\in (0,x), x\in (0,L)\}$. To solve this problem, we change variables as
	$s=x - y$ and $t = y$ and write $G(s,t) \equiv k(x,y)$. Then \eqref{kernel_k} transforms into
	\begin{eqnarray}\label{kernel_G}
	\begin{cases}
	\beta(3G_{sst}-3G_{tts}+G_{ttt})+i\alpha(G_{tt}-2G_{ts})+\delta G_t+ r G=0,\\
	G(0,t)= G_t(s,0) = G(s,0)=0,\\
	G_s(0,t)=\frac{rt}{3\beta},
	\end{cases}
	\end{eqnarray}
	where $(s,t)$ belongs to $\Delta_{s,t}\doteq\{(s,t)\in \mathbb{R}^2\,|\,t\in (0,L-s), s\in (0,L)\}$. See Figure \ref{fig:delta} for transformation of the triangular region under the above change of variables.
	\begin{figure}[H]
		\begin{tikzpicture}
		\draw [black, fill={rgb:orange,1;yellow,2;pink,5}] (0,0) -- (2.5,2.5) -- (2.5,0) -- (0,0);
		\draw [thick, <->] (0,3) node[above]{$y$} -- (0,0) -- (3,0) node[right]{$x$};
		\filldraw
		(2.5,0) circle (1pt) node[align=center, below] {$L$};
		\filldraw
		(0,2.5) circle (1pt) node[align=center, left] {$L$};
		\draw[dotted]
		(0,2.5) -- (2.5,2.5);
		\node[align=center] (title) at (1.5,-0.7) {{Triangular region $\Delta_{x,y}$}};
		
		\node  (A) at (3,1.8)  {};
		\node  (B) at (6,1.8)  {};
		\draw[->,thick]  (A.north)  to [bend left = 30]  node[above,yshift=12pt] {$s = x - y$} node[above] {$t = y$} (B.north);
		
		\draw [black, fill={rgb:orange,1;yellow,2;pink,5}] (7,2.5) -- (7,0) -- (9.5,0) --  (7,2.5);
		\draw [thick, <->] (7,3) node[above]{$t$} -- (7,0) -- (10,0) node[right]{$s$};
		\filldraw
		(9.5,0) circle (1pt) node[align=center, below] {$L$};
		\filldraw
		(7,2.5) circle (1pt) node[align=center, left] {$L$};
		\node[align=center] (title) at (8.5,-0.7) {{Triangular region $\Delta_{s,t}$}};
		\end{tikzpicture}
		\caption{Triangular regions} \label{fig:delta}
	\end{figure}
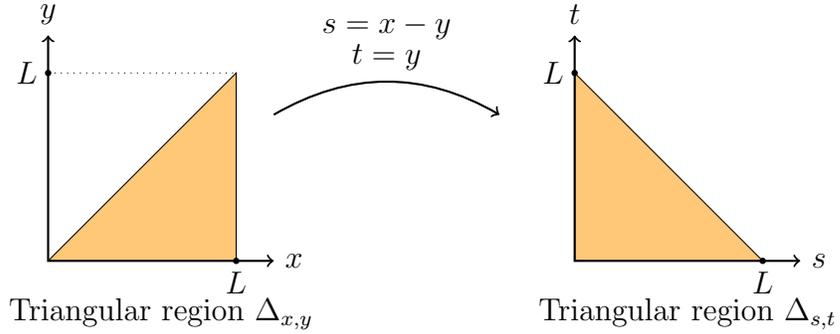
	Observe that there is a mismatch between $G_t(s,0)$ and $G_s(0,t)$ in the sense that $0=G_{ts}(0,0) \neq G_{st}(0,0)=\frac{r}{3\beta}$. This implies that \eqref{kernel_G} cannot have a smooth solution and the standard algorithm of backstepping method fails. This issue was previously observed in Korteweg de-Vries equation \cite{cerpacoron2013} and later treated in the case of uncritical domains in \cite{coronlu2014} and in the case of critical domains in \cite{BatalOzsari2018-1}.  The idea of the latter work was to drop one of the boundary conditions from the kernel pde model and take $r$ sufficiently small.  Note that if $r$ is small, then the mismatch is also small and one can hope that the solution of the corrected pde model will yield a kernel which is good enough for our purposes. Once we drop the boundary condition $G_t(s,0)=0$ from \eqref{kernel_G}, the corrected version of the pde model \eqref{kernel_G} becomes
	\begin{eqnarray}\label{kernel_pG}
		\begin{cases}
			\beta(3G^*_{sst}-3G^*_{tts}+G^*_{ttt})+i\alpha(G^*_{tt}-2G^*_{ts})+\delta G^*_t+ r G^*=0,\\
			G^*(0,t)= G^*(s,0)=0,\\
			G^*_s(0,t)=\frac{rt}{3\beta},
		\end{cases}
	\end{eqnarray}
	where $(s,t)\in \Delta_{s,t}$. Setting $k^*(x,y)=G^*(s,t)$, we deduce that $k^*$ is the sought after solution of
	\begin{eqnarray} \label{kernel_pk}
		\begin{cases}
			\beta(k^*_{xxx}+k^*_{yyy})-i\alpha(k^*_{xx}-k^*_{yy})+\delta(k^*_x+k^*_y)+rk^*=0, \\
			k^*(x,x)= k^*(x,0)=0,\\
			k^*_x(x,x)=\frac{rx}{3\beta},
		\end{cases}
	\end{eqnarray}
	where $(x,y)\in\Delta_{x,y}$.  Existence of a smooth $k^*$ is given in Lemma \ref{lemkernel}.
	
	If we use the backstepping transformation
	\begin{equation} \label{p_bt}
		w^*(x,t) = u(x,t) - \int_0^x k^*(x,y) u(y,t) dy,
	\end{equation}
	then the corresponding target model becomes
	\begin{eqnarray} \label{p_tar_lin}
	\begin{cases}
	iw^*_t + i\beta w^*_{xxx} +\alpha w^*_{xx} +i\delta w^*_x  + ir w^*= i\beta k^*_y(x,0) w^*_x(0,t), x\in (0,L), t\in (0,T),\\
	w^*(0,t)= w^*(L,t)= w^*_x(L,t)=0,\\
	w^*(x,0)=w^*_0(x) \doteq u_0(x) - \int_0^x k^*(x,y)u_0(y)dy.
	\end{cases}
	\end{eqnarray}
	See \eqref{k_cond_1}-\eqref{k_cond_5} in Appendix \ref{app1} for details. Notice that \eqref{p_tar_lin} is a modified version of \eqref{tar_lin} in the sense that there is a trace term at the right hand side of the main equation. This trace term is due to disregarding the condition $k_y(x,0)=0$ and using the relation $u_x(0,t) = w_x^*(0,t)$. It is shown in Proposition \ref{ctrl_tar_stab} that the solution of \eqref{p_tar_lin} exponentially decays to zero for small $r$. Furhermore, it is important that the transformation \eqref{p_bt} has a bounded inverse (see Lemma \ref{inverselem}). Graphical illustration of the new scheme is shown in Figure \ref{fig:pseudo-bs}.
	\begin{figure}[H]
		\begin{tikzpicture}
		\node[draw] (plant) at (0,0)
		{\begin{tabular}{l l}
			\textbf{Linear Plant} \\ \textbf{}  &\\
			\multicolumn{2}{l}{State variable $u(x,t)$}
			\end{tabular}};
		
		\node[draw] (target) at (7.5,0)
		{\begin{tabular}{l l}
			\textbf{Modified target} \\ \textbf{with a trace term}  &\\
			\multicolumn{2}{l}{State variable $w^*(x,t)$}
			\end{tabular}};
		
		\draw [->,bend right=-18] (plant) to node [above,midway] {$w^*(x,t) = u(x,t) - \int_0^x k^*(x,y) u(y,t)dy$} (target);
		
		\draw [->,bend right=-18] (target) to node [below,midway] {Inverse transformation} (plant);
		\end{tikzpicture}
		\caption{Backstepping with an imperfect kernel}
		\label{fig:pseudo-bs}
	\end{figure}
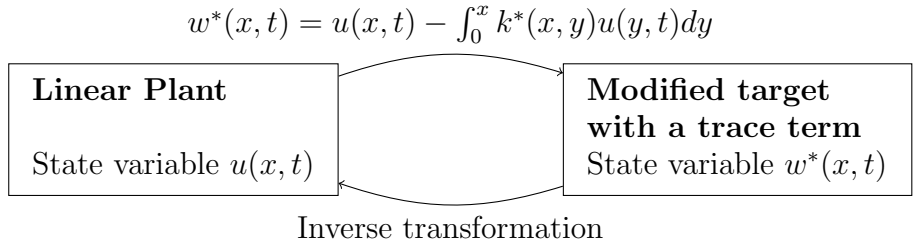
	Based on the above strategy feedback controllers take the following forms:
	\begin{equation} \label{ctrl_controllers}
		h_0(t) = \int_0^L k^*(L,y)u(y,t)dy, \quad h_1(t) = \int_0^L k_x^*(L,y)u(y,t)dy.
	\end{equation}

	Let us introduce the space $$X_T^\ell \doteq C([0,T];H^\ell(0,L)) \cap L^2(0,T;H^{\ell+1}(0,L)), \ell\ge 0.$$ Now, regarding plant \eqref{plant_lin}, we prove the following theorem.
	\begin{thm} \label{ctrl_thm}
		Let $T, \beta > 0$, $\alpha, \delta \in \mathbb{R}$, $u_0 \in L^2(0,L)$. Assume that the right endpoint feedback controllers $h_0(t)$, $h_1(t)$ are given by \eqref{ctrl_controllers} and  let $k^*$ be a smooth backstepping kernel solving \eqref{kernel_pk}. Then, we have the following:
		\begin{enumerate}
			\item [(i)]{(Wellposedness)} \eqref{plant_lin} has a unique solution $u \in X_T^0$ satisfying also $u_x \in C([0,L];L^2(0,T))$. Moreover, if $u_0 \in H^3(0,L)$ and $w_0^* = (I - \Upsilon_{k^*})[u_0]$ satisfies the  compatibility conditions, then $u \in X_T^3$.
			
			\item[(ii)]{(Decay)} Suppose $u_0 \in L^2(0,L)$. Then, there is $r > 0$ such that $$\lambda = \beta\left(\frac{r}{\beta} -  \frac{\|k^*_y(\cdot,0;r)\|_{L^2(0,L)}^2}{2}\right) > 0.$$ Moreover, the solution $u$ of \eqref{plant_lin} with feedback controllers \eqref{ctrl_controllers} satisfies
			\begin{equation*}
				\|u(\cdot,t)\|_{L^2(0,L)} \leq c_{k^*} \|u_0\|_{L^2(0,L)} e^{-\lambda t},\quad t \geq 0,
			\end{equation*}
			where $c_{k^*}$ is a nonnegative constant which depends on $k^*$.
		\end{enumerate}
	\end{thm}

	\begin{rem}
		It is shown in the proof of Proposition \ref{ctrl_tar_stab} that there exists $r > 0$ which guarantees the condition $\lambda > 0$. See also Table \ref{table_decay} for different values of $r, \lambda$. Note also from the same table that smallness of $r$ is essential for $\lambda$ to be positive.
	\end{rem}

	In the second part of the paper, we consider the case where the state of the system is not fully measurable, in particular at time $t=0$. However, we assume that the first order boundary trace $y_1(t) = u_x(0,t)$ and the second order boundary trace $y_2(t) = u_{xx}(0,t)$ are known, say detectable through boundary sensors. In order to deal with the robustness of the state, we construct an observer, which uses the given boundary measurements. To this end, we propose the following observer model
	\begin{eqnarray} \label{obs_lin}
		\begin{cases}
			i\hat{u}_t + i\beta \hat{u}_{xxx} +\alpha \hat{u}_{xx} +i\delta \hat{u}_x - p_1(x) (y_1(t) - \hat{u}_{x}(0,t)) \\
			- p_2(x) (y_2(t) - \hat{u}_{xx}(0,t))= 0, \quad x\in (0,L), t\in 	(0,T),\\
			\hat{u}(0,t)=0, \hat{u}(L,t)=h_0(t), \hat{u}_x(L,t)=h_1(t),\\
			\hat{u}(x,0)=\hat{u}_0(x).
		\end{cases}
	\end{eqnarray}
	In the above model, the feedback controllers $h_0(t)$, $h_1(t)$ depend on the state of the observer model. The same controllers will also be applied to the original plant \eqref{plant_lin}. $p_1$, $p_2$ are set to be observer gains and they will be constructed in such a way that the so called error $\tilde{u} = u - \hat{u}$ must approach to zero as $t$ gets larger. More precisely, we want to solve the following problem:
	\begin{prob} \label{prob2}
		Given $L > 0$, find observer gains $p_1$, $p_2$, and feedback laws $h_0(t) = h_0(\hat{u}(\cdot,t))$, $h_1(t) = h_1(\hat{u}(\cdot,t))$ such that there exists $\lambda > 0$ for which the solution $u$ of \eqref{plant_lin} satisfies $\|u(\cdot,t)\|_{L^2(0,L)} = O(e^{-\lambda t})$.
	\end{prob}	
	Note that the error function $\tilde{u}$ satisfies
	\begin{eqnarray} \label{err_lin}
	\begin{cases}
	i\tilde{u}_t + i\beta \tilde{u}_{xxx} +\alpha \tilde{u}_{xx} +i\delta \tilde{u}_x \\
+ p_1(x) \tilde{u}_x(0,t)+ p_2(x) \tilde{u}_{xx}(0,t)= 0, \quad x\in (0,L), t\in (0,T),\\
	\tilde{u}(0,t)=\tilde{u}(L,t)=\tilde{u}_x(L,t)=0,\\
	\tilde{u}(x,0)=\tilde{u}_0(x).
	\end{cases}
	\end{eqnarray}
	In order to guarantee the decay of solutions of \eqref{err_lin} at an exponential rate, we treat $p_1$ and $p_2$ as control inputs and suitably construct them via the backstepping technique. To this end, we transform \eqref{err_lin} using the transformation
	\begin{equation} \label{btp}
		\tilde{u}(x,t) = \tilde{w}(x,t) - \int_0^x p(x,y) \tilde{w}(y,t)dy,
	\end{equation}
	to a target error model
	\begin{eqnarray} \label{err_tar_lin_rem}
		\begin{cases}
			i\tilde{w}_t + i\beta \tilde{w}_{xxx} +\alpha \tilde{w}_{xx} +i\delta \tilde{w}_x + ir \tilde{w} = 0, \quad x\in (0,L), t\in (0,T),\\
			\tilde{w}(0,t)=\tilde{w}(L,t)=\tilde{w}_x(L,t)= 0,\\
			\tilde{w}(x,0)=\tilde{w}_0(x),
		\end{cases}
	\end{eqnarray}
	which has exponential decay property. Differentiating \eqref{btp} and using \eqref{err_lin} with \eqref{err_tar_lin_rem}, one can see that $p$ must satisfy the following boundary value problem
	\begin{eqnarray} \label{kernel_p_old}
		\begin{cases}
			\beta(p_{xxx}+p_{yyy})-i\alpha(p_{xx}-p_{yy})+\delta(p_x+p_y)-rp=0, \\
			p(x,x)= p(L,y)= p_x(L,y)0,\\
			p_x(x,x) = \frac{r}{3\beta}(L - x),
		\end{cases}
	\end{eqnarray}
	on $\Delta_{x,y}$ (see Appendix \ref{app2} for detailed calculations). However, changing variables as $s = L - x$, $t = x - y$ and defining $H(s,t) \equiv p(x,y)$, one can see that the resulting boundary value problem is overdetermined in the sense that there is a mismatch between the boundary conditions: $0 = H_{st}(0,0) \neq H_{ts}(0,0) = \frac{r}{3\beta}$. Therefore, there cannot exist a smooth kernel satisfying all boundary conditions. Following a similar approach as in the earlier part of the paper, we could consider disregarding one of the boundary conditions, namely $p_x(L,y) = 0$, and take $r$ sufficiently small. Then, the corrected version of the pde model \eqref{kernel_p_old} becomes
	\begin{eqnarray} \label{kernel_p}
		\begin{cases}
			\beta(p^*_{xxx}+p^*_{yyy})-i\alpha(p^*_{xx}-p_{yy})+\delta(p^*_x+p^*_y)-rp^*=0, \\
			p^*(x,x)= p^*(L,y)=0,\\
			p^*_x(x,x) = \frac{r}{3\beta}(L - x).
		\end{cases}
	\end{eqnarray}
	
	Now if we use the backstepping transformation
	\begin{equation} \label{bt_p}
		\tilde{u}(x,t) = \tilde{w}^*(x,t) - \int_0^x p^*(x,y) \tilde{w}^*(y,t)dy,
	\end{equation}
	where $p^*$ solves \eqref{kernel_p}, then the corresponding target error model for \eqref{tar_lin} becomes
	\begin{eqnarray} \label{err_tar_lin}
		\begin{cases}
			i\tilde{w}^*_t + i\beta \tilde{w}^*_{xxx} +\alpha \tilde{w}^*_{xx} +i\delta \tilde{w}^*_x + ir \tilde{w}^* = 0, \quad x\in (0,L), t\in (0,T),\\
			\tilde{w}^*(0,t)=\tilde{w}^*(L,t)=0,
			\tilde{w}^*_x(L,t)=\int_0^L p^*_x(L,y) \tilde{w}^*(y,t)dy,\\
			\tilde{w}^*(x,0)=\tilde{w}^*_0(x),
		\end{cases}
	\end{eqnarray}
	if the control gains are chosen such that $p_1(x) = i\beta p^*_y(x,0) - \alpha p^*(x,0)$ and $p_2(x) = -i\beta p^*(x,0)$ (see Appendix \ref{app2} for details). Note that nonhomogeneous Neumann type boundary condition in \eqref{err_tar_lin} is due to disregarding the condition $p_x(L,y) = 0$. Nevertheless, we still have the exponential decay of solutions of \eqref{err_tar_lin} assuming that $r$ is sufficiently small. This is given in Proposition \ref{tar_err_stab}. Also, it is not difficult to see that $p^*$ and $k^*$ are related via
	\begin{equation*}
		p^*(x,y) \equiv k^*(L-y,L-x;-r),
	\end{equation*}
	which immediately guarantees the existence of a smooth kernel $p^*$. This yields the exponential decay of the solution of the error model.
	
	Next, we apply the backstepping method to the observer model. To this end, we differentiate
	\begin{equation} \label{obs_trans}
		\hat{w}^*(x,t) = \hat{u}(x,t) - \int_0^x k^*(x,y)\hat{u}(y,t)dy
	\end{equation}
	and use \eqref{obs_lin} to deduce that $\hat{w}^*$ solves target observer model given by
	\begin{eqnarray} \label{tar_p_obs_lin}
	\begin{cases}
			i\hat{w}^*_t + i\beta \hat{w}^*_{xxx} +\alpha \hat{w}^*_{xx} +i\delta \hat{w}^*_x  + ir 	\hat{w}^*= i\beta k^*_y(x,0) \hat{w}^*_x(0,t) \\
			+ [(I - \Upsilon_{k^*})p_1](x) \tilde{w}_{x}(0,t) + [(I - \Upsilon_{k^*})p_2](x) \tilde{w}_{xx}(0,t), \quad x\in (0,L), t\in (0,T),\\
			\hat{w}^*(0,t)= \hat{w}^*(L,t)= \hat{w}^*_x(L,t)=0,\\
			\hat{w}^*(x,0)=\hat{w}^*_0(x) \doteq \hat{u}_0(x) - \int_0^x k^*(x,y)\hat{u}(y,t)dy.
		\end{cases}
	\end{eqnarray}
	Notice that we still use the solution of the corrected kernel pde model in the backstepping transformation given in \eqref{obs_trans}. Therefore, as in the first part of this paper, an extra trace term shows up in the main equation of \eqref{tar_p_obs_lin}. We prove in Proposition \ref{tar_obs_stab} that the solution of \eqref{tar_p_obs_lin} exponentially decays to zero in time, again assuming that $r$ is sufficiently small.
	
	Thanks to the bounded invertibility of the transformations $(I - \Upsilon_{k^*})$, $(I - \Upsilon_{p^*})$, stability estimates with same decay rates also hold for observer and error models. We have the following theorem for the wellposedness and stabilization of the plant-observer-error system \eqref{plant_lin}-\eqref{obs_lin}-\eqref{err_lin}:
	\begin{thm} \label{obs_thm}
		Let $T, \beta > 0$, $\alpha, \delta \in \mathbb{R}$, and $u_0, \hat{u}_0 \in H^6(0,L)$. Assume that the right endpoint feedback controllers are given by
		\begin{equation*}
			h_0(t) = \int_0^L k^*(L,y;r) \hat{u}(y,t) dy, \quad  	h_1(t) = \int_0^L k_x^*(L,y;r) \hat{u}(y,t) dy,
		\end{equation*}
		where $k^*$ and $p^*$ are smooth solutions of \eqref{kernel_pk} and \eqref{kernel_p}, respectively. Then, we have the following:
		\begin{enumerate}
			\item [(i)]{(Wellposedness)} Suppose that $\hat{w}_0^* = (I - \Upsilon_{k^*})[\hat{u}_0]$ satisfies the compatibility conditions and the pair $(\tilde{w}_0^*,\psi)$, where $\tilde{w}_0^* = (I - \Upsilon_{p^*})^{-1}[\tilde{u}_0]$, $\psi = \psi(\tilde{w}) \doteq \int_0^L p_x^*(L,y)\tilde{w}(y,t)dy$, satisfies the higher order compatibility conditions. Then the plant-observer-error system \eqref{plant_lin}-\eqref{obs_lin}-\eqref{err_lin} has a unique solution $(u, \hat{u},\tilde{u}) \in X_T^3 \times X_T^3 \times X_T^6$.
			\item [(ii)]{(Decay)} Moreoever, for sufficiently small $r > 0$, there exists $\mu > \nu > 0$ such that,
			\begin{align*}
			\|u(\cdot,t)\|_{L^2(0,L)} \lesssim& c_{k,p}\left(\|\hat{u}_0\|_{L^2(0,L)} + c_p \|u_0 - \hat{u}_0\|_{H^3(0,L)}\right) e^{-\nu t}\\
			 &+ \|u_0 - \hat{u}_0\|_{L^2(0,L)} e^{-\mu t}, \\
			\|\hat{u}(\cdot,t)\|_{L^2(0,L)} \lesssim& c_{k,p} \left( \|\hat{u}_0\|_{L^2(0,L)} + \|u_0 - \hat{u}_0\|_{H^3(0,L)} \right) e^{-\nu t}, \\
			\|(u - \hat{u})(\cdot,t)\|_{L^2(0,L)} \lesssim& c_p \|u_0 - \hat{u}_0\|_{L^2(0,L)} e^{-\mu t}, \\
			\|(u - \hat{u})(\cdot,t)\|_{H^3(0,L)} \lesssim& c_{p^\prime} \|u_0 - \hat{u}_0\|_{H^3(0,L)} e^{-\mu t}.
			\end{align*}
			for $t \geq 0$, where $c_k, c_p, c_{p^\prime}, c_{k,p}$ are nonnegative constants depending on their subindices.
		\end{enumerate}
	\end{thm}
	
	\subsection{Preliminaries}
	In this section, we state a few important inequalities and notation which will be useful in our proofs.
	
	\subsubsection{Notation}
	$L^p(0,L)$, $1 \leq p < \infty$, is the usual Lebesgue space and given $u \in L^p(0,L)$, a Lebesgue measurable function, we will denote its $L^p-$norm by $\|u\|_p$, i.e.
	\begin{equation*}
		\|u\|_p \doteq \left(\int_0^L |u|^p dx\right)^\frac{1}{p}.
	\end{equation*}
	If $p = \infty$, then the corresponding norm is given by
	\begin{equation*}
		\|u\|_\infty \doteq \underset{x \in (0,L)}{\esssup} |u|.
	\end{equation*}
	Given $k > 0$, we denote the $L^2-$based Sobolev space by $H^k(0,L)$. In particular, $H_0^1(0,L)$ is the space of functions which belong to $H^1(0,L)$ that vanish at the endpoints in the sense of traces. If $A$ is a linear and bounded operator on $L^2(0,L)$, we will denote its operator norm on $L^2(0,L)$ by $\|A\|_{2 \to 2}$. We will write $a \lesssim b$ to denote an inequality $a \leq c b$ where $c > 0$ may only depend on fixed parameters of the problem under consideration which are not of interest. Sometimes, to prove the wellposedness of the models that we are interested in, we require compatibility conditions between initial and boundary data. We define the notion of compatibility in the following sense.
	\begin{defn}[Compatibility]
		Let $T, L > 0$.
		\begin{enumerate}
			\item [(i)] If $\phi \in H^{3}(0,L)$, $\psi \in H^{1}(0,T)$ are such that
			\begin{equation}
				\phi(0) = 0, \quad \phi(L) = 0, \quad \phi^\prime(L) = \psi(0)
			\end{equation} then we say $(\phi,\psi)$ satisfies compatibility conditions.
			\item[(ii)] If $\phi \in H^{6}(0,L)$, $\psi \in H^{2}(0,T)$ and $\tilde{\phi} \doteq \beta \phi^{\prime\prime\prime} - i\alpha \phi^{\prime\prime} + \delta \phi^\prime$, then we say $(\phi,\psi)$ satisfies higher order compatibility conditions provided it satisfies compatibility conditions and
			\begin{equation}
				\tilde\phi(0) = 0, \quad \tilde\phi(L) = 0, \quad \tilde\phi^\prime(L) = \psi^\prime(0).
			\end{equation}
		\end{enumerate}
		If $\psi \equiv 0$, then we only refer to $\phi$ regarding compatibility conditions.
	\end{defn}
	Finally, starting from Section \ref{estimates}, we drop the superscript notation $^*$ that expresses modified target models and modified backstepping kernels, and simply write $k, p, w, \hat{w}, \tilde{w}$, etc.
	
	\subsubsection{Some useful inequalities}
	We will use the Cauchy--Schwarz inequality, for $u,v \in L^2 (0,L)$
	\begin{equation*}
		\left|\int_0^L u(x) \overline{v(x)} dx \right| \leq \|u\|_2 \|v\|_2,
	\end{equation*}
	and $\epsilon-$Young's inequality
	\begin{equation*}
		|u v| \leq \frac{\epsilon}{p} |u|^p + \frac{1}{q \epsilon^{1/(p-1)}} |v|^q,
	\end{equation*}
	where $\epsilon > 0$, $1 < p < \infty$ and $\frac{1}{p} + \frac{1}{q} = 1$. We will use Gagliardo--Nirenberg's interpolation inequality in our estimates: Let $p \geq 2$, $\alpha = 1/2 - 1/p$ and $u \in H_0^1(0,L)$. Then,
	\begin{equation*}
		\|u\|_p \leq c \|u^\prime\|_2^\alpha \|u\|_2^{1 - \alpha}.
	\end{equation*}
	We will also use its higher order version: Let $u \in H^k(0,L) \cap H_0^1(0,L)$
	for $\alpha = j/k \leq 1$ where $j,k \in \mathbb{N}$. Then
	\begin{equation*}
		\|u^{(j)}\|_p \leq c \|u^{(k)}\|_2^\alpha \|u\|_2^{1 - \alpha}.
	\end{equation*}
	Special case of the Gronwall's inequality reads: given $f : [0,t] \to \mathbb{R}^+$ and $\alpha, \beta > 0$, the inequality
	\begin{equation*}
		f(t) \leq \alpha + \beta \int_0^t f(s)ds
	\end{equation*}
	implies
	\begin{equation*}
		f(t) \leq \alpha e^{\beta t}.
	\end{equation*}

	\subsection{Outline}
	In Section \ref{estimates}, we prove smoothing properties for a nonhomogeneous initial--boundary value problem with inhomogeneous boundary conditions. These will be useful for the wellposedness analysis that will be carried out in Section \ref{ctrl} and Section \ref{obs}. The tools we mainly use are semigroup theory, multipliers and the Laplace transform. In Section \ref{ctrl}, we first show the existence of an infinitely differentiable smooth backstepping kernel in Lemma \ref{lemkernel} and then state the invertibility of the backstepping transformation with a bounded inverse in Lemma \ref{inverselem}. Then, we study the wellposedness and exponential decay properties of modified target model \eqref{p_tar_lin}. Finally, thanks to the bounded invertibility of the backstepping transformation, we obtain the wellposedness and exponential decay for the original plant. Section \ref{obs} is devoted to the observer design problem where we assume that the state of the system is not known and only some partial boundary measurements are available. We first make the wellposedness analysis for modified target error model \eqref{err_tar_lin} and modified target observer model \eqref{tar_p_obs_lin}, respectively. Thanks to the bounded invertibility of the backstepping transformations, we show  the wellposedness for error and observer, which imply the wellposedness of the original plant. Next, we study the decay properties of the target error and target observer models. Again, by using the invertibility of backstepping transformations, we obtain the exponential stability of plant--observer--error system. In Section \ref{num_sim}, we introduce a numerical algorithm and then provide two numerical simulations for controller and observer designs. Finally, in appendices, we give details of some calculations.
	
	\section{Auxiliary lemmas} \label{estimates}
	In this section we prove some auxiliary lemmas which will be useful in order to show wellposesness results in Section \ref{ctrl_wp} and Section \ref{obs_wp}. Let us start by considering the following model
	\begin{eqnarray} \label{wp_pde}
	\begin{cases}
	iy_t + i\beta y_{xxx} +\alpha y_{xx} +i\delta y_x = f(x,t), x\in (0,L), t\in (0,T),\\
	y(0,t)=\psi_1(t), y(L,t)= \psi_2(t), y_x(L,t)= \psi_3 (t),\\
	y(x,0)= \phi(x).
	\end{cases}
	\end{eqnarray}
    We will denote the solution of \eqref{wp_pde} by $y[\phi,f,\psi_1,\psi_2,\psi_3]$.
	Let $A$ be the linear operator defined by
	\begin{equation} \label{A_opr}
	Ay \doteq -\beta y^{\prime\prime\prime} + i\alpha y^{\prime\prime} - \delta y^\prime
	\end{equation}
	with domain
	\begin{equation}\label{A_dom}
	D(A) = \{y \in H^3(0,L) : y(0) = y(L) = y^\prime(L) = 0 \}.
	\end{equation}
	It is shown in \cite{Ceballos05} that $A$ generates a strongly continuous semigroup of contractions on $L^2(0,L)$ denoted by $S(t)$, $t \geq 0$. Thus, by standard semigroup theory, \eqref{wp_pde} with $f=0$ and $\psi_i \equiv 0$, $i=1,2,3$, admits a unique mild solution (see \cite{Pazy}) for $\phi\in L^2(0,L)$. In the presence of nonhomogeneous boundary conditions $\psi_i$, $i=1,2,3$, but with zero forcing and zero initial state, i.e. $\phi \equiv f \equiv 0$, analysis of solutions of \eqref{wp_pde} will be carried out by obtaining a representation for the solution via the Laplace transform in $t$. Also, in the following lemmas below we obtain regularity estimates for solutions corresponding to initial, interior and boundary data in \eqref{wp_pde}.
	
	\begin{lem} \label{wp_lem_1}
		Let $f \equiv \psi_i \equiv 0$, $i=1,2,3$. Then, for $T > 0$ and $\phi \in L^2(0,L)$, $y(\cdot) = S(\cdot)\phi=y[\phi,0,0,0,0]$  satisfies space-time estimates
		\begin{enumerate}
			\item[(i)] $\|y\|_{C([0,T];L^2(0,L))}^2 + \beta \|y_x(0,\cdot)\|_{L^2(0,T)}^2 = \|\phi\|_2^2$,
			\item [(ii)] $\|y\|_{L^2(0,T;H^1(0,L))}^2 \lesssim (1+ T)\|\phi\|_2^2$
		\end{enumerate}
		and the time-space estimate
		\begin{enumerate}
			\item[(iii)] $\underset{x\in[0,L]}{\sup}\|y_x(x,\cdot)\|_{L^2(0,T)} \lesssim (1+\sqrt{T}) \|\phi\|_2$.
		\end{enumerate}
	\end{lem}
	
	\begin{proof}
		We first assume that $\phi \in D(A)$ and the solution is sufficiently smooth. The general case then can be shown by using the classical density argument.
		\begin{enumerate}
			\item[(i)] We take $L^2-$inner product of the main equation in \eqref{wp_pde} by $2y$ and get
			\begin{equation}
				\label{lin1iden}
				2\Im\int_0^L iy_t\bar{y}dx +  2\Im\int_0^Li\beta y_{xxx}\bar{y}dx +2\Im\int_0^L\alpha  y_{xx}\bar{y}dx
				+ 2\Im\int_0^Li\delta y_x\bar{y}dx = 0.
			\end{equation}
			The first term at the left hand side of \eqref{lin1iden} can be written as
			\begin{equation}
				\label{lin1iden2}
				2\Im\int_0^Liy_t\bar{y}dx=2\Re\int_0^Ly_t\bar{y}dx=\displaystyle\frac{d}{dt}\left|y(\cdot,t)\right|_2^2.
			\end{equation}
			The second term can be integrated by parts in $x$, and using boundary conditions we have
			\begin{equation}
				\label{lin1iden3}
				2\Im\int_0^Li\beta y_{xxx}\bar{y}dx=-2\Re\int_0^L\beta y_{xx}\bar{y}_xdx=\beta|y_x(0,t)|^2.
			\end{equation}
			The third term, again via integration by parts in $x$, gives
			\begin{equation}
				\label{lin1iden4}
				2\Im\int_0^L\alpha  y_{xx}\bar{y}dx =-2\Im\int_0^L\alpha  |y_{x}|^2dx=0.
			\end{equation}
			The fourth term vanishes since
			\begin{equation}
				\label{lin1iden5}
				2\Im\int_0^Li\delta u_x\bar{u}dx= \delta|u(x,t)|^2\bigg|_0^L=0.
			\end{equation}
			Combining \eqref{lin1iden2}-\eqref{lin1iden5} and integrating with respect to $t$, we arrive at
			\begin{equation} \label{wp_lem21_1}
			\|y(\cdot,t)\|_2^2 + \beta \|y_x(0,\cdot)\|_{L^2(0,T)}^2 = \|\phi\|_2^2.
			\end{equation}
			Passing to supremum on both sides over $[0,T]$ yields the desired result.
			
			\item[(ii)] Now we take $L^2-$inner product of the main equation \eqref{wp_pde} by $2xy$ and consider the imaginary parts of both sides to get
			\begin{equation}
			\label{wp_lem21_2}
			2\Im\int_0^L ixy_t\bar{y}dx +  2\Im\int_0^Li\beta xy_{xxx}\bar{y}dx +2\Im\int_0^L\alpha x y_{xx}\bar{y}dx \\
			+ 2\Im\int_0^Li\delta xy_x\bar{y}dx  = 0.
			\end{equation}
			The first term at the left hand side of \eqref{wp_lem21_2} can be written as
			\begin{equation}
			\label{wp_lem21_3}
			2\Im\int_0^Lixy_t\bar{y}dx= \frac{d}{dt}\int_0^L x |y|^2dx.
			\end{equation}
			The second term can be integrated by parts in $x$ and due to the boundary conditions we have
			\begin{align}
			\label{wp_lem21_4}
			2\Im\int_0^Li\beta xy_{xxx}\bar{y}dx &=2\beta\left(-\int_0^L y_{xx}\bar{u}dx - \int_0^L xy_{xx}\overline{y}_x dx\right) \nonumber\\
			&= 2\beta \left(\int_0^L|y_x|^2 dx - \frac{1}{2} \int_0^L x \frac{d}{dx} |y_x|^2 dx \right) \nonumber\\
			&= 2\beta \left(\|y_x(\cdot,t)\|_2^2 dx  + \frac{1}{2} \int_0^L |y_x|^2 dx \right) \nonumber\\
			&= 3\beta \|y_x(\cdot,t)\|_2^2.
			\end{align}
			Third and fourth terms, again via integration by parts in $x$ give us
			\begin{align}
			\label{wp_lem21_5}
			2\Im\int_0^L\alpha  xy_{xx}\bar{y}dx =2\alpha \left(-\Im\int_0^L y_x \overline{y} dx - \Im\int_0^L x |y_x|^2 dx\right)
			\end{align}
			and
			\begin{equation}
			\label{wp_lem21_6}
			2\Im\int_0^Li\delta x y_x\bar{y}dx= 2\delta \int_0^L x \frac{d}{dx} |y|^2 dx = -\delta\|y(\cdot,t)\|_2^2.
			\end{equation}
			Combining \eqref{wp_lem21_3}-\eqref{wp_lem21_6}, we get
			\begin{equation*}
			\frac{d}{dt} \int_0^L x |y|^2dx + 3\beta \|y_x(\cdot,t)\|_2^2 = \delta \|y(\cdot,t)\|_2^2 + 2\alpha \Im\int_0^L y_x \overline{y}dx
			\end{equation*}
			which, by $\epsilon-$Young's inequality applied to the second term at the right hand side, is equivalent to
			\begin{equation*}
			\frac{d}{dt} \int_0^L x |y|^2dx + (3\beta - \epsilon) \|y_x(\cdot,t)\|_2^2 \le c_{\alpha,\delta,\epsilon}\|y(\cdot,t)\|_2^2.
			\end{equation*}
			Integrating this result with respect to $t$ over $[0,T]$ yields
			\begin{equation*}
			\int_0^L x|y|^2dx + (3\beta - \epsilon) \int_0^T \|y_x(\cdot,t)\|_2^2 dt = \int_0^L x|\phi|^2dx + c_{\alpha,\delta,\epsilon}\int_0^T \|y(\cdot,t)\|_2^2 dt.
			\end{equation*}
			Combining this result with \eqref{wp_lem21_1}, using Poincare inequality and choosing $\epsilon > 0$ sufficiently small, we obtain (ii).
			
			\item[(iii)] Let us take an extension of $\phi$ in $L^2(\mathbb{R})$, denoted ${\phi^*}$, with the property that $\|\phi^*\|_{L^2(\mathbb{R})}\lesssim \|\phi\|_{L^2(0,L)}$. Consider the Cauchy problem
			\begin{eqnarray} \label{lem21_cauchy}
			\begin{cases}
			iv_t + i\beta v_{xxx} +\alpha v_{xx} +i\delta v_x = 0, x\in \mathbb{R}, t\in (0,T),\\
			v(x,0)= {\phi^*}(x).
			\end{cases}
			\end{eqnarray}
			Using the Fourier transform $\hat{\varphi}(\xi) = \int_{-\infty}^\infty e^{-i x \xi} \varphi(x)dx$, the solution of the above model can be represented as
			\begin{equation*}
				v(x,t) = \int_{-\infty}^{\infty} e^{i x \xi} e^{i\omega(\xi) t}\hat{{\phi^*}}(\xi) d\xi,
			\end{equation*}
			where $\omega(\xi) \doteq \beta\xi^3 - \alpha\xi^2 -\delta\xi$. We pick a smooth cut-off function
			\begin{equation*}
			\theta(\xi) =
			\begin{cases}
				1, & a \leq \xi \leq b, \\
                \text{smooth}, & a-\epsilon<\xi<a \text{ or } b<\xi<b+\epsilon\\
				0, & \xi \le a - \epsilon \text{ or } \xi \ge b + \epsilon,
			\end{cases}
			\end{equation*}
			where $\epsilon > 0$ is arbitrary, $|\theta|\le 1$, and  $a$ and $b$ will be chosen below in a suitable manner. Now, we decompose $v$ as
			\begin{equation} \label{wp_lem21_6.5}
				\begin{split}
				v(x,t) &=  \int_{-\infty}^{\infty} e^{i x \xi} e^{i\omega(\xi)t}\theta(\xi)\hat{{\phi^*}}(\xi) d\xi +  \int_{-\infty}^{\infty} e^{i x \xi} e^{i\omega(\xi)t}(1 - \theta(\xi))\hat{{\phi^*}}(\xi) d\xi \\
				&\doteq v_1(x,t) + v_2(x,t).
				\end{split}
			\end{equation}
			Using Cauchy-Schwarz inequality on $v_1$, Plancherel theorem and considering that $\theta$ is a compactly supported function, we get
			\begin{align*}
				|\partial_x v_1(x,t)| &= \left|\int_{-\infty}^{\infty} i\xi e^{i x \xi} e^{i\omega(\xi)t}\theta(\xi)\hat{{\phi^*}}(\xi) d\xi \right| \\
				&= \left(\int_{a-\epsilon}^{b+\epsilon} |i\xi|^2 |\theta(\xi)|^2 d\xi\right)^{\frac{1}{2}} \left(\int_{-\infty}^\infty |\hat{{\phi^*}}(\xi)|^2 d\xi\right)^{\frac{1}{2}} \\
				&\lesssim \|{\phi^*}\|_{L^2(\mathbb{R})}.
			\end{align*}
			Taking square of both sides and integrating over $[0,T]$ yields
			\begin{equation} \label{wp_lem21_7}
				\|\partial_xv_1(x,\cdot)\|_{L^2(0,T)}^2 \lesssim T \|{\phi^*}\|_{L^2(\mathbb{R})}^2.
			\end{equation}
By similar arguments, we also have
\begin{equation} \label{wp_lem21_7dir}
				\|\partial_t^jv_1(x,\cdot)\|_{L^2(0,T)}^2 \lesssim T \|{\phi^*}\|_{L^2(\mathbb{R})}^2
			\end{equation} for $j=0,1$, where the constant of the inequality depends on $j$ and $\theta$.
Interpolating, we get
\begin{equation} \label{wp_lem21_7int}
				\|v_1(x,\cdot)\|_{H^m(0,T)}^2 \lesssim T \|{\phi^*}\|_{L^2(\mathbb{R})}^2
			\end{equation} for any $m\in [0,1]$, where the constant of the inequality depends on $\theta$ and $m$.  In particular, for $m=1/3$, we have
\begin{equation} \label{wp_lem21_7int13}
				\|v_1(x,\cdot)\|_{H^{1/3}(0,T)} \lesssim \sqrt{T} \|{\phi^*}\|_{L^2(\mathbb{R})}.
			\end{equation} The last inequality is a smoothing property and will have a particular importance in our wellposedness analysis later.
			Next, consider the second term in \eqref{wp_lem21_6.5} and rewrite it as $$v_2 = \int_{-\infty}^a \cdot + \int_b^{\infty} \cdot \doteq v_{2-} + v_{2+}.$$ Consider the change of variable given by
			\begin{equation} \label{wp_lem21_8}
				\begin{split}
				\tau &= \omega_{-}(\xi) = \omega(\xi) : (-\infty,a] \to (-\infty,\omega(a)], \\
				\tau &= \omega_{+}(\xi) = \omega(\xi) : [b,\infty) \to [\omega(b),\infty),
				\end{split}
			\end{equation}
			where we define their inverses as
			\begin{equation} \label{wp_lem21_8.5}
			\begin{split}
				\xi &= \omega_{-}^{-1}(\tau) \doteq \xi_{-}(\tau), \\
				\xi &= \omega_{+}^{-1}(\tau) \doteq \xi_{+}(\tau)
			\end{split}
			\end{equation}
			for each integral respectively. Indeed a suitable choice of the support of the cut-off function ensures that the transformations given by \eqref{wp_lem21_8} are $1-1$, therefore their inverses exist.  Moreover, same choices given below will provide us that $\omega'(\xi)$ stays away from zero in \eqref{wp_lem21_9}. Depending on the sign of $\alpha^2 + 3\beta \delta$, we have three different cases:
			\begin{enumerate}
				\item [(a)] Let $\alpha^2 + 3\beta \delta > 0$. Then, any choice $a < \frac{\alpha - \sqrt{\alpha^2 + 3\beta\delta}}{3\beta}$ and $b > \frac{\alpha + \sqrt{\alpha^2 + 3\beta\delta}}{3\beta}$ provides that the mapping is $1-1$ and $\omega'(\xi)$ stays away from zero on $(-\infty,a] \cup [b,\infty)$.
				
				\item[(b)] Let $\alpha^2 + 3\beta\delta = 0$. Then the mapping is $1-1$ for all choices of $a < b$. On the other hand, choosing $a < \frac{\alpha}{3\beta} < b$ will provide that  $\omega'(\xi)$ stays away from zero on $(-\infty,a] \cup [b,\infty)$.
				
				\item[(c)] Let $\alpha^2 + 3\beta\delta < 0$. Then the mapping is $1-1$ and  $\omega'(\xi)$ stays away from zero on $(-\infty,a]  \cup [b,\infty) $ for all choices of $a < b$ .
			\end{enumerate}
			Assume that we choose appropriate $a$ and $b$ values for each case of $\alpha^2 + 3\beta\delta$ described above. Following from \eqref{wp_lem21_8}-\eqref{wp_lem21_8.5}, we have
			\begin{equation} \label{wp_lem21_8.6}
				d\xi = \frac{1}{3\beta \xi_{\mp}^2(\tau) - 2\alpha\xi_{\mp}(\tau) - \delta} d\tau.
			\end{equation}
			Hence $v_2$ becomes
			\begin{equation*}
			\begin{split}
			v_2(x,t) =& \int_{-\infty}^{\omega(a)} e^{i\xi_-(\tau) x} e^{i\tau t} (1 - \theta(\xi_-(\tau)))  \frac{\hat{{\phi^*}}(\xi_-(\tau))}{3\beta \xi_{-}^2(\tau) - 2\alpha\xi_{-}(\tau) - \delta} d\tau \\
			&+ \int_{\omega(b)}^{\infty}  e^{i\xi_+(\tau) x} e^{i\tau t} (1 - \theta(\xi_+(\tau)))  \frac{\hat{{\phi^*}}(\xi_+(\tau))}{3\beta \xi_{+}^2(\tau) - 2\alpha\xi_{+}(\tau) - \delta} d\tau.
			\end{split}
			\end{equation*}
			Let us first consider $v_{2-}$ and observe that the function
			\begin{equation*}
				\begin{cases}
					e^{i\xi_-(\tau) x} (1 - \theta(\xi_-(\tau)))  \frac{\hat{{\phi^*}}(\xi_-(\tau))}{3\beta \xi_{-}^2(\tau) - 2\alpha\xi_{-}(\tau) - \delta}, & \tau \in (-\infty,\omega(a)], \\
					0, & \text{elsewhere,}
				\end{cases}
			\end{equation*}
			is the Fourier transform of $v_{2-}$ with respect to the second component. Thus,
			\begin{align*}
				\|v_{2-}(x,\cdot)\|_{H^{\frac{1}{3}}(0,T)}^2 &\leq \|v_{2-}(x,\cdot)\|_{H_t^{\frac{1}{3}}(\mathbb{R})}^2 \\
				&= \int_{-\infty}^\infty(1+\tau^2)^{\frac{1}{3}} |\hat{v}_{2-}(x,\tau)|^2 d\tau \\
				&= \int_{-\infty}^{\omega(a)}(1+\tau^2)^{\frac{1}{3}}\left| e^{i\xi_-(\tau) x}(1 - \theta(\xi_-(\tau))) \frac{\hat{{\phi^*}}(\xi_-(\tau))}{3\beta \xi_{-}^2(\tau) - 2\alpha\xi_{-}(\tau) - \delta}\right|^2 d\tau \\
				&\le \int_{-\infty}^{\omega(a)} (1+\tau^2)^{\frac{1}{3}}|\frac{|\hat{{\phi^*}}(\xi_-(\tau))|^2}{|3\beta \xi_{-}^2(\tau) - 2\alpha\xi_{-}(\tau) - \delta|^2} d\tau.
			\end{align*}
			Changing variables back as $\tau = \omega_-(\xi)$, it follows from the above estimate that
			\begin{align}
				\label{wp_lem21_9}
				\|v_{2-}(x,\cdot)\|_{H^{\frac{1}{3}}(0,T)}^2 &\leq \int_{-\infty}^a (1+\omega^2(\xi))^{\frac{1}{3}} \frac{|\hat{{\phi^*}}(\xi)|^2}{|3\beta \xi^2 - 2\alpha\xi - \delta|^2} (3\beta \xi^2 - 2\alpha \xi - \delta) d\xi \\
            &\lesssim \int_{-\infty}^a (1+\xi^6)^{\frac{1}{3}} \frac{|\hat{{\phi^*}}(\xi)|^2}{3\beta \xi^2 - 2\alpha\xi - \delta} d\xi\simeq \int_{-\infty}^a |\hat{{\phi^*}}(\xi)|^2 d\xi \nonumber \\
				&\leq \|{\phi^*}\|_{L^2(\mathbb{R})}^2. \nonumber
			\end{align}
			$\|v_{2+}(x,\cdot)\|_{H^{\frac{1}{3}}(0,T)}^2\lesssim \|{\phi^*}\|_{L^2(\mathbb{R})}^2$ can be shown similarly. Hence $$\|v_{2}(x,\cdot)\|_{H^{\frac{1}{3}}(0,T)} \lesssim \|{\phi^*}\|_{L^2(\mathbb{R})}^2.$$ Combining this with \eqref{wp_lem21_7}, we get
			\begin{equation*}
				\|v(x,\cdot)\|_{H^{\frac{1}{3}}(0,T)} \lesssim (1+\sqrt{T}) \|{\phi^*}\|_{L^2(\mathbb{R})}.
			\end{equation*}
Differentiating in $x$ and repeating the above arguments it also follows that
\begin{equation*}
				\|\partial_x v(x,\cdot)\|_{L^{2}(0,T)} \lesssim (1+\sqrt{T}) \|{\phi^*}\|_{L^2(\mathbb{R})}.
			\end{equation*}
			We also have the continuity of the mappings $x \to \|v(x,\cdot)\|_{H^{1/3}(0,T)}$ and $x \to \|\partial_xv(x,\cdot)\|_{L^{2}(0,T)}$. To this end, one needs to show that, given $\{x_n\} \subset \mathbb{R}$ converging to $x \in \mathbb{R}$
\begin{equation*}
				\|v(x,\cdot) -v(x_n,\cdot)\|_{H^{1/3}(0,T)} \to 0, \quad \text{as } n \to \infty
			\end{equation*} and
			\begin{equation*}
				\|\partial_xv(x,\cdot) - \partial_xv(x_n,\cdot)\|_{L^2(0,T)} \to 0, \quad \text{as } n \to \infty
			\end{equation*}
			hold. These can be easily shown by using the dominated convergence theorem.	
    Now, we can represent the solution of \eqref{wp_pde} as $$y[\phi,0,0,0,0]=v|_{(0,L)}-y[0,0,v(0,\cdot),v(L,\cdot),v_x(L,\cdot)],$$ where $y[0,0,v(0,\cdot),v(L,\cdot),v_x(L,\cdot)]$ is the solution of \eqref{wp_pde} with  $f\equiv \phi\equiv 0,$ $$y(0,t)\doteq v(0,t)\in H^{1/3}(0,T), y(L,t)\doteq v(L,t)\in H^{1/3}(0,T),$$ $$y_x(L,t)\doteq v_x(L,t)\in L^{2}(0,T).$$
    Hence, part (iii) follows by combining the boundary smoothing property of $v$ and the inhomogeneous boundary value problem given in Lemma \ref{wp_lem_3} below.
	\end{enumerate}
	\end{proof}

	\begin{lem} \label{wp_lem_2}
		Let $\phi \equiv \psi_i \equiv 0$, $i=1,2,3$, $T > 0$, and $f \in L^1(0,T;L^2(0,L))$. Then the solution $y=y[0,f,0,0,0]$ of \eqref{wp_pde} satisfies space-time estimates
		\begin{enumerate}
			\item [(i)] $\|y\|_{C([0,T];L^2(0,L))}^2 + \beta \|y_x(0,\cdot)\|_{L^2(0,T)}^2 \lesssim \|f \|_{L^1(0,T;L^2(0,L))}^2$,
			\item [(ii)] $\|y\|_{L^2(0,T;H^1(0,L))}^2 \lesssim (1 + T)\|f\|_{L^1(0,T;L^2(0,L))}^2$
		\end{enumerate}
		and the time-space estimate
		\begin{enumerate}
			\item [(iii)] $\underset{x\in[0,L]}{\sup}\|y_x(x,\cdot)\|_{L^2(0,T)} \lesssim (1+\sqrt{T})\|f\|_{L^1(0,T;L^2(0,L))}$.
		\end{enumerate}
	\end{lem}
	
	\begin{proof} \hfill
		\begin{enumerate}
			\item [(i)] Multiplying the main equation by $2\bar{y}$, integrating over $[0,T]\times[0,L]$ and using \eqref{lin1iden2}-\eqref{lin1iden5}, we get
			\begin{equation}\label{220}
				\|y(\cdot,t)\|_2^2 + \beta \|y_x(0,\cdot)\|_{L^2(0,T)}^2 \leq 2\int_0^T\int_0^L |f(x,t)| |y(x,t)| dx dt.
			\end{equation}						
			We apply Cauchy--Schwarz and $\epsilon-$Young's inequalities to the right hand side of \eqref{220} to obtain
			\begin{equation*}
				\|y(\cdot,t)\|_2^2 + \beta \|y_x(0,\cdot)\|_{L^2(0,T)}^2 \leq \epsilon \underset{t\in[0,T]}{\sup}\|y(\cdot,t)\|_2^2 + c_\epsilon\|f\|_{L^1(0,T;L^2(0,L))}^2.
			\end{equation*}
			Right hand side is independent of $t$. So passing to supremum on both sides over $[0,T]$ and choosing $\epsilon > 0$ sufficiently small yield (i).
			
			\item[(ii)] Multiplying the main equation by $2xy$, integrating over $[0,T]\times[0,L]$ and using the same arguments in \eqref{wp_lem21_3}-\eqref{wp_lem21_6}, we get
			\begin{multline}\label{wp_lem22_3}
			\int_0^L x |y(x,t)|^2 dx + 3\beta \int_0^T \|y_x(\cdot,t)\|_2^2 dt \\= \delta \int_0^T \|y(\cdot,t)\|_2^2 dt + 2\alpha \Im\int_0^T\int_0^L y_x(x,t) \overline{y}(x,t)dx dt + 2\Im\int_0^T\int_0^L x f(x,t) \overline{y}(x,t) dx dt.
			\end{multline}
			Second term at the right hand side of \eqref{wp_lem22_3} can be estimated via $\epsilon-$Young's inequality as
			\begin{equation} \label{wp_lem22_4}
				2\alpha \Im\int_0^T\int_0^L y_x(x,t) \overline{y}(x,t)dx dt \leq \epsilon\int_0^T\|y_x(\cdot,t)\|_2^2dt + c_{\alpha,\epsilon} \int_0^T\|y(\cdot,t)\|_2^2 dt.
			\end{equation}
			Using Cauchy--Schwarz inequality and $\epsilon-$Young's inequality, and thanks to (i), the third term at the right hand side of \eqref{wp_lem22_3} can be estimated as	
			\begin{equation} \label{wp_lem22_5}
				\begin{split}
					2\Im\int_0^T\int_0^L x f(x,t) \overline{y}(x,t) dx dt &\leq \underset{t\in[0,T]}{\sup}\|y(\cdot,t)\|_2^2 + L^2\|f\|_{L^1(0,T;L^2(0,L))}^2 \\
				&\leq c_L \|f\|_{L^1(0,T;L^2(0,L))}^2.
				\end{split}
			\end{equation}
			Using \eqref{wp_lem22_4}-\eqref{wp_lem22_5}, it follows from \eqref{wp_lem22_3} that
			\begin{equation}
				\begin{split}
					&\int_0^L x |y(x,t)|^2 dx + (3\beta-\epsilon) \int_0^T \|y_x(\cdot,t)\|_2^2 dt \\ \leq& c_{\alpha,\delta,\epsilon} \int_0^T\|y(\cdot,t)\|_2^2 +  c_L \|f\|_{L^1(0,T;L^2(0,L))}^2 \\
					\leq& c_{\alpha,\delta,\epsilon} T \underset{t\in[0,T]}{\sup}\|y(\cdot,t)\|_2^2 + c_L \|f\|_{L^1(0,T;L^2(0,L))}^2 \\
					\leq& c_{L,\alpha,\delta,\epsilon} (1 + T) \|f\|_{L^1(0,T;L^2(0,L))}^2,
				\end{split}
			\end{equation}
			where we used (i) in the last line again. Finally, choosing $\epsilon\in (0, 3\beta)$, using the Poincare inequality and dropping the first term at the left hand side, we conclude with (ii).
			
			\item [(iii)] Using Duhamel's principle, the solution is of the form
			\begin{equation*}
				y(x,t) = \int_0^t S(t - \tau) f(x,\tau)d\tau.
			\end{equation*}
			By differentiating with respect to $x$
			\begin{equation*}
				\begin{split}
					|\partial_xy(x,t)| &= \left|\partial_x \left[ \int _0^t S(t - \tau) f(x,\tau) d\tau \right] \right| \\
					&\leq \int_0^t \left|\partial_x \left[  S(t - \tau) f(x,\tau) \right] \right| d\tau \\
					&\leq \int_0^T \left|\partial_x \left[  S(t - \tau) f(x,\tau) \right] \right| d\tau.
				\end{split}
			\end{equation*}
			Taking $L^2-$ norm of both sides with respect to $t$ on $[0,T]$ and using the result in Lemma \ref{wp_lem_1}-(iii)
			\begin{equation*}
				\begin{split}
					\|\partial_x y(x,\cdot)\|_{L^2(0,T)} &\leq \left\| \int_0^T \left|\partial_x \left[  S(\cdot - \tau) f(x,\tau) \right] \right| d\tau \right\|_{L^2(0,T)} \\
					&\leq \int_0^T \left\| \partial_x \left[  S(\cdot - \tau) f(x,\tau) \right] \right\|_{L^2(0,T)} d\tau \\
					&\lesssim \int_0^T (1+\sqrt{T}) \|f(\cdot,\tau)\|_{L^2(\mathbb{R})} d\tau \\
					&\lesssim (1+\sqrt{T})\|f\|_{L^1(0,T;L^2(0,L))}.
				\end{split}
			\end{equation*}
			Passing to supremum in $x$ over $[0,L]$ ends the proof of (iii).
 		\end{enumerate}	
	\end{proof}
		
	Now let us turn our attention to the nonhomogeneous boundary value problem with $\phi \equiv f \equiv 0$ and let us first obtain an explicit representation for $y[0,0,\psi_1,\psi_2,\psi_3]$ in terms of the boundary data $\psi_m$, $m = 1, 2, 3$, where we consider, $\psi_m^*$, of $\psi_m$'s from $(0,T)$ to $\mathbb{R}$ satisfying $\|\psi_j^*\|_{H^{1/3}(\mathbb{R})} \lesssim \|\psi\|_{H^{1/3}(0,T)}$, $j = 1,2$ and $\|\psi_3^*\|_{L^2(\mathbb{R})} \lesssim \|\psi\|_{L^2(0,T)}$. We can further assume that $\psi_m^*(t)=0$ for $t<0$. For simplicity, we denote the extended functions again by $\psi_m$. Our approach for obtaining a representation for the solution is to apply the Laplace transformation in time:
	\begin{equation*}
		\tilde{f}(s) = \int_0^\infty e^{-st} f(t)dt.
	\end{equation*}
	This approach is motivated from \cite{Bona03} on the KdV equation. However, due to the parameters $\beta, \alpha, \delta$ and assuming that $L$ may be critical, the situation gets more complicated and the treatment of the problem is more subtle.
	
	To this end, we apply the Laplace transformation and transform \eqref{wp_pde} to the following infinite family of third--order boundary value problems
	\begin{eqnarray} \label{lap_trans}
		\begin{cases}
			is\tilde{y}(x,s) + i\beta \tilde{y}_{xxx} (x,s) + \alpha \tilde{y}_{xx}(x,s) + i\delta \tilde{y}_x(x,s) = 0, (x,s) \in (0,L) \times \mathbb{C}, \\
			\tilde{y}(0,s) = \tilde{\psi_1}(s),  \tilde{y}(L,s) = \tilde{\psi_2}(s), \tilde{y}_x(L,s) = \tilde{\psi_3}(s),
		\end{cases}
	\end{eqnarray}
	where a suitable set for the complex valued independent variable $s$ is specified below. 
	Using the Bromwich integral, $y$ can be represented as
	\begin{equation} \label{inv_lap}
		y(x,t) = \frac{1}{2\pi i} \int_{r - i\infty}^{r + i \infty} e^{st} \tilde y (x,s) ds,
	\end{equation}
	where the vertical integration path $(r-i\infty,r+i\infty)$ in the complex plane is chosen so that, all possible singularities of $\tilde y$ lie at the left of it. Note that for sufficiently large $r$ the characteristic equation,
	\begin{equation} \label{char_eqn}
		s+ \beta \lambda^3  - i\alpha \lambda^2 + \delta \lambda = 0,
	\end{equation}
	for \eqref{lap_trans} has distinct roots. In fact, there exists only finitely many $s$ for which \eqref{char_eqn} has double or possibly triple roots. We can classify these cases depending on the sign of the quantity $\alpha^2 + 3\beta \delta$. To this end, let $\lambda_j$, $j = 1, 2, 3$, denote the roots of \eqref{char_eqn} and assume that $\lambda_2 = \lambda_3$ for some $s \in \mathbb{C}$. Then direct calculations (see Appendix \ref{app_charroots} for details) yield the following cases.
	\begin{enumerate}
		\item [(i)] If $\alpha^2 + 3\beta \delta > 0$, then there exists only two possible values of $s$ and these values belong to the imaginary axis.
		
		\item[(ii)] If $\alpha^2 + 3\beta \delta = 0$, then we have one and only one possible value of $s$ and this value belongs to the imaginary axis. Note that for this value of $s$, we have $\lambda_1 = \lambda_2 = \lambda_3$.
		
		\item[(iii)] If $\alpha^2 + 3\beta \delta < 0$, then there exists only two possible values of $s$ and these values are symmetric with respect to the imaginary axis.
	\end{enumerate}
	Consequently, for sufficiently large $r$ \eqref{char_eqn} has distinct roots on the line $\Re(s)=r$ and solution of \eqref{lap_trans} is of the form
	\begin{equation} \label{soln_lap}
		\tilde{y}(x,s) = \sum_{j=1}^3 c_j(s)e^{\lambda_j(s)x},
	\end{equation}
	where the column vector $(c_1(s), c_2(s), c_3(s))^T$ is the solution of the linear system
	\begin{equation} \label{cj_system}
		\begin{pmatrix}
		1 & 1  & 1 \\
		e^{\lambda_1(s)L} & e^{\lambda_2(s)L}  & e^{\lambda_3(s)L} \\
		\lambda_1(s) e^{\lambda_1(s)L}   & \lambda_2(s) e^{\lambda_2(s)L} & \lambda_3(s) e^{\lambda_3(s)L}
		\end{pmatrix}
		\begin{pmatrix}
		c_1(s) \\
		c_2(s) \\
		c_3(s)
		\end{pmatrix}
		=
		\begin{pmatrix}
		\tilde{\psi_1}(s) \\
		\tilde{\psi_2}(s) \\
		\tilde{\psi_3}(s)
		\end{pmatrix}.
	\end{equation}
	Applying Cramer's rule, these coefficients can be obtained as $c_j = \frac{\Delta_j}{\Delta}$, where $\Delta$ is the determinant of the coefficient matrix and $\Delta_j$'s are determinants of the matrices formed by replacing the $j-$th column of the coefficient matrix by the column vector $(\tilde{\psi_1},\tilde{\psi_2},\tilde{\psi_3})^T$. Thus $y$ is of the form
	\begin{equation} \label{inv_lap1.5}
		y(x,t) = \frac{1}{2\pi i}  \sum_{j=1}^3 \int_{r - i\infty}^{r + i \infty} e^{st} \frac{\Delta_j(s)}{\Delta(s)} e^{\lambda_j(s)x } ds.
	\end{equation}
	We can rewrite $y[0,0,\psi_1,\psi_2,\psi_3]$ as $y \equiv \sum_{m=1}^3 y_m$, where $y_m$ solves the same problem with boundary data $\psi_j \equiv 0$ if $j \neq m$, $j = 1, 2, 3$. Thus $y_m$'s can be expressed as
	\begin{equation} \label{inv_lap2}
		y_m(x,t) = \frac{1}{2\pi i}  \sum_{j=1}^3 \int_{r - i\infty}^{r + i \infty} e^{st} \frac{\Delta_{j,m}(s)}{\Delta(s)} e^{\lambda_j(s)x} \tilde{\psi_m}(s) ds, \quad m = 1, 2, 3.
	\end{equation}
	Here $\Delta_{j,m}$'s, $m = 1,2,3$, are obtained from $\Delta_j$, where $\psi_m$ is replaced by $1$ and $\psi_j$'s, $j \neq m$, are replaced by $0$ for each $j$.
	
	To change the integration path in \eqref{inv_lap2} by a more convenient one, one needs to investigate possible zeros of $\Delta(s)$ in the complex plane. These points occur not only due to the double or possibly triple roots of \eqref{char_eqn} but may also occur due to the eigenvalues of the operator $A$ defined in \eqref{A_opr} with domain $D(A)$ defined in \eqref{A_dom}. Note that $A$ is a dissipative operator:
	\begin{equation*}
		\Re ( A\varphi, \varphi ) = \Re \int_0^L (-\beta \varphi^{\prime\prime\prime} + i\alpha \varphi^{\prime\prime} - \delta \varphi^\prime)(x) \overline{\varphi} (x) dx = - \frac{\beta |\varphi^\prime(0)|^2}{2} \leq 0.
	\end{equation*}
	Thus, in particular, all eigenvalues of $A$ lie on the left complex half  plane or possibly on the imaginary axis. The latter situation occurs only if the problem
	\begin{equation} \label{eigenprob}
		\begin{cases}
		-\beta \varphi^{\prime\prime\prime} + i\alpha \varphi^{\prime\prime} - \delta \varphi^\prime = \lambda \varphi, \quad \text{in }(0,L), \\
		\varphi(0) = \varphi(L) = \varphi^\prime(0) = \varphi^\prime(L),
		\end{cases}
	\end{equation}
	has nontrivial solutions. Using the corresponding characteristic equation $$- \beta m^3 + i\alpha m^2 - \delta m = \lambda$$ for the main equation of \eqref{eigenprob} together with the boundary conditions, one can obtain that the roots $m_j$, $j = 1,2,3$, must be distinct, i.e. $\varphi(x) = \sum_{j = 1}^3 d_1 e^{m_j x}$. Together with the boundary conditions, this implies that $m_j$'s must satisfy
	\begin{equation*}
		e^{m_1 L} = e^{m_2 L} = e^{m_3 L}
	\end{equation*}		
	(see \cite[Proposition 2]{Glass10} for similar calculations). Therefore, we have
	\begin{equation*}
		\begin{split}
		m_2 - m_1 &= \frac{2k\pi i}{L}, \\
		m_3 - m_2 &= \frac{2l\pi i}{L},
		\end{split}
	\end{equation*}
	where without loss of generality, upon relabeling $m_j$'s we can assume that $k, l \in \mathbb{Z}^+$. Using $m_1 + m_2 + m_3 = \frac{i\alpha}{\beta}$, we get
	\begin{equation*}
		\begin{split}
			m_1 &= \frac{i \alpha}{3\beta} + \frac{2(-2k - l) \pi i}{3L}, \\
			m_2 &= \frac{i \alpha}{3\beta} + \frac{2(k - l) \pi i}{3L}, \\
			m_3 &= \frac{i \alpha}{3\beta} + \frac{2(k + 2l) \pi i}{3L}. \\
		\end{split}
	\end{equation*}
	Substituting these into $m_1 m_2 + m_1 m_3 + m_2 m_3 = \frac{\delta}{\beta}$, after some calculations, we obtain
	\begin{equation*}
		\frac{\delta}{\beta} = -\frac{\alpha^2}{3\beta^2} + \frac{4\pi^2(k^2 + kl + l^2)}{3L^2}
	\end{equation*}
	or equivalently
	\begin{equation} \label{critic_rel}
 		\alpha^2 + 3\beta\delta = \frac{4\pi^2 \beta^2 (k^2 + kl + l^2)}{L^2}.
	\end{equation}
	Consequently, depending on the sign of $\alpha^2 + 3\beta\delta$ and the interval length $L$, it is possible to obtain a nontrivial solution of \eqref{eigenprob}, therefore there can be some eigenvalues on the imaginary axis.
	
	\begin{enumerate}
	\item [(i)] Let $\alpha^2 + 3\beta\delta > 0$. Then, choosing $L$ from the set of critical lengths given in \eqref{critic_len} imply the existence of some eigenvalues that are located on the imaginary axis. Now from the equation $m_1 m_2 m_3 = - \frac{\lambda}{\beta}$ and using $L = 2\pi \beta \sqrt{\frac{k^2 + kl + l^2}{\alpha^2 + 3\beta\delta}}$, one obtains after some calculations that
	\begin{equation} \label{eigval}
		\lambda = \frac{i}{27 \beta^2} \left[\alpha^3 - 3\alpha (\alpha^2 + 3\beta\delta) + 2 (\alpha^2 + 3\beta\delta)^{3/2} H(k,l)\right]
	\end{equation}
	where
	\begin{equation}
		H(k,l) = \frac{(-2k-l)(k-l)(k+2l)}{2(k^2 + kl + l^2)^{\frac{3}{2}}}.
	\end{equation}
	It is not difficult to see that $-1 < H(k,l) < 1$ for $k, l \in \mathbb{Z}^+$. Thus, from \eqref{eigval}, we deduce that $\Im(\lambda) \in (\Im s_2^+, \Im s_1^+)$, where
	\begin{equation*}
		s_1^+ \doteq \frac{i}{27 \beta^2} \left[\alpha^3 - 3\alpha (\alpha^2 + 3\beta\delta) + 2 (\alpha^2 + 3\beta\delta)^{3/2}\right]
	\end{equation*}
	and
	\begin{equation*}
		s_2^+ \doteq \frac{i}{27 \beta^2} \left[\alpha^3 - 3\alpha (\alpha^2 + 3\beta\delta) - 2 (\alpha^2 + 3\beta\delta)^{3/2}\right].
	\end{equation*}
	See Figure \ref{fig:hkl} for the graph of $H(k,l)$.
	\begin{figure}[h]
		\centering
		\includegraphics[width=8cm]{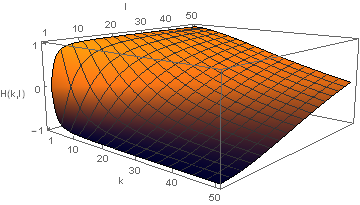}
		\caption{$H(k,l) : [1,\infty) \times [1,\infty) \to (-1,1)$}
		\label{fig:hkl}
	\end{figure}

	On the other hand $s_1^+$ and $s_2^+$ are the points where \eqref{lap_trans} assumes double root (see Appendix \ref{app_charroots} for detailed calculations), hence zeros of $\Delta(s)$. This fact together with the location of the possible pure imaginary eigenvalues imply that all possible singular points of \eqref{inv_lap1.5} belonging to the imaginary axis lie in the closed interval $[\Im(s_2^+),\Im(s_1^+)]$. Thus, we can deform the vertical integration path of \eqref{inv_lap1.5} by shifting the parts $\{s:\Im(s) > \Im(s^+_1) + \rho\}$ and $\{s:\Im(s) < \Im(s^+_2) - \rho\}$ to $C_1^+ \doteq (s^+_1 + i\rho,i\infty)$ and $C_3^+ \doteq (-i\infty,s^+_2 - i\rho)$  respectively, $\rho > 0$ is fixed, whereas we shift the rest of the integration path up to $\rho$ units from the imaginary axis, and avoid the points $s_1^+$, $s_2^+$ by a quarter-circular arcs to the upper-right and lower-right respectively, denoted by $C_2^+$. See Figure \ref{fig:int_path+} below.
	
	\begin{figure}[H]
		\begin{tikzpicture}[decoration={markings,
				mark=at position 1cm   with {\arrow[line width=1.5pt]{stealth}},
				mark=at position 4.5cm with {\arrow[line width=1.5pt]{stealth}},
				mark=at position 7cm   with {\arrow[line width=1.5pt]{stealth}},
				mark=at position 9.5cm with {\arrow[line width=1.5pt]{stealth}}
			}]
			
			\draw [<->] (-3,0) -- (0,0) -- (3,0) node[right]{$\Re(s)$};
			\draw [<->] (0,-1.2) -- (0,0) -- (0,4.7) node[right]{$\Im(s)$};
			
			\draw[dashed]
			(0.5,0.9) -- (0,0.9) node[left] {$s_2^+$};
			\draw[dashed]
			(0.5,2.5) -- (0,2.5) node[left] {$s_1^+$};
			
			\draw (0,1.3) node[cross=2pt,rotate=0] {};
			\draw (0,1.7) node[cross=2pt,rotate=0] {};
			\draw (0,2.1) node[cross=2pt,rotate=0] {};

			\filldraw
			(0,0.9) circle (1.25pt) node[align=center, above] {};
			\filldraw
			(0,2.5) circle (1.25pt) node[align=center, below] {};
			
			
			\tikzset{->-/.style={decoration={
						markings,
						mark=at position .5 with {\arrow{>}}},postaction={decorate}}}
			
			\draw [->-, blue, very thick] (0,-0.8) -- (0,-0.3) node[right] {$C_3^+$} --  (0,0.4) {};
			\draw [->-, blue, very thick] (0,0.4) arc[start angle=-90, end angle=0, radius=5mm] {};
			\draw [->-, blue, very thick] (0.5,0.9) -- (0.5,1.7) node[right] {$C_2^+$} -- (0.5,2.5){};
			\draw [->-, blue, very thick] (0.5,2.5) arc[start angle=0, end angle=90, radius=5mm] {};
			\draw [->-, blue, very thick] (0,3) -- (0,3.6) node[right] {$C_1^+$} --  (0,4.2) {};			
		\end{tikzpicture}
		\caption{Integration path for the case $\alpha^2 + 3\beta\delta > 0$.} \label{fig:int_path+}
	\end{figure}
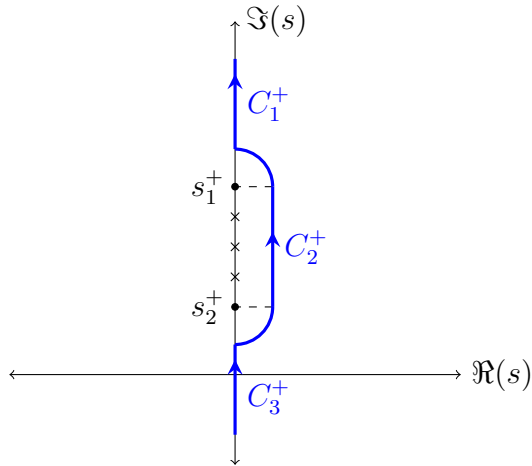
	Thus we can express \eqref{inv_lap2} as
	\begin{equation} \label{invlap>0}
		\begin{split}
			y_m(x,t) =& \frac{1}{2\pi i} \sum_{j=1}^3  \int_{C_1^+} e^{st} \frac{\Delta_{j,m}(s)}{\Delta(s)} e^{\lambda_j(s)x} \tilde{\psi_m}(s) ds
			\\ &+\frac{1}{2\pi i}\sum_{j=1}^3 \int_{C_2^+} e^{st} \frac{\Delta_{j,m}(s)}{\Delta(s)} e^{\lambda_j(s)x} \tilde{\psi_m}(s) ds \\
			&+ \frac{1}{2\pi i} \sum_{j=1}^3  \int_{C_3^+} e^{st} \frac{\Delta_{j,m}(s)}{\Delta(s)} e^{\lambda_j(s)x} \tilde{\psi_m}(s) ds.
		\end{split}
	\end{equation}
	Now we change the variable in the first and the third integral as $s = i\omega(\xi) = i(\beta \xi^3 - \alpha \xi^2 - \delta \xi)$. For $\alpha^2 + 3\beta\delta > 0$, the function $\omega(\xi)$ has one local maximum and one local minimum. After some calculations one can find that the most right inverse image of $s_1^+$ and the most left inverse image of $s_2^+$ are given by
	\begin{equation*}
		\xi_1^+  \doteq \frac{\alpha +2 \sqrt{\alpha^2 + 3\beta\delta}}{3\beta}, \quad 	\xi_2^+  \doteq \frac{\alpha -2 \sqrt{\alpha^2 + 3\beta\delta}}{3\beta}
	\end{equation*}
	respectively (see Figure \ref{fig:int_trans_>0}). Thus inverse images of the paths $C_1^+$ and $C_3^+$ under the transformation $s = i \omega(\xi)$ become $(\xi_1^+ + \eta_1^+, \infty)$ and $(-\infty,\xi_2^+ - \eta_2^+)$ respectively for $\eta_1^+, \eta_2^+ > 0$.
	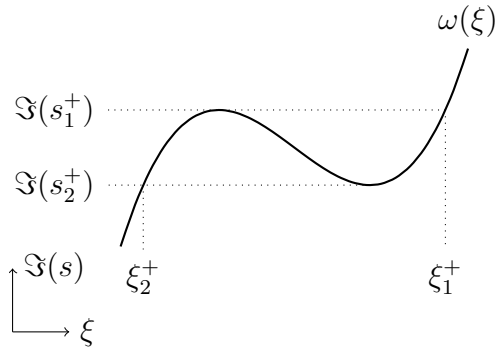
\begin{figure}[h]
		\begin{tikzpicture}
			\begin{axis}[
			axis x line=none,
			axis y line=none,
			ticks=none,
			xmin=-2.85,
			xmax=1.7,
			ymin=-3,
			ymax=3.75]
			\addplot [domain=-28/15:1.2, smooth, thick] { x^3 + x^2 - x } node[above,pos=1] {$ \omega(\xi)$};
			\addplot [dotted] coordinates {(-5/3,-5/27) (-5/3,-1.2)}
			node[below,pos=1] {$\xi_{2}^+$};
			
			
			
			\addplot [dotted] coordinates {(1,1) (1,-1.2)}
			node[below,pos=1] {$\xi_1^+$};
			
			\addplot [dotted] coordinates {(1/3,-5/27) (-2,-5/27)}
			node[left,pos=1] {$\Im(s_2^+)$};
			
			\addplot [dotted] coordinates {(1,1) (-2,1) }
			node[left,pos=1] {$\Im(s_1^+)$};
			
			\addplot [->] coordinates {(-2.82,-2.5) (-2.32,-2.5) }
			node[right,pos=1] {$\xi$};
			
			\addplot [->] coordinates {(-2.82,-2.5) (-2.82,-1.5) }
			node[right,pos=1] {$\Im(s)$};
			\end{axis}
		\end{tikzpicture}
		\caption{Plot of transformation $\Im(s) = \omega(\xi)$ when $\alpha^2 + 3\beta\delta > 0$.} \label{fig:int_trans_>0}
	\end{figure}
	Consequently,  \eqref{invlap>0} becomes
	\begin{equation} \label{invlap>0*}
		\begin{split}
		y_m(x,t) =& \frac{1}{2\pi i} \sum_{j=1}^3  \int_{\xi_1^+ + \eta_1^+}^{\infty} e^{i\omega(\xi)t} \frac{\Delta_{j,m}^*(\xi)}{\Delta^*(\xi)} e^{\lambda_j^*(\xi)x} (3\beta\xi^2 - 2\alpha \xi - \delta) \tilde{\psi_m}^*(\xi) d\xi \\
		&+\frac{1}{2\pi i}\sum_{j=1}^3 \int_{C_2^+} e^{st} \frac{\Delta_{j,m}(s)}{\Delta(s)} e^{\lambda_j(s)x} \tilde{\psi_m}(s) ds \\
		&+ \frac{1}{2\pi i} \sum_{j=1}^3  \int_{-\infty}^{\xi_2^+ - \eta_2^+} e^{i\omega(\xi)t} \frac{\Delta_{j,m}^*(\xi)}{\Delta^*(\xi)} e^{\lambda_j^*(\xi)x} (3\beta\xi^2 - 2\alpha \xi - \delta) \tilde{\psi_m}^*(\xi) d\xi \\
		\doteq& y_{m,1}^+(x,t) + y_{m,2}^+(x,t) + y_{m,3}^+(x,t),
		\end{split}
	\end{equation}
	where the superscript $^*$ stands for the transformed functions under the change of variable given above. Note also that, with respect to the new variable, we have the following explicit representation for the roots of the characteristic equation \eqref{char_eqn}:
	\begin{equation} \label{char_roots_>0}
		\begin{split}
		\lambda_1^*(\xi) &= i\xi, \\
		\lambda_{2}^*(\xi) &= \frac{-i(\beta \xi - \alpha) - \sqrt{3\beta^2 \xi^2 -2\beta\alpha\xi - \alpha^2 - 4\beta\delta}}{2\beta}, \\
		\lambda_{3}^*(\xi) &= \frac{-i(\beta \xi - \alpha) + \sqrt{3\beta^2 \xi^2 -2\beta\alpha\xi - \alpha^2 - 4\beta\delta}}{2\beta}.
		\end{split}
	\end{equation}
	
	\item[(ii)] Let $\alpha^2 + 3\beta\delta = 0$. Then we see that \eqref{critic_rel} does not hold for any $k,l \in \mathbb{Z}^+$, thus the eigenvalue problem \eqref{eigenprob} has only trivial solution. But this contradicts with the fact that $\varphi$ is an eigenvalue. Thus $\Re (A\varphi, \varphi) < 0$ and the real parts of the all eigenvalues of the operator $A$ are strictly negative. On the other hand
	\begin{equation*}
		s^0 \doteq \frac{i\alpha^3}{27 \beta^2}
	\end{equation*}
	is the point where \eqref{char_eqn} assumes triple root (see Appendix \ref{app_charroots} for details). Thus the integrand of \eqref{inv_lap2} is continuous for all $r > 0$ and we can shift the contour of integration onto the imaginary axis, provided that we avoid $s^0$ by a half-circular arc to the right with a radius $\rho > 0$ denoted by $C_2^0$. Defining also $C_1^0 \doteq (s^0 + i\rho, i\infty)$ and $C_3^0 \doteq (-i\infty, s^0 - i\rho)$ (see Figure \ref{fig:int_path0} below)
	\begin{figure}[H]
		\begin{tikzpicture}[decoration={markings,
				mark=at position 1cm   with {\arrow[line width=1.5pt]{stealth}},
				mark=at position 4.5cm with {\arrow[line width=1.5pt]{stealth}},
				mark=at position 7cm   with {\arrow[line width=1.5pt]{stealth}},
				mark=at position 9.5cm with {\arrow[line width=1.5pt]{stealth}}
			}]
			
			\draw [<->] (-3,0) -- (0,0) -- (3,0) node[right]{$\Re(s)$};
			\draw [<->] (0,-1.2) -- (0,0) -- (0,4.7) node[right]{$\Im(s)$};
			
			\draw[dashed]
			(0.5,1.7) node[right, blue] {$C_2^0$} -- (0,1.7) node[left] {$s^0$};
			\filldraw
			(0,1.7) circle (1.25pt) node[align=center, below] {};
			
			\tikzset{->-/.style={decoration={
						markings,
						mark=at position .5 with {\arrow{>}}},postaction={decorate}}}
			
			\draw [->-, blue, very thick] (0,-0.8) -- (0,0.2) node[right] {$C_3^0$} --  (0,1.2) {};
			\draw [->-, blue, very thick] (0,1.2) arc[start angle=-90, end angle=90, radius=5mm] {};
			\draw [->-, blue, very thick] (0,2.2) -- (0,3.2) node[right] {$C_1^0$} --  (0,4.2) {};			
		\end{tikzpicture}
		\caption{Integration path for the case $\alpha^2 + 3\beta\delta = 0$.} \label{fig:int_path0}
	\end{figure}
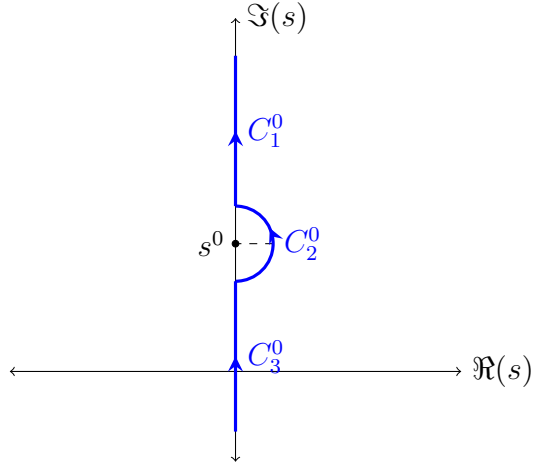
	we can express \eqref{inv_lap2} as
	\begin{equation} \label{invlap=0}
		\begin{split}
			y_m(x,t) =& \frac{1}{2\pi i} \sum_{j=1}^3  \int_{C_1^0} e^{st} \frac{\Delta_{j,m}(s)}{\Delta(s)} e^{\lambda_j(s)x} \tilde{\psi_m}(s) ds
			\\ &+\frac{1}{2\pi i}\sum_{j=1}^3 \int_{C_2^0} e^{st} \frac{\Delta_{j,m}(s)}{\Delta(s)} e^{\lambda_j(s)x} \tilde{\psi_m}(s) ds \\
			&+ \frac{1}{2\pi i} \sum_{j=1}^3  \int_{C_3^0} e^{st} \frac{\Delta_{j,m}(s)}{\Delta(s)} e^{\lambda_j(s)x} \tilde{\psi_m}(s) ds.
		\end{split}
	\end{equation}
	Now let us consider changing variables as $s = i\omega(\xi) = i(\beta \xi^3 - \alpha \xi^2 - \delta \xi)$ for the first and the third integrals. For $\alpha^2 + 3\beta\delta = 0$, note that $\omega(\xi)$ is nondecreasing and $(\xi^0,s^0)$ is the inflection point of $\omega(\xi)$, where after some calculations, $\xi^0$ can be obtained as
	\begin{equation*}
	\xi^0 \doteq \frac{\alpha}{3\beta}.
	\end{equation*}
	See Figure \ref{fig:int_trans_=0} for a graphical illustration.
	\begin{figure}[h]
		\begin{tikzpicture}
		\begin{axis}[
		axis x line=none,
		axis y line=none,
		ticks=none,
		xmin=-1.5,
		xmax=1.5,
		ymin=-2.5,
		ymax=2.5]
		\addplot [domain=-0.78:0.78, smooth, thick] { 3*x^3 } node[above,pos=1] {$\omega(\xi)$};
		
		\addplot [dotted] coordinates {(0,0) (0,-1.5)}
		node[below,pos=1] {$\xi^0$};
		
		\addplot [dotted] coordinates {(0,0) (-0.8,0)}
		node[left,pos=1] {$\Im(s^0)$};
		
		\addplot [->] coordinates {(-1.1,-2.3) (-0.7,-2.3) }
		node[right,pos=1] {$\xi$};
		
		\addplot [->] coordinates {(-1.1,-2.3) (-1.1,-1.5) }
		node[above,pos=1] {$\Im(s)$};
		\end{axis}
		\end{tikzpicture}
		\caption{Plot of transformation $\Im(s) = \omega(\xi)$ when $\alpha^2 + 3\beta\delta = 0$.} \label{fig:int_trans_=0}
	\end{figure}
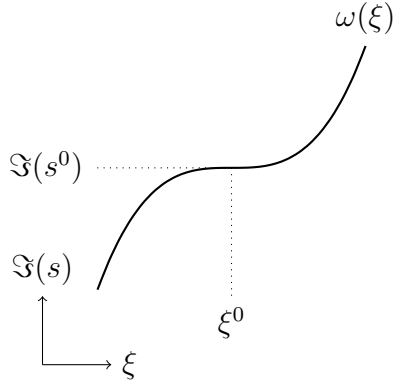
	Hence we can find the unique inverse images of $C_1^0$ and $C_3^0$ as $(\xi^0 + \eta_1^0,\infty)$ and $(-\infty,\xi^0-\eta_2^0)$ respectively for some $\eta_1^0, \eta_2^0 > 0$. Thus \eqref{invlap=0} becomes
	\begin{equation} \label{invlap=0*}
	\begin{split}
	y_m(x,t) =&  \frac{1}{2\pi i} \sum_{j=1}^3  \int_{\xi^0 + \eta_1^0}^{\infty} e^{i\omega(\xi)t} \frac{\Delta_{j,m}^*(\xi)}{\Delta^*(\xi)} e^{\lambda_j^*(\xi)x} (3\beta\xi^2 - 2\alpha \xi - \delta) \tilde{\psi_m}^*(\xi) d\xi \\
	&+\frac{1}{2\pi i}\sum_{j=1}^3 \int_{C_2^+} e^{st} \frac{\Delta_{j,m}(s)}{\Delta(s)} e^{\lambda_j(s)x} \tilde{\psi_m}(s) ds \\
	&+ \frac{1}{2\pi i} \sum_{j=1}^3  \int_{-\infty}^{\xi^0 - \eta_2^0} e^{i\omega(\xi)t} \frac{\Delta_{j,m}^*(\xi)}{\Delta^*(\xi)} e^{\lambda_j^*(\xi)x} (3\beta\xi^2 - 2\alpha \xi - \delta) \tilde{\psi_m}^*(\xi) d\xi \\
	\doteq& y_{m,1}^0(x,t) + y_{m,2}^0(x,t) + y_{m,3}^0(x,t),
	\end{split}
	\end{equation}
	where $\lambda_j^*$'s are given as in \eqref{char_roots_>0}.
	
	\item[(iii)]Let $\alpha^2 + 3\beta \delta < 0$. Then \eqref{critic_rel} does not hold for any $k, l \in \mathbb{Z}^+$ and all eigenvalues lie on the left half complex plane. On the other hand, there exits two values of $s$ for which \eqref{char_eqn} assumes double root. These values, say $s_1^-$ and $s_2^-$ with $\Re(s_1^-) > 0 > \Re(s_2^-)$ which are symmetric with respect to the imaginary axis (see Appendix \ref{app_charroots}), are also the branch points of the square root function
	\begin{equation*}
		\sqrt{(s - s_1^-)(s - s_2^-)}
	\end{equation*}
	where we choose the branch cut as $\{s \in \mathbb{C} : \Im(s) = \Im (s_1^-), \Re(s_2^-) < \Re(s) < \Re(s_1^-) \}$. Indeed changing variables as $s = i\Im(s_1^-) + r$ and than making some calculations, roots of the characteristic equation \eqref{char_eqn} can be expressed as
	\begin{equation*}
		\lambda_j^\dagger(r) = \frac{1}{3\beta} \left(i\alpha - \frac{\alpha^2 + 3\beta\delta}{\Lambda_j(r)} + \Lambda_j(r)\right),
	\end{equation*}
	where
	\begin{equation*}
		\Lambda_j(r) = -3 \beta^{\frac{2}{3}} e^{\frac{2\pi i j}{3}} \left(\frac{1 - \sqrt{r^2 + \frac{4(\alpha^2+3\beta\delta)}{729 \beta^4}}}{2}\right)^{\frac{1}{3}}, \quad j = 1, 2, 3.
	\end{equation*}
	Note that $\Re(s_1^-)$ and $\Re(s_2^-)$ are now the zeros of the square root above.
	
	In conclusion, what distinguishes this case from the previous cases is that, we have now a zero of $\Delta(s)$ that lies on the right half complex plane which is at the endpoint of the branch cut. Therefore, to deform the integration path, we first shift the vertical integration line to the left until we meet $s_1^-$. Then we deform a part of the path near $s_1^-$ by a half-circular arc to the right with a radius $\rho > 0$. Next we deform the rest of the integration path as, first by horizontal line segments to the left starting from the end points of the arc through the imaginary axis and second continuing from the imaginary axis in the vertical direction towards $+i\infty$ and $-i\infty$ respectively. See Figure \ref{fig:int_path} for the path deformation described here.
	\begin{figure}[H]
		\begin{tikzpicture}[decoration={markings,
			mark=at position 1cm   with {\arrow[line width=1.5pt]{stealth}},
			mark=at position 4.5cm with {\arrow[line width=1.5pt]{stealth}},
			mark=at position 7cm   with {\arrow[line width=1.5pt]{stealth}},
			mark=at position 9.5cm with {\arrow[line width=1.5pt]{stealth}}
		}]
	
		\draw [<->] (-3,0) -- (0,0) -- (3,0) node[right]{$\Re(s)$};
		\draw [<->] (0,-1.2) -- (0,0) -- (0,3.7) node[right]{$\Im(s)$};

		\draw[dotted]
		(1.6,1.25) -- (1.6,0) node[below] {$\Re(s_1^-)$};
		\draw[dotted]
		(1.6,1.25) -- (0,1.25) node[above, left] {};
		
		\draw (-1.2,1.25) node[cross=2pt,rotate=0] {};
		\draw (-0.8,1.25) node[cross=2pt,rotate=0] {};
		\draw (-0.4,1.25) node[cross=2pt,rotate=0] {};
		\draw (0,1.25) node[cross=2pt,rotate=0] {};
		\draw (0.4,1.25) node[cross=2pt,rotate=0] {};
		\draw (0.8,1.25) node[cross=2pt,rotate=0] {};
		\draw (1.2,1.25) node[cross=2pt,rotate=0] {};
		\filldraw
		(-1.6,1.25) circle (1.25pt) node[align=center, above] {$s_2^{-}$};
		\filldraw
		(1.6,1.25) circle (1.25pt) node[align=center, below] {};

		\tikzset{->-/.style={decoration={
					markings,
					mark=at position .5 with {\arrow{>}}},postaction={decorate}}}
				
		\draw [->-, blue, very thick] (0,-0.7) -- (0,-0.3) node[right] {$C_5^-$} --  (0,0.85) {};
		\draw [->-, blue, very thick] (0,0.85) -- (0.8,0.85) node[below] {$C_4^-$} -- (1.6,0.85) {};
		\draw [->-, blue, very thick] (1.6,0.85) arc[start angle=-90, end angle=90, radius=4mm]  node[right] {$\text{ }\text{ }C_3^-$} {};
		\draw [->-, blue, very thick] (1.6,1.65) -- (0.8,1.65) node[above] {$C_2^-$} -- (0,1.65) {};
		\draw [->-, blue, very thick] (0,1.65) -- (0,2.5) node[right] {$C_1^-$} -- (0,3.2) {};
		
		\end{tikzpicture}
		\caption{Integration path for the case $\alpha^2 + 3\beta\delta < 0$.} \label{fig:int_path}
	\end{figure}
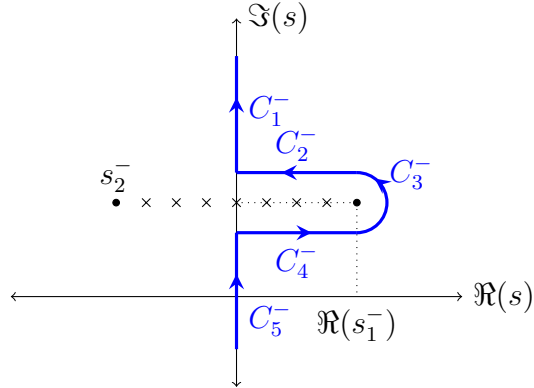
	Consequently, we can write \eqref{inv_lap2} as
	\begin{equation} \label{invlap<0}
	\begin{split}
	y_m(x,t) =& \frac{1}{2\pi i} \sum_{j=1}^3  \int_{C_1^-} e^{st} \frac{\Delta_{j,m}(s)}{\Delta(s)} e^{\lambda_j(s)x} \tilde{\psi_m}(s) ds \\
	&+\frac{1}{2\pi i}\sum_{j=1}^3 \int_{C_2^- \cup C_3^- \cup C_4^-} e^{st} \frac{\Delta_{j,m}(s)}{\Delta(s)} e^{\lambda_j(s)x} \tilde{\psi_m}(s) ds \\
	&+ \frac{1}{2\pi i} \sum_{j=1}^3  \int_{C_5^-} e^{st} \frac{\Delta_{j,m}(s)}{\Delta(s)} e^{\lambda_j(s)x} \tilde{\psi_m}(s) ds.
	\end{split}
	\end{equation}
	Now let us apply change of variable $s = i\omega(\xi) = i(\beta\xi^3 - \alpha\xi^2 - \delta\xi)$ for the first and third integrals. Note that for $\alpha^2 + 3\beta\delta < 0$, this mapping is strictly increasing and the inverse image of $\Im(s_1^-)$ under the transformation $\omega(\xi)$ is the point
	\begin{equation*}
	\xi^- \doteq \frac{\alpha}{3\beta}.
	\end{equation*}
	Then $C_1^-$ and $C_5^-$ are mapped to  $(\xi^- + \eta_1^-,\infty)$ and $(-\infty,\xi^- - \eta_2^-)$ for some $\eta_1^-, \eta_2^-> 0$. See Figure \ref{fig:int_trans_<0}.
	\begin{figure}[h]
		\begin{tikzpicture}
		\begin{axis}[
		axis x line=none,
		axis y line=none,
		ticks=none,
		xmin=-1.5,
		xmax=1.5,
		ymin=-2.5,
		ymax=2.5]
		\addplot [domain=-0.65:0.65, smooth, thick] { 3*x^3+x } node[above,pos=1] {$\omega(\xi)$};
		
		\addplot [dotted] coordinates {(0,0) (0,-1.5)}
		node[below,pos=1] {$\xi^-$};
		
		\addplot [dotted] coordinates {(0,0) (-0.8,0)}
		node[left,pos=1] {$\Im(s_1^-)$};
		
		\addplot [->] coordinates {(-1.1,-2.3) (-0.7,-2.3) }
		node[right,pos=1] {$\xi$};
		
		\addplot [->] coordinates {(-1.1,-2.3) (-1.1,-1.5) }
		node[above,pos=1] {$\Im(s)$};
		\end{axis}
		\end{tikzpicture}
		\caption{Plot of transformation $\Im(s) = \omega(\xi)$ when $\alpha^2 + 3\beta \delta < 0$.} \label{fig:int_trans_<0}
	\end{figure}
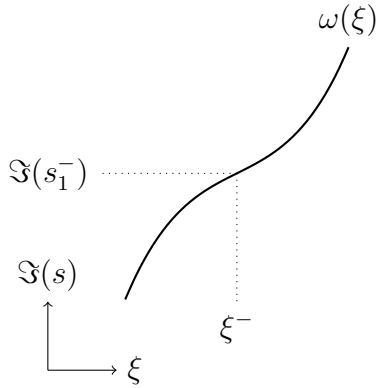
	Thus \eqref{invlap<0} becomes
	\begin{equation} \label{invlap<0*}
		\begin{split}
			y_m(x,t) =& \frac{1}{2\pi i} \sum_{j=1}^3 \int_{\xi^- + \eta_1^-}^\infty e^{i\omega(\xi)t} \frac{\Delta_{j,m}^*(\xi)}{\Delta^*(\xi)} e^{\lambda_j^*(\xi) x} (3\beta^2 \xi^2 - 2\alpha\xi - \delta) \tilde{\psi_m} d\xi \\
			&+\frac{1}{2\pi i}\sum_{j=1}^3 \int_{C_2^- \cup C_3^- \cup C_4^-} e^{st} \frac{\Delta_{j,m}(s)}{\Delta(s)} e^{\lambda_j(s)x} \tilde{\psi_m}(s) ds \\
			&+ \frac{1}{2\pi i} \sum_{j=1}^3 \int_{-\infty}^{\xi^- - \eta_2^-} e^{i\omega(\xi)t} \frac{\Delta_{j,m}^*(\xi)}{\Delta^*(\xi)} e^{\lambda_j^*(\xi) x} (3\beta^2 \xi^2 - 2\alpha\xi - \delta) \tilde{\psi_m}^*(\xi) d\xi \\
			\doteq & y_{m,1}^-(x,t) + y_{m,2}^-(x,t) + y_{m,3}^-(x,t).
		\end{split}
	\end{equation}
	where $\lambda_j^*$'s are given by \eqref{char_roots_>0}.
	\end{enumerate}
	
	In the following three lemmas we provide estimates for $y_m$ for each $m$. Note that for each solution representation \eqref{invlap>0*}, \eqref{invlap=0*} and \eqref{invlap<0*} corresponding to the different cases of $\alpha^2 + 3\beta\delta$, second integrals are bounded on the corresponding integration paths. However, these paths lie on the right half complex plane. Therefore, for a given $T > 0$, we can find $c > 0$ such that the norm estimates that we will obtain below for the first and third integrals also hold for the second integrals but with a constant $e^{cT}$. On the other hand, for each case of $\alpha^2 + 3\beta\delta$, we have to estimate the following form of integrals
	\begin{equation} \label{int_I}
		I_m(x,t) = \frac{1}{2\pi i} \sum_{j=1}^3  \int_{\gamma_1}^{\infty} e^{i\omega(\xi)t} \frac{\Delta_{j,m}^*(\xi)}{\Delta^*(\xi)} e^{\lambda_j^*(\xi)x} (3\beta\xi^2 - 2\alpha \xi - \delta) \tilde{\psi_m}^*(\xi) d\xi
	\end{equation}
	and
	\begin{equation} \label{int_J}
		J_m(x,t) =  \frac{1}{2\pi i} \sum_{j=1}^3  \int_{-\infty}^{\gamma_2} e^{i\omega(\xi)t} \frac{\Delta_{j,m}^*(\xi)}{\Delta^*(\xi)} e^{\lambda_j^*(\xi)x} (3\beta\xi^2 - 2\alpha \xi - \delta) \tilde{\psi_m}^*(\xi) d\xi,
	\end{equation}
	where $\gamma_1 \in \{\xi_1^+ + \eta_1^+,\xi^0 + \eta_1^0,\xi^- + \eta_1^-\}$ and $\gamma_2 \in \{\xi_2^+ + \eta_2^+,\xi^0 + \eta_2^0,\xi^- + \eta_2^-\}$. Thus, it is enough to study \eqref{int_I} and \eqref{int_J} in order to obtain desired norm estimates for $y_m$, $m = 1, 2, 3$.
	
	\begin{lem} \label{wp_lem_ld}
		Let $T, \beta > 0$, $\alpha,\delta \in \mathbb{R}$ and $\psi_1 \in H^{1/3}(0,T)$. Then  $y_1 = y[0,0,\psi_1,0,0]$ belongs to the space $C([0,T];L^2(0,L)) \cap L^2(0,T;H^1(0,L))$ and it also satisfies $\partial_x y_1 \in C([0,L];L^2(0,T))$. Moreover, there exists a constant $c > 0$ such that
		\begin{equation} \label{wp_est1_ld}
			\|y_1\|_{C([0,T];L^2(0,L))} + \|y_1\|_{L^2(0,T;H^1(0,L))} \lesssim e^{cT} \|\psi_1\|_{H^{1/3}(0,T)}
		\end{equation}
		and
		\begin{equation} \label{wp_est2_ld}
			\underset{x \in [0,L]}{\sup} \|\partial_xy_1(x,\cdot)\|_{L^2(0,T)} \lesssim e^{cT} \|\psi_1\|_{H^{1/3}(0,T)}.
		\end{equation}
		If $\alpha^2 + 3\beta \delta < 0$, then $c > \Re(s_1^-) > 0$ where $s_1^-$ is the value for which \eqref{char_eqn} assumes double root.
	\end{lem}
	\begin{proof}
		Let us first obtain the asymptotic behaviours of the ratios $\left|\frac{\Delta^*_{j,1}(\xi)}{\Delta^*(\xi)}\right|$ for large values of $\xi$. Using the relation $\lambda_1^* + \lambda_2^* + \lambda_3^* = \frac{i\alpha}{\beta}$, we have
		\begin{multline}\label{det_lap}
		\Delta(\xi)^* \\= e^{\frac{i\alpha L}{\beta}}
		\left(e^{-\lambda_1^*(\xi)L} (\lambda_3^*(s) - \lambda_2^*(\xi))
		- e^{-\lambda_2^*(\xi)L} (\lambda_3^*(\xi) - \lambda_1^*(\xi))
		+ e^{-\lambda_3^*(\xi)L} (\lambda_2^*(\xi) - \lambda_1^*(\xi))\right)
		\end{multline}
		and
		\begin{align}
		\label{det11_lap} \Delta_{1,1}^* (\xi) &= e^{\frac{i\alpha L}{\beta}} e^{-\lambda_1^*(\xi) L} \left(\lambda_3^*(\xi) - \lambda_2^*(\xi)\right),\\
		\label{det21_lap} \Delta_{2,1}^* (\xi) &= e^{\frac{i\alpha L}{\beta}} e^{-\lambda_2^*(\xi) L} \left(\lambda_1^*(\xi) - \lambda_3^*(\xi)\right),\\
		\label{det31_lap} \Delta_{3,1}^* (\xi) &= e^{\frac{i\alpha L}{\beta}} e^{-\lambda_3^*(\xi) L} \left(\lambda_2^*(\xi) - \lambda_1^*(\xi)\right).
		\end{align}
		Using the roots of the characteristic equation \eqref{char_roots_>0} in the variable $\xi$, we obtain the following large $\xi$ asymptotics
		\begin{equation}  \label{asym_cj_right1}
		\left|\frac{\Delta^*_{j,1}(\xi)}{\Delta^*(\xi)}\right| \sim
		\begin{cases}
		e^{-\frac{\sqrt{3}\xi L}{2}}, &j = 1, \\
		1, &j = 2, \\
		e^{-\sqrt{3}\xi L}, &j = 3.
		\end{cases}
		\end{equation}
		Let us start by taking $L^2-$norm of $I_1$ with respect to its first component and apply \cite[Lemma 2.5]{Bona03} to get
		\begin{equation*}
			\|I_1(\cdot,t)\|_2^2 \lesssim  \sum_{j=1}^3 \int_{\gamma_1}^{\infty} \left(e^{L \Re(\lambda_j^*(\xi))} + 1\right)^2 \left|\frac{\Delta_{j,1}^*(\xi)}{\Delta^*(\xi)}\right|^2 \left|3\beta\xi^2 - 2\alpha \xi - \delta\right|^2 \left|\tilde{\psi_1}^*(\xi)\right|^2 d\xi.
		\end{equation*}
		Using the asymptotic behaviours \eqref{asym_cj_right1}, we have
		\begin{equation}\label{lem23_>0_intasym}
			\left(e^{L \Re(\lambda_j^*(\xi))} +1\right)^2 \left|\frac{\Delta_{j,1}^*(\xi)}{\Delta^*(\xi)} \right|^2  \sim
			\begin{cases}
				e^{-\sqrt{3}\xi L}, & j = 1,\\
				1, & j = 2,\\
				e^{-\sqrt{3}\xi L}, & j = 3,
			\end{cases}
		\end{equation}
		as $\xi \to \infty$. Thus, we can write
		\begin{equation*}
			\|I_1(\cdot,t)\|_2^2 \lesssim \int_{\gamma_1}^{\infty} \left|3\beta\xi^2 - 2\alpha \xi - \delta\right|^2 \left|\tilde{\psi_1}^*(\xi)\right|^2 d\xi.
		\end{equation*}
		Changing variables as $\mu = \beta \xi^3 - \alpha \xi^2 - \delta \xi$, we get
		\begin{equation*}
			\begin{split}
				\|I_1(\cdot,t)\|_2^2 &\lesssim \int_{\omega(\gamma_1)}^{\infty} (1 + \mu^2)^{\frac{1}{3}} \left| \int_0^\infty e^{-i\mu\tau}\psi_1(\tau)d\tau\right|^2 d\mu \\
				&\leq \|\psi_1\|_{H^{1/3}(0,T)}^2.
			\end{split}
		\end{equation*}
		Passing to supremum over $t \in [0,T]$, we obtain
		\begin{equation} \label{lem23_>01_est1}
			\|I_1\|_{C([0,T];L^2(0,L))} \lesssim \|\psi_1\|_{H^{1/3}(0,T)}.
		\end{equation}
			
		Next, we differentiate $I_1$ with respect to its first component, take $L^2-$norm on $(0,T)$ and change variables as $\mu = \beta \xi^3 - \alpha \xi^2 - \delta \xi$ to get
		\begin{equation} \label{lem23_>0-l2t}
			\begin{split}
			&\|\partial_xI_1(x,\cdot)\|_{L^2(0,T)}^2 \\
			=& \left\|\sum_{j=1}^3 \frac{1}{2\pi} \int_{\gamma_1}^{\infty} e^{i\omega(\xi)t} \lambda_j^*(\xi) e^{\lambda_j^*(\xi)x} \frac{\Delta_{j,1}^*(\xi)}{\Delta^*(\xi)} (3\beta\xi^2 - 2\alpha \xi - \delta) \tilde{\psi_1}^*(\xi) d\xi \right\|_{L^2(0,T)}^2 \\
			\lesssim& \sum_{j=1}^3  \left\|\int_{\omega(\gamma_1)}^{\infty} e^{i \mu t} \lambda_j^*(\theta(\mu)) e^{\lambda_{j}^*(\theta(\mu))x} \frac{\Delta_{j,1}^*(\theta(\mu))}{\Delta^*(\theta(\mu))} \tilde{\psi_1}^*(\theta(\mu)) d\mu \right\|_{L^2(0,T)}^2,
			\end{split}
		\end{equation}
		where $\theta(\mu)$ is the real solution of $\mu = \beta\xi^3 - \alpha \xi^2 - \delta \xi$ for $\gamma_1 < \xi < \infty$. Observe that the function
		\begin{equation*}
			\begin{cases}
			\lambda_j^*(\theta(\mu)) e^{\lambda_j^*(\theta(\mu))x} \frac{\Delta_{j,1}^*(\theta(\mu))}{\Delta^*(\theta(\mu))} \tilde{\psi_1}^*(\theta(\mu)), & \mu \in (\omega(\gamma_1),\infty), \\
			0, & \text{elsewhere,}
			\end{cases}
		\end{equation*}
		is the Fourier transform of the function given by the integral. So, thanks to the Plancherel's theorem, we can write
		\begin{equation}\label{lem23_>0_passup}
			\|\partial_x I_1(x,\cdot)\|_{L^2(0,T)}^2 \lesssim \sum_{j=1}^3 \int_{\omega(\gamma_1)}^{\infty} \left| \lambda_j^*(\theta(\mu)) e^{\lambda_j^*(\theta(\mu))x} \frac{\Delta_{j,1}^*(\theta(\mu))}{\Delta^*(\theta(\mu))} \tilde{\psi_1}^*(\theta(\mu)) \right|^2  d\mu
		\end{equation}
		for all $x \in [0,L]$. It follows that
		\begin{multline}\label{lem23_>0l2l2}
			\|\partial_x I_1\|_{L^2(0,L;L^2(0,T))}^2 \leq \underset{x \in [0,L]}{\sup} \|\partial_x I_1(x,\cdot)\|_{L^2(0,T)}^2 \\ \lesssim \sum_{j=1}^3 \int_{\gamma_1}^\infty |\lambda_j^*(\xi)|^2 \underset{x \in [0,L]}{\sup} \left(e^{2 \Re(\lambda_j^*(\xi))x} \right) \left|\frac{\Delta_{j,1}^*(\xi)}{\Delta^*(\xi)}\right|^2 (3\beta\xi^2 - 2\alpha\xi - \delta) \left|\tilde{\psi_1}^*(\xi) \right|^2 d\xi.
		\end{multline}
		Using \eqref{char_roots_>0} and \eqref{asym_cj_right1}, one can obtain the following asymptoic behaviours in $\xi$
		\begin{equation}\label{lem23_>0_intasym2}
			|\lambda_j^*(\xi)|^2 \underset{x \in [0,L]}{\sup} \left(e^{2 \Re(\lambda_j^*(\xi))x} \right) \left|\frac{\Delta_{j,1}^*(\xi)}{\Delta^*(\xi)}\right|^2 \sim
			\begin{cases}
				\xi^2 e^{-\sqrt{3}\xi L}, & j = 1, \\
				\xi^2, & j = 2, \\
				\xi^2 e^{-\sqrt{3}\xi L}, & j = 3.
			\end{cases}
		\end{equation}			
		Using \eqref{lem23_>0_intasym2} in \eqref{lem23_>0l2l2}, and then changing variables back as $\mu = \beta \xi^3 - \alpha \xi^2 - \delta \xi$, we get
		\begin{equation} \label{lem23_>0_intasym3}
			\begin{split}
				\|\partial_x I_1\|_{L^2(0,L;L^2(0,T))}^2 \leq& \underset{x \in [0,L]}{\sup} \|\partial_x I_1(x,\cdot)\|_{L^2(0,T)}^2 \\
				\lesssim& \int_{\gamma_1}^\infty \xi^2 (3\beta \xi ^2 - 2\alpha \xi - \delta) |\tilde{\psi_1}^*(\xi)|^2 d\xi \\
				\lesssim& \int_{\omega(\gamma_1)}^\infty (1 + \mu^2)^\frac{1}{3} \left| \int_0^\infty e^{-i\mu\tau}\psi_1(\tau)d\tau\right|^2 d\mu \\
				\lesssim& \|\psi_1\|_{H^{1/3}(0,T)}^2
			\end{split}
		\end{equation}
		Changing the integration order on $\|\partial_x I_1\|_{L^2(0,L;L^2(0,T))}^2$ and using Poincare inequality, we conclude that \eqref{wp_est1_ld} and \eqref{wp_est2_ld} holds for $I_1$.
	
		To show that the mapping $x \in [0,L] \to \|\partial_x I_1(x,\cdot)\|_{L^2(0,T)}$ is continuous, let $\{x_n\}_{n\in \mathbb{N}} \subset [0,L]$ be such that $x_n \to x$ as $n \to \infty$ and let us write
		\begin{multline}
			\partial_x I_1 (x,t) - \partial_x I_1(x_n,t) \\
			= \frac{1}{2\pi i} \sum_{j = 1}^3 \int_{\gamma_1}^\infty e^{i \omega(\xi) t} \lambda_j^*(\xi) \left(e^{\lambda_j^*(\xi) x} - e^{\lambda_j^*(\xi)x_n}\right) \frac{\Delta_{j,1}^*(\xi)}{\Delta^*(\xi)} \tilde{\psi_1^*}(\xi) d\xi
		\end{multline}
		Applying the arguments above in \eqref{lem23_>0-l2t}-\eqref{lem23_>0_intasym3}, one can obtain that
		\begin{equation*}
			\begin{split}
				&\|\partial_x I_1 (x,\cdot) - \partial_x I_1(x_n,\cdot)\|_{L^2(0,T)}^2 \\
				\lesssim& \sum_{j=1}^3 \int_{\omega(\gamma_1)}^{\infty} \left| \lambda_j^*(\theta(\mu)) \left( e^{\lambda_j^*(\theta(\mu))x} - e^{\lambda_j^*(\theta(\mu))x_n}\right) \frac{\Delta_{j,1}^*(\theta(\mu))}{\Delta^*(\theta(\mu))} \tilde{\psi_1}^*(\theta(\mu)) \right|^2  d\mu \\
				\lesssim& \|\psi_1\|_{H^{1/3}(0,T)}^2,
			\end{split}
		\end{equation*}
		for all $n \in \mathbb{N}$. Hence, by the dominated convergence theorem, we see that
		\begin{equation*}
			\lim_{n \to \infty} \|\partial_x I_1 (x,\cdot) - \partial_x I_1(x_n,\cdot)\|_{L^2(0,T)} \to 0.
		\end{equation*}
		
		Applying a similar procedure yields the same results for $J_1$.
	\end{proof}
	
	\begin{lem} \label{wp_lem_rd}
		Let $T, \beta > 0$, $\alpha,\delta \in \mathbb{R}$ and $\psi_2 \in H^{1/3}(0,T)$. Then $y_2 = y[0,0,0,\psi_2,0]$ belongs to the space $C([0,T];L^2(0,L)) \cap L^2(0,T;H^1(0,L))$ and also satisfies $\partial_x y_2 \in C([0,L];L^2(0,T))$. Moreover, there exists a constant $c > 0$ such that
		\begin{equation} \label{wp_est1_rd}
			\|y_2\|_{C([0,T];L^2(0,L))} + \|y_2\|_{L^2(0,T;H^1(0,L))} \lesssim e^{cT} \|\psi_2\|_{H^{1/3}(0,T)}
		\end{equation}
		and
		\begin{equation} \label{wp_est2_rd}
			\underset{x \in [0,L]}{\sup} \|\partial_xy_2(x,\cdot)\|_{L^2(0,T)} \lesssim e^{cT} \|\psi_2\|_{H^{1/3}(0,T)}.
		\end{equation}
		If $\alpha^2 + 3\beta \delta < 0$, then $c > \Re(s_1^-) > 0$ where $s_1^-$ is the value for which \eqref{char_eqn} assumes double root.
	\end{lem}
	\begin{proof}
		We start by obtaining large $\xi$ asymptotics for $\left|\frac{\Delta_{j,2}^*(\xi)}{\Delta^*(\xi)}\right|$. Let us write
		\begin{align}
		\label{det12_lap} \Delta_{1,2}^* (\xi) &= \lambda_2^*(\xi) e^{\lambda_2^*(\xi) L} - \lambda_3^*(\xi) e^{\lambda_3^*(\xi) L},\\
		\label{det22_lap} \Delta_{2,2}^* (\xi) &= \lambda_3^*(\xi) e^{\lambda_3^*(\xi) L} - \lambda_1^*(\xi) e^{\lambda_1^*(\xi) L},\\
		\label{det32_lap} \Delta_{3,2}^* (\xi) &= \lambda_1^*(\xi) e^{\lambda_1^*(\xi) L} - \lambda_2^*(\xi) e^{\lambda_2^*(\xi) L},
		\end{align}
		and then, use $\Delta^*$ given in \eqref{det_lap} and characteristic roots given in \eqref{char_roots_>0} to obtain
		\begin{equation}  \label{asym_cj_right2}
		\left|\frac{\Delta^*_{j,2}(\xi)}{\Delta^*(\xi)}\right| \sim
		\begin{cases}
		1, &j = 1, \\
		1, &j = 2, \\
		e^{-\frac{\sqrt{3}\xi L}{2}}, &j = 3.
		\end{cases}
		\end{equation}
		The rest of the proof is as in the proof of previous lemma.
		\end{proof}
	
	\begin{lem} \label{wp_lem_rn}
		Let $T, \beta > 0$, $\alpha,\delta \in \mathbb{R}$ and $\psi_3 \in L^2(0,T)$. Then $y_3 = y[0,0,\psi_3,0,0]$ belongs to the space $C([0,T];L^2(0,L)) \cap L^2(0,T;H^1(0,L))$ and also satisfies $\partial_x y_3 \in C([0,L];L^2(0,T))$. Moreover, there exists a constant $c > 0$ such that
		\begin{equation} \label{wp_est1_rn}
			\|y_3\|_{C([0,T];L^2(0,L))} + \|y_3\|_{L^2(0,T;H^1(0,L))} \lesssim e^{cT} \|\psi_3\|_{L^2(0,T)}
		\end{equation}
		and
		\begin{equation} \label{wp_est2_rn}
			\underset{x \in [0,L]}{\sup} \|\partial_xy_3(x,\cdot)\|_{L^2(0,T)} \lesssim e^{cT} \|\psi_3\|_{L^2(0,T)}
		\end{equation}
		If $\alpha^2 + 3\beta \delta < 0$, then $c > \Re(s_1^-) > 0$ where $s_1^-$ is the value for which \eqref{char_eqn} assumes double root.
	\end{lem}

	\begin{proof}
	As in the previous proofs, let us obtain large $\xi$ asymptotics of the ratios $\left|\frac{\Delta_{j,3}^*(\xi)}{\Delta^*(\xi)}\right|$. To this end, we write
	\begin{align}
	\label{det13_lap} \Delta_{1,3}^* (\xi) &= e^{\lambda_3^*(\xi) L} - e^{\lambda_2^*(\xi) L},\\
	\label{det23_lap} \Delta_{2,3}^* (\xi) &= e^{\lambda_1^*(\xi) L} - e^{\lambda_3^*(\xi) L},\\
	\label{det33_lap} \Delta_{3,3}^* (\xi) &= e^{\lambda_2^*(\xi) L} - e^{\lambda_1^*(\xi) L},
	\end{align}
	and then use $\Delta^*$ given in \eqref{det_lap} and characteristic roots given in \eqref{char_roots_>0} to obtain
	\begin{equation}  \label{asym_cj_right3}
	\left|\frac{\Delta^*_{j,3}(\xi)}{\Delta^*(\xi)}\right| \sim
	\begin{cases}
	\xi^{-1}, &j = 1, \\
	\xi^{-1}, &j = 2, \\
	\xi^{-1} e^{-\frac{\sqrt{3}\xi L}{2}}, &j = 3.
	\end{cases}
	\end{equation}
	Proceeding as in the proof of Lemma \ref{wp_lem_ld}, we can obtain the desired result.
	\end{proof}
	
	Combining Lemmas \ref{wp_lem_ld}, \ref{wp_lem_rd} and \ref{wp_lem_rn}, we obtain the following result for $y[0,0,\psi_1,\psi_2,\psi_3]$.
	
	\begin{lem} \label{wp_lem_3}
		Let $\phi \equiv f \equiv 0$. For $T > 0$ and $(\psi_1,\psi_2,\psi_3) \in H^{1/3}(0,T) \times H^{1/3}(0,T) \times L^2(0,T)$, \eqref{wp_pde} admits a unique solution which belongs to the space $C([0,T];L^2(0,L)) \cap L^2(0,T;H^1(0,L))$ with $y_x \in C([0,L];L^2(0,T))$. Moreover there exists a constant $c > 0$ such that
		\begin{equation} \label{wp_est_5}
			\|y_3\|_{C([0,T];L^2(0,L))} + \|y_3\|_{L^2(0,T;H^1(0,L))} \lesssim e^{cT} \|\psi_3\|_{L^2(0,T)}
		\end{equation}
		and
		\begin{equation} \label{wp_est_6}
			\underset{x \in [0,L]}{\sup} \|\partial_xy_3(x,\cdot)\|_{L^2(0,T)} \lesssim e^{cT} \|\psi_3\|_{L^2(0,T)}
		\end{equation}
		If $\alpha^2 + 3\beta \delta < 0$, then $c > \Re(s_1^-) > 0$ where $s_1^-$ is the value for which \eqref{char_eqn} assumes double root.
	\end{lem}
		
	Now for the sake of our study, let us consider the problem
	\begin{eqnarray} \label{wp_pde_hreg}
		\begin{cases}
			iy_t + i\beta y_{xxx} +\alpha y_{xx} +i\delta y_x = 0, x\in (0,L), t\in (0,T),\\
			y(0,t)=0, y(L,t)=0, y_x(L,t)=\psi(t),\\
			y(x,0)=\phi(x),
		\end{cases}
	\end{eqnarray}
	From Lemma \ref{wp_lem_1} and Lemma \ref{wp_lem_rn}, we know that solution $y$ of \eqref{wp_pde_hreg} belongs to the space $X_T^0$ and satisfies
	\begin{equation*}
		\|y\|_{X_T^0} \lesssim \|\phi\|_2 + e^{cT}\|\psi\|_{L^2(0,T)}.
	\end{equation*}
	Let $v = y_t$. Then $v$ solves the linear model below
	\begin{eqnarray}\label{targetfv}
		\begin{cases}
			iv_t + i\beta v_{xxx} +\alpha v_{xx} +i\delta v_x = 0, x\in (0,L), t\in (0,T),\\
			v(0,t)=0, v(L,t)=0, v_x(L,t)=\psi^\prime(t),\\
			v(x,0)=\tilde{\phi}(x),
		\end{cases}
	\end{eqnarray}
	where $\tilde{\phi} \doteq -\beta \phi^{\prime\prime\prime} +i\alpha \phi^{\prime\prime} -\delta \phi^\prime$. Assume that $\tilde{\phi} \in L^2(0,L)$. From Lemma \ref{wp_lem_1} and Lemma \ref{wp_lem_rn}, $v$ satisfies
	\begin{equation*}
		\|v\|_{X_T^0} \lesssim \|\tilde\phi\|_2 + e^{cT}\|\psi^\prime\|_{L^2(0,T)}.
	\end{equation*}
	Set $y(x,t) = \phi(x) + \int_0^t v(x,\tau) d\tau$. Then due to compatibility conditions  and $\phi^\prime(L) = \psi(0)$, $y$ satisfies the initial and boundary conditions
	\begin{equation*}
		\begin{split}
		y(x,0) &= \phi(x), \\
		y(0,t) &= \phi(0) + \int_0^t v(0,\tau) d\tau = \phi(0) + y(0,t) - y(0,0) = 0,\\
		y(L,t) &= \phi(L) + \int_0^t v(L,\tau) d\tau = \phi(L) + y(L,t) - y(L,0) = 0, \\
		y_x(L,t) &= \phi^\prime(L) + \int_0^t v_x(L,\tau) d\tau = \phi^\prime(L) + y_x(L,t) - y_x(L,0) = \psi(t).
		\end{split}
	\end{equation*}
	Moreover,
	\begin{equation*}
		\begin{split}
		(i y_t &+ i\beta y_{xxx} + \alpha y_{xx} + i\delta y_x)(x,t)\\
		=& i v(x,t) + \int_0^t (i\beta v_{xxx} + \alpha v_{xx} + i\delta v_x)(x,\tau) d\tau + i\beta \phi^{\prime\prime\prime}(x) +\alpha \phi^{\prime\prime}(x) + i\delta \phi^\prime(x)\\
		=& i v(x,0) + \int_0^t iv_t(x,\tau) d\tau + \int_0^t (i\beta v_{xxx} + \alpha v_{xx} + i\delta v_x)(x,\tau) d\tau \\
		&+ i\beta \phi^{\prime\prime\prime}(x) +\alpha \phi^{\prime\prime}(x) + i\delta \phi^\prime(x) \\
		=&  i v(x,0) + \int_0^t (-i\beta v_{xxx} -\alpha v_{xx} -i\delta v_x)(x,\tau) d\tau \\
		&+ \int_0^t (i\beta v_{xxx} + \alpha v_{xx} + i\delta v_x)(x,\tau) d\tau + i\beta \phi^{\prime\prime\prime}(x) +\alpha \phi^{\prime\prime}(x) + i\delta \phi^\prime(x)\\
		=& 0.
		\end{split}
	\end{equation*}
	Thus, $y$ solves \eqref{wp_pde_hreg}. Now, from the main equation of \eqref{wp_pde_hreg}, we have
	\begin{equation} \label{xt3-eq1}
		\beta \|y_{xxx}(\cdot,t)\|_2 \leq \|v(\cdot,t)\|_2 +\alpha \|y_{xx}(\cdot,t)\|_2 +\delta \|y_x(\cdot,t)\|_2
	\end{equation}
	Applying Gagliardo--Nirenberg interpolation inequality and then $\epsilon-$Young's inequality to the second term at the right hand side of \eqref{xt3-eq1}, we get
	\begin{equation}
		\begin{split}
			\label{xt3-eq2}
			\alpha^2 \|y_{xx}(\cdot,t)\|_2^2 &\leq c_\alpha \|y_{xxx}(\cdot,t)\|_2^\frac{4}{3} \|y(\cdot,t)\|_2^\frac{2}{3} \\
			&\leq \epsilon \|y_{xxx}(\cdot,t)\|_2^2 + c_{\alpha,\epsilon } \|y(\cdot,t)\|_2^2.
		\end{split}
	\end{equation}
	Similarly, for the third term at the right hand side of \eqref{xt3-eq1}, we have
	\begin{equation}
		\begin{split}
			\label{xt3-eq3}
			\delta^2 \|y_{x}(\cdot,t)\|_2^2 &\leq  c_\delta \|y_{xxx}(\cdot,t)\|_2^\frac{2}{3} \|y(\cdot,t)\|_2^\frac{4}{3} \\
			&\leq \epsilon \|y_{xxx}(\cdot,t)\|_2^2 + c_{\delta,\epsilon } \|y(\cdot,t)\|_2^2.
		\end{split}
	\end{equation}
	Using \eqref{xt3-eq2}-\eqref{xt3-eq3} on \eqref{xt3-eq1}, we obtain
	\begin{equation*}
		(\beta - 2\epsilon)\|y_{xxx}(\cdot,t)\|_2^2 \leq \|v(\cdot,t)\|_2^2 + c_{\alpha,\delta,\epsilon}\|y(\cdot,t)\|_2^2.
	\end{equation*}
	Therefore, for sufficiently small $\epsilon > 0$, we get
	\begin{equation} \label{xt3-eq4}
		\|y_{xxx}(\cdot,t)\|_2^2 \lesssim \|v(\cdot,t)\|_2^2 + \|y(\cdot,t)\|_2^2.
	\end{equation}
	Passing to supremum over $t \in [0,T]$ and using the fact that the right hand side belongs to $C([0,T],L^2(0,L))$, we have $y \in C([0,T];H^3(0,L))$. Next, differentiating the main equation of \eqref{wp_pde_hreg} with respect to $x$ and taking $L^2-$norms of each term, we get
	\begin{equation} \label{xt3-eq5}
		\beta \|y_{xxxx}(\cdot,t)\|_2 \leq \|v_x(\cdot,t)\|_2 +\alpha \|y_{xxx}(\cdot,t)\|_2 +\delta \|y_{xx}(\cdot,t)\|_2.
	\end{equation}
	Thanks to Gagliardo--Nirenberg's interpolation inequality, $\epsilon-$Young's inequality and Poincare inequality, the second term at the right hand side of \eqref{xt3-eq5} can be estimated as
	\begin{equation}
		\label{xt3-eq6}
		\begin{split}
			\alpha^2 \|y_{xxx}(\cdot,t)\|_2^2 &\leq c_\alpha \|y_{xxxx}(\cdot,t)\|_2^\frac{3}{2} \|y(\cdot,t)\|_2^\frac{1}{2},\\
			&\leq \epsilon \|y_{xxxx}(\cdot,t)\|_2^2 + c_{\alpha,\epsilon} \|y(\cdot,t)\|_2^2 \\
			&\leq \epsilon \|y_{xxxx}(\cdot,t)\|_2^2 + c_{\alpha,\epsilon} \|y_x(\cdot,t)\|_2^2.
		\end{split}
	\end{equation}
	Using the same inequalities, the third term in \eqref{xt3-eq5} is estimated as
	\begin{equation}
		\label{xt3-eq7}
		\begin{split}
			\delta^2\|y_{xx}(\cdot,t)\|_2^2 &\leq c_\delta \|y_{xxxx}(\cdot,t)\|_2 \|y(\cdot,t)\|_2 \\
			&\leq \epsilon \|y_{xxxx}(\cdot,t)\|_2^2 + c_{\delta,\epsilon} \|y(\cdot,t)\|_2^2\\
	&\leq \epsilon \|y_{xxxx}(\cdot,t)\|_2^2 + c_{\delta,\epsilon} \|y_x(\cdot,t)\|_2^2.
		\end{split}
	\end{equation}
	Using \eqref{xt3-eq6}-\eqref{xt3-eq7} on \eqref{xt3-eq5} and choosing $\epsilon > 0$ sufficiently small, we obtain
	\begin{equation} \label{xt3-eq8}
		\|y_{xxxx}(\cdot,t)\|_2^2 \lesssim \|v_x(\cdot,t) \|_2^2 + \|y_x(\cdot,t)\|_2^2.
	\end{equation}
	Right hand side belongs to $L^2(0,T)$, so the left hand side does too. This implies $y \in L^2(0,T;H^4(0,L))$. Combining this result with the previous one, we proved the following lemma.
	\begin{lem} \label{lem_wp_pde_xt3}
		Let $T > 0$, $(\phi,\psi) \in H^3(0,L) \times H^1(0,T)$ satisfy the compatibility conditions. Then \eqref{wp_pde_hreg} has a unique solution $y \in X_T^3$ with $y_t \in X_T^0$ and it satisfies the following estimate
		\begin{equation*}
		\|y\|_{X_T^3} + \|y_t\|_{X_T^0} \lesssim \|\phi\|_{H^3(0,L)} + e^{cT}\|\psi\|_{H^1(0,T)}.
		\end{equation*}
	\end{lem}
	
	Now letting $z = v_t$, one can see that $z$ solves the following model
	\begin{eqnarray}\label{targetfz}
	\begin{cases}
	iz_t + i\beta z_{xxx} +\alpha z_{xx} +i\delta z_x =0, x\in (0,L), t\in (0,T),\\
	z(0,t)=0, z(L,t)=0, z_x(L,t)=\psi^{\prime\prime}(t),\\
	z(x,0)=\tilde{\tilde{\phi}}(x),
	\end{cases}
	\end{eqnarray}
	where $\tilde{\tilde{\phi}} \doteq -\beta \tilde{\phi}^{\prime\prime\prime} +i\alpha \tilde{\phi}^{\prime\prime} -\delta \tilde{\phi}^\prime$. Let $\psi^{\prime\prime}(t) \in L^2(0,T)$ and $\tilde{\tilde{\phi}} \in L^2(0,L)$. Then from Lemma \ref{wp_lem_1} and Lemma \ref{wp_lem_rn}, $z$ satisfies
	\begin{equation}
		\|z\|_{X_T^0} \lesssim \|\tilde{\tilde{\phi}}\|_2 + e^{cT}\|\psi^{\prime\prime}\|_{L^2(0,T)}.
	\end{equation}
	Define $v(x,t) \doteq \tilde{\phi}(x) + \int_0^t z(x,\tau)d\tau$. If $(\tilde{\phi},\psi^\prime)$ satisfies the compatibility conditions, one can show that $v$ satisfies the following initial--boundary conditions:
	\begin{equation*}
		v(x,0) = \tilde{\phi}(x)
	\end{equation*}
	and
	\begin{equation*}
		v(0,t) = 0, \quad v(L,t) = 0, \quad v_x(L,t) = \phi^\prime(t).
	\end{equation*}
	Then one can also show that $v$ solves \eqref{targetfv}. Now defining $v = y_t$ and repeating the same analysis as we did through \eqref{targetfv}-\eqref{xt3-eq8}, one concludes the following lemma.
	\begin{lem} \label{lem_wp_pde_xt6}
		Let $T > 0$, $(\phi,\psi) \in H^6(0,L) \times H^2(0,T)$ satisfy the higher order compatibility conditions. Then \eqref{wp_pde_hreg} has a unique solution in $X_T^6$ with $y_{tt} \in X_T^0$ and it satisfies the following estimate,
		\begin{equation*}
		\|y\|_{X_T^6} + \|y_{tt}\|_{X_T^0} \lesssim \|\phi\|_{H^6(0,L)} + e^{cT}\|\psi\|_{H^2(0,T)}.
		\end{equation*}
	\end{lem}

	\section{Controller design} \label{ctrl}
	In this section, first we prove the existence of a smooth backstepping kernel. Then we state the result of the invertibility of the backstepping transformation with a bounded inverse. Next, we prove the global wellposedness and exponential stability results.
	
	\subsection{Backstepping kernel}
	Let us express the main equation in \eqref{kernel_pG} as
	\begin{equation*}
	G_{sst}=DG \doteq \frac{1}{3\beta}\left[\beta(3G_{tts}-G_{ttt})-i\alpha(G_{tt}-2G_{ts})-\delta G_t-rG\right].
	\end{equation*}
	Integrating the above expression in the first variable and using $G_s(0,t) = \frac{rt}{3\beta}$ we obtain
	\begin{eqnarray*}
		G_{st}(s,t)=\frac{r}{3\beta}+\int_0^s[DG](\xi,t)d\xi.
	\end{eqnarray*}
	Integrating once again in the first variable and using $G(0,t) = 0$ we get
	\begin{eqnarray*}
		G_{t}(s,t) = \frac{r}{3\beta}s+\int_0^s\int_0^\omega[DG](\xi,t)d\xi d\omega.
	\end{eqnarray*}
	Finally, integrating in the second variable and using $G(s,0) = 0$ we obtain that $G$ solves
	\begin{eqnarray} \label{kernel_int_pG}
	\label{GepsInt}
	G(s,t)=\frac{r}{3\beta}st+\int_0^t\int_0^s\int_0^\omega[DG](\xi,\eta)d\xi d\omega d\eta.
	\end{eqnarray}
	So the solution of the boundary value problem \eqref{kernel_pG} can be constructed by applying a successive approximation method to the integral equation \eqref{kernel_int_pG}.
	\begin{lem} \label{lemkernel}
		There exists a $C^\infty-$function $G$ such that $G$ solves the integral equation \eqref{kernel_int_pG}.
	\end{lem}
	\begin{proof}
		Let $P$ be defined by
		\begin{equation}\label{aP}
		(P f)(s,t) \doteq \int_0^t\int_0^s\int_0^\omega[Df](\xi,\eta)d\xi d\omega d\eta.
		\end{equation}
		Then we express \eqref{GepsInt} as
		\begin{equation}\label{GPG}
		G(s,t)=\frac{{r}}{3\beta}st+PG(s,t).
		\end{equation}
		Define $G^0\equiv 0,$ $\displaystyle G^1(s,t)=-\frac{{r}}{3\beta}st,$ and $G^{n+1}=G^1+P G^n$, $n\geq 1$. Then we have
		\begin{equation}\label{GPG_Cauchy}
		G^{n+1} - G^n = P (G^n - G^{n-1}), \quad n \geq 1.
		\end{equation}
		To prove the existence of a solution of \eqref{GPG}, it is enough to show that the sequence $(G^n)$ and its partial derivatives are Cauchy with respect to the supremum norm $\|\cdot\|_\infty$. To this end, define $H^0(s,t) = st$, $H^n = \frac{3\beta}{r}(G^{n+1} - G^n)$. Then by \eqref{GPG_Cauchy}, $H^{n+1} = PH^n$ and for $j > i$,
		\begin{equation}\label{aCauchy}
		G^j-G^i= \sum_{n=i}^{j-1}(G^{n+1}-G^{n})=\frac{r}{3\beta}\sum_{n=i}^{j-1}H^{n}.
		\end{equation}
		We see from \eqref{aCauchy} that the sequence $(G^n)$ (and its partial derivatives) is Cauchy with respect to the norm $\|\cdot\|_\infty$, which implies that $(G^n)$ is convergent and its limit solves \eqref{kernel_int_pG} if and only if the sequence $(H^n)$ (and its partial derivatives) is absolutely summable sequence with respect to the same norm.
		
		To show that $H^n$'s are absolutely summable, let us express $P$ as sum of six operators
		$$P= P_{1,-1}+P_{2,-2}+P_{2,-1}+P_{1,0}+P_{2,0}+P_{2,1},$$
		where
		
		$$P_{1,-1}f\doteq\int_0^t\int_0^s\int_0^\omega f_{tts}(\xi,\eta) d\xi d\omega d\eta, \quad P_{2,-2}f\doteq -\frac{1}{3}\int_0^t\int_0^s\int_0^\omega f_{ttt}(\xi,\eta) d\xi d\omega d\eta,$$
		$$P_{2,-1}f\doteq-\frac{i\alpha}{3\beta}\int_0^t\int_0^s\int_0^\omega f_{tt}(\xi,\eta) d\xi d\omega d\eta, \quad P_{1,0}f\doteq \frac{2 i\alpha}{3\beta}\int_0^t\int_0^s\int_0^\omega f_{ts}(\xi,\eta) d\xi d\omega d\eta,$$
		$$P_{2,0}f\doteq -\frac{\delta}{3\beta}\int_0^t\int_0^s\int_0^\omega f_{t}(\xi,\eta) d\xi d\omega d\eta, \quad P_{2,1}f\doteq -\frac{r}{3\beta}\int_0^t\int_0^s\int_0^\omega f(\xi,\eta) d\xi d\omega d\eta.$$
		Then
		\begin{equation}\label{aproduct}
			\begin{split}
		H^n=P^nH^0&=(P_{1,-1}+P_{2,-2}+P_{2,-1}+P_{1,0}+P_{2,0}+P_{2,1})^nst\\
		&=\sum_{r=1}^{6^n}R_{r,n}st,
			\end{split}
		\end{equation}
		where
		\begin{equation*}
		R_{r,n}\doteq P_{i_{r,n}, j_{r,n}}P_{i_{r,n-1}, j_{r,n-1}}\cdot\cdot\cdot P_{i_{r,1},j_{r,1}}, \quad i_{r,q} \in \{1,2\}, \quad j_{r,q} \in \{-2,-1,0,1\},
		\end{equation*}
		for $1\le q\le n$. Observe that for positive integers $m$ and nonnegative integers $k$,
		\begin{equation} \label{acm1}
		P_{1,-1}s^mt^k = c_{1,-1}s^{m+1}t^{k-1}, \quad c_{1,-1} =
		\begin{cases*}
		0, & $k \leq 0$, \\
		\frac{k}{m+1}, & else,
		\end{cases*}
		\end{equation}
		\begin{equation} \label{acm2}
		P_{2,-2}s^mt^k = c_{2,-2}s^{m+2}t^{k-2}, \quad c_{2,-2} =
		\begin{cases*}
		0, & $k \leq 1$, \\
		-\frac{k(k-1)}{3(m+1)(m+2)}, & else,
		\end{cases*}
		\end{equation}
		\begin{equation}
		P_{2,-1}s^mt^k = c_{2,-1}s^{m+2}t^{k-1}, \quad c_{2,-1} =
		\begin{cases*}
		0, & $k \leq 0$, \\
		-\frac{i\alpha k}{3\beta(m+1)(m+2)}, & else,
		\end{cases*}
		\end{equation}
		\begin{equation}
		P_{1,0}s^mt^k = c_{1,0}s^{m+1}t^{k}, \quad c_{1,0} = \frac{2i\alpha}{3\beta(m+1)},
		\end{equation}
		\begin{equation}
		P_{2,0}s^mt^k = c_{2,0}s^{m+2}t^{k}, \quad c_{2,0} = -\frac{\delta}{3\beta(m+1)(m+2)},
		\end{equation}
		\begin{equation} \label{ac1}
		P_{2,1}s^mt^k = c_{2,1}s^{m+2}t^{k+1}, \quad c_{2,1} = -\frac{r}{3\beta(m+1)(m+2)(k+1)}.
		\end{equation}
		Let $\sigma = \sigma(r) \equiv \sum_{q = 1}^n j_{r,q}$. Then from \eqref{acm1}-\eqref{ac1}, for eeach $n$ and $r$,
		\begin{equation}\label{amonomials}
		R_{r,n}st=
		\begin{cases}
		0, & \text { if } \sigma \leq-1,\\
		C_{r,n}s^\gamma t^{\sigma+1}, & \text { if } \sigma > -1,\\
		\end{cases}
		\end{equation}
		where $n+1\leq \gamma \leq 2n+1$ and $C_{r,n}$ is a constant which only depends on $r$ and $n$. Let $M=\max\{1,\frac{\alpha}{\beta},\frac{\delta}{\beta},\frac{r}{\beta}\}$. We claim that for each $r$ and $n$,
		\begin{equation}\label{aclaim}
		|C_{r,n}|\leq \frac{M^n}{(n+1)!(\sigma+1)!}.
		\end{equation}
		Taking $m=1$, $k=1$ in  \eqref{acm1}-\eqref{ac1} we see that \eqref{aclaim} holds for $n=1$. Suppose it holds for $n=\ell-1$ and for all $r \in \{1,2,\dotsc ,6^{\ell -1}\}$. Then for $n=\ell$ and $r^* \in \{1,2,\dotsc ,6^{\ell}\}$, using \eqref{acm1} and \eqref{amonomials}, we get
		\begin{equation*}
			R_{r^*,\ell}st=P_{i,j} R_{r,\ell-1}st= C_{r,\ell-1}P_{i,j} s^\gamma t^{\sigma+1}=C_{r,\ell-1}c_{i,j} s^{\gamma^*} t^{\sigma^*+1}
		\end{equation*}
		for some $i\in\{1,2\}$, $j\in\{-2,-1,0,1\}$ and $r \in \{1,2,.. ,6^{\ell -1}\}$, where $\gamma^*$ is either $\gamma+1$ or $\gamma+2$, $\sigma^*=\sigma +j$. By the induction assumption \eqref{aclaim},
		\begin{equation*}
		C_{r,\ell-1}\leq \frac{M^{\ell-1}}{\ell!(\sigma+1)!}.
		\end{equation*}
		Moreover using \eqref{acm1}-\eqref{ac1} and the fact that $\gamma\geq \ell$  we see that $|c_{i,j}|\leq M\frac{\sigma+1}{\ell+1}$ for $j=-1,-2$, $|c_{i,0}| <\frac{M}{\ell+1}$, and $|c_{i,1}|< \frac{M}{(\sigma+2)(\ell+1)}$. Hence for each $i\in \{1,2\}$ and $j\in \{-2,-1,0,1\}$ we obtain
		\begin{equation*}
			|C_{r^*,\ell}|= |C_{r,(\ell-1)}c_{i,j}| \leq \frac{M^{\ell}}{(\ell+1)!(\sigma+j+1)!}=\frac{M^{\ell}}{(\ell+1)!(\sigma^*+1)!},
		\end{equation*}
		which proves that the claim holds for $n=\ell$ as well.
		
		Using \eqref{aproduct}, \eqref{amonomials}, \eqref{aclaim} and the fact that $0\leq s, t \leq L$ in the triangle $\Delta_{s,t}$ we obtain
		\begin{equation}
		\label{Hnest0}\|H^n\|_{\infty}\leq \frac{6^n M^n L^{3n+2}}{(n+1)!}.
		\end{equation}
		This shows $H^n$ is absolutely summable. On the other hand since $H^n$ is a linear combination of $6^n$ monomials of the form $s^\gamma t^{\sigma+1}$ with $\gamma\leq 2n+1$ and $\sigma\leq n$,
		any partial derivative $\partial^a_s \partial^b_t H^n$ of $H^n$ will be absolutely less than
		\begin{equation}
		\label{Hnest}\displaystyle\frac{(2n+1)^a (n+1)^b 6^n M^nL^{3n+2-a-b}}{(n+1)!},
		\end{equation}
		which is a summable sequence.
	\end{proof}

	See Figure \ref{fig:p_kernel} for a graph and a contour plot of the kernel and Figure \ref{fig:p_ctrlgain} for the corresponding control gains for $L=\pi$, $\beta = 1$, $\alpha = 2$, $\delta = 8$ and $r = 0.05$.
	\begin{figure}[h]
		\centering
		\begin{subfigure}[b]{0.5\textwidth}
			\includegraphics[width=\textwidth]{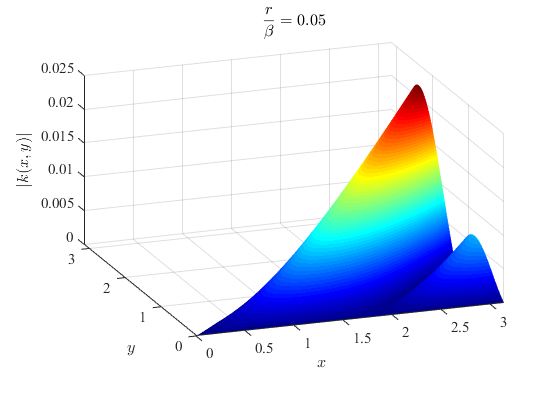}
			\label{fig:kernel_k_3d}
		\end{subfigure}
		~ 
		\begin{subfigure}[b]{0.5\textwidth}
			\includegraphics[width=\textwidth]{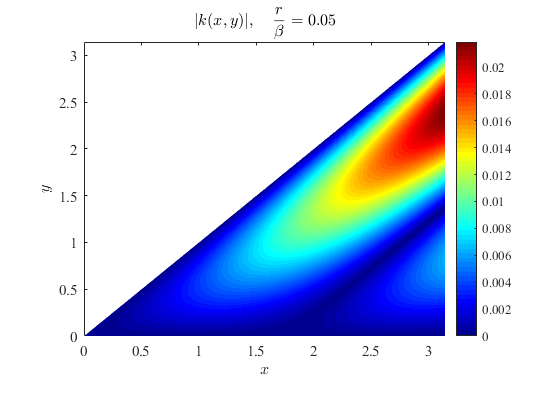}
			\label{fig:kernel_k_contour}
		\end{subfigure}
		\vspace*{-15mm}
		\caption{Backstepping kernel on $\Delta_{x,y}$ for $L=\pi$, $\beta = 1$, $\alpha = 2$, $\delta = 8$ and $r = 0.05$.}
		\label{fig:p_kernel}
	\end{figure}

	\begin{figure}[h]
		\centering
		\begin{subfigure}[b]{0.5\textwidth}
			\includegraphics[width=\textwidth]{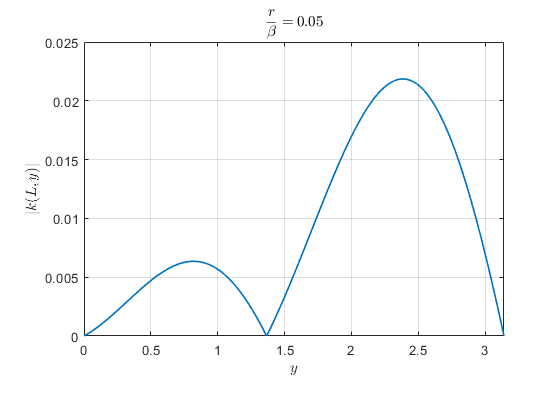}
			\label{fig:control_gain_g}
		\end{subfigure}
		~ 
		\begin{subfigure}[b]{0.5\textwidth}
			\includegraphics[width=\textwidth]{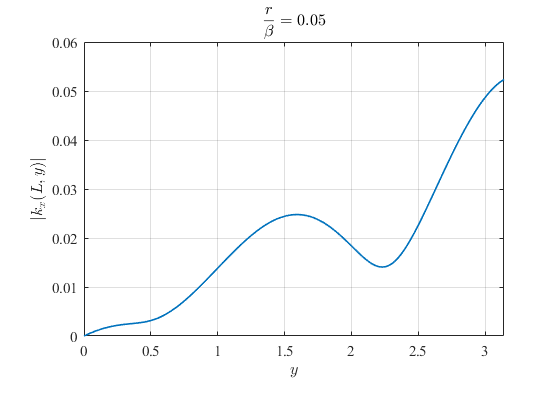}
			\label{fig:control_gain_h}
		\end{subfigure}
		\vspace*{-15mm}
		\caption{Control gains for the Dirichlet (left) and Neumann (right) boundary conditions. $L=\pi$, $\beta = 1$, $\alpha = 2$, $\delta = 8$ and $r = 0.05$.}
		\label{fig:p_ctrlgain}
	\end{figure}
	
	Next let $\eta = \eta(x,y)$ be a $C^\infty-$function defined on $\Delta_{x,y}$ and $\Upsilon_\eta : H^l(0,L) \to H^l(0,L)$, $l \geq 0$ be an integral operator defined by
	\begin{equation*}
		[\Upsilon_\eta \varphi](x) \doteq \int_0^x \eta(x,y) \varphi(y) dy.	
	\end{equation*}
	Then, we have the following lemma for the operator $I - \Upsilon_\eta$.

	\begin{lem} \label{inverselem}
		$I-\Upsilon_{\eta}$ is invertible with a bounded inverse from $H^l(0,L)\rightarrow H^l(0,L)$ ($l\ge 0$). Moreover,  $(I-\Upsilon_{\eta})^{-1}$ can be written as $I+\Phi$, where $\Phi$ is a bounded operator from $L^2(0,L)$ into $H^l(0,L)$ for $l=0,1,2$ and from $H^{l-2}(0,L)$ into $H^{l}(0,L)$ for $l> 2$.
	\end{lem}
	We omit the proof since it can be done as in \cite{Liu03,BatalOzsari2018-1}.
	
	\subsection{Wellposedness} \label{ctrl_wp}
	We first investigate the local and global wellposedness of the target model.
	Then, using Lemma \ref{lemkernel} and Lemma \ref{inverselem}, we deduce the wellposedness of the original plant \eqref{plant_lin}. To this end, let us consider the modified target model
	\begin{eqnarray} \label{wp_pde_2}
		\begin{cases}
			iw_t + i\beta w_{xxx} +\alpha w_{xx} +i\delta w_x  + ir w= i\beta k_y(x,0) w_x(0,t), x\in (0,L), t\in (0,T),\\
			w(0,t)= w(L,t)= w_x(L,t)=0,\\
			w(x,0)=w_0(x) \doteq u_0(x) - \int_0^x k(x,y)u(y,t)dy.
		\end{cases}
	\end{eqnarray}
	Consider the operator $A$ defined in \eqref{A_opr} with domain $D(A)$ defined in \eqref{A_dom}. Let us express \eqref{wp_pde_2} in the abstract operator theoretic form as
	\begin{equation*}
		\begin{cases}
			\dot y = Ay + Fy,\\
			y(0) = y_0,
		\end{cases}
	\end{equation*}
	where $F\varphi \doteq -ir \varphi + i\beta a(\cdot) \Gamma_0^1 \varphi$ and $\Gamma_0^1$ is the first order trace operator at the left end point. Operator $A$ generates strongly continuous semigroup of contractions, $\{S(t)\}_{t \geq 0}$, in $L^2(0,L)$ \cite{Ceballos05}. Define the operator
	\begin{equation} \label{wp_op}
		w = [\Psi z](t) \doteq S(t)w_0 + \int_0^t S(t - s)Fz(s) ds
	\end{equation}
	and the space
	\begin{equation}
		Y_T \doteq \{ z \in X_T^0 : z_x \in C([0,L];L^2(0,T)) \}
	\end{equation}
	endowed with the norm
	\begin{equation}
		\|z\|_{Y_T} \doteq \left(\|z\|_{C([0,T];L^2(0,L))}^2 + \|z\|_{L^2(0,T;H^1(0,L))}^2 + \|z_x\|_{C([0,L];L^2(0,T))}^2 \right)^{\frac{1}{2}}.
	\end{equation}
	We prove the following result.
	
	\begin{prop}[Local wellposedness] \label{ctrl_localwp} Let $T^\prime > 0$ and $w_0 \in L^2(0,L)$. Then, there exists $T \in (0,T^\prime)$ which is independent of size of $w_0$ such that \eqref{wp_pde_2} possesses a unique local solution $w \in Y_T$.
	\end{prop}

	\begin{proof}
		We first show that $\Psi$, defined by \eqref{wp_op} maps $Y_T$ into itself. To see this, first of all, we obtain from \eqref{wp_op} that
		\begin{align*}
			\|w\|_{Y_T} &= \|\Psi z\|_{Y_T} \leq \| S(t) w_0\|_{Y_T} + \left\|\int_0^t S(t-s)[Fz](s)ds \right\|_{Y_T}.
		\end{align*}
		By using Lemma \ref{wp_lem_1}, the first term at the right hand side of the above inequality can be estimated as
		\begin{equation} \label{into_1}
			\| S(t) w_0\|_{Y_T} \lesssim \sqrt{1 + T}\|w_0\|_2.
		\end{equation}
		Using Lemma \ref{wp_lem_2} and then applying Cauchy--Schwarz inequality, the second term can be estimated as
		\begin{equation}
		\label{into_2}
			\begin{split}
				\left\|\int_0^t S(t-s)Fz(s)ds \right\|_{Y_T} &\leq C \sqrt{1 + T} \|-rz + \beta k_y(\cdot,0) z_x(0,\cdot)\|_{L^1(0,T;L^2(0,L))} \\
				&\leq c_{\beta,r} \sqrt{T(1 + T)} (1 + \|k_y(\cdot,0)\|_2) \| z\|_{Y_T}.
			\end{split}
		\end{equation}
		Combining \eqref{into_1} and \eqref{into_2}, we see that $\Psi$ maps $Y_T$ into itself. To see that $\Psi$ is contraction on $Y_T$, let $z_1, z_2 \in Y_T$ and $w_1 = \Psi z_1$, $w_2 = \Psi z_2$. Using the similar arguments as above, we get
		\begin{equation*}
			\begin{split}
				\|w_2 - w_1 \|_{Y_T} &= \|\Psi z_2  - \Psi z_1\|_{Y_T} \\
				&\leq c_{\beta,r} \sqrt{T(1 + T)} (1 + \|k(\cdot,0)\|_2) \|z_2 - z_1\|_{Y_T}.
			\end{split}
		\end{equation*}
		In order for the map $\Psi$ to be a contraction, we choose $T \in (0,T^\prime)$ such that $0 < \sqrt{T(1 + T)} \leq \left(c_{\beta,r}(1 + \|k(\cdot,0)\|_2) \right)^{-1}$ which is independent of the size of the initial datum. This guarantees the existence of a unique local solution $w \in Y_T$.
	\end{proof}
		This proposition shows the existence of a maximal time, $T_{\max}$, of the existence of the solution $w \in Y_T$ for all $T < T_{\max}$. To prove that $w$ is global, it is enough to show that $\lim_{T \to T_{\max}^-} \|w\|_{Y_T} < \infty$.
		
	\begin{prop}[Global wellposedness] \label{ctrl_globalwp}
		Let $w_0 \in L^2(0,L)$. Then $w$ extends as a global solution in $Y_T$.
	\end{prop}
	\begin{proof}
		Taking $L^2-$inner product of the main equation of \eqref{wp_pde_2} by $2w$, taking the imaginary parts of both sides and applying several integration by parts together with imposing the boundary conditions, we derive
		\begin{equation} \label{gwp_1}
		\frac{d}{dt}\|w(\cdot,t)\|_2^2 + \beta|w_x(0,t)|^2 + 2r\left\|w(\cdot,t)\right\|_2^2 = 2\beta \Re \int_0^L k_y(x,0) w_x(0,t) \overline{w}(x,t) dx.
		\end{equation}
		By using $\epsilon-$Young's inequality, right hand side of \eqref{gwp_1} can be estimated as
		\begin{equation*}
			2\beta \Re \int_0^L k_y(x,0) w_x(0,t) \overline{w}(x,t) \leq \frac{\beta}{ 2\epsilon} \int_0^L |k_y(x,0)|^2 |w(x,t)|^2 dx + 2\epsilon\beta L |w_x(0,t)|^2
		\end{equation*}
		Choosing $\epsilon = \frac{1}{4L}$, \eqref{gwp_1} becomes
		\begin{equation*}
			\frac{d}{dt}\|w(\cdot,t)\|_2^2 + \frac{\beta}{2}|w_x(0,t)|^2 \leq 2(\beta L \|k_y(\cdot,0)\|_{\infty}^2 - r) \|w(\cdot,t)\|_2^2.
		\end{equation*}
		Now integrating the above inequality over $(0,t)$ yields
		\begin{equation} \label{gwp_est_0.5}
			2\|w(\cdot,t)\|_2^2 + \beta\int_0^t|w_x(0,\tau)|^2d\tau \leq 2\|w_0\|_2^2 + 4(\beta L \|k_y(\cdot,0)\|_{\infty}^2 - r) \int_0^t \|w(\cdot,\tau)\|_2^2 d\tau.
		\end{equation}
		Define $E(t) \doteq 2\|w(\cdot,t)\|_2^2 + \beta\int_0^t|w_x(0,\tau)|^2d\tau$. Then, from \eqref{gwp_est_0.5}
		\begin{equation*}
			E(t) \leq 2\|w_0\|_2^2 + 4\left| \beta L \|k_y(\cdot,0)\|_{\infty}^2 - r\right| \int_0^t E(\tau)d\tau.
 		\end{equation*}
 		Thanks to Gronwall's inequality,
 		\begin{equation} \label{gwp_est_1}
 			E(t) = 2\|w(\cdot,t)\|_2^2 + \beta\int_0^t|w_x(0,\tau)|^2d\tau \leq 2 \|w_0\|_2^2 e^{4\left(\beta L \|k_y(\cdot,0)\|_{\infty}^2 - r\right)t},
 		\end{equation}
 		for all $t \in [0,T]$. Passing to supremum on $[0,T]$ and then letting $T \to T_\text{max}^-$, we get
 		\begin{equation} \label{gwp_est_2}
 			\lim_{T \to T_\text{max}} \|w\|_{C([0,T];L^2(0,L))} \leq \|w_0\|_2 e^{2\left(\beta L \|k(\cdot,0)\|_{\infty}^2 -r\right)T_\text{max}} < \infty.
 		\end{equation}
 		Using
		\begin{equation}
			\underset{0 \leq t \leq T}{\sup} \|w(\cdot,t)\|_2^2 = \frac{1}{T} \|w\|_{L^2(0,T;L^2(0,L))}^2
		\end{equation}		 		
		and then letting $T \to T_\text{max}^-$, we also get
 		\begin{equation} \label{gwp_est_3}
 			\lim_{T \to T_\text{max}^-} \|w\|_{L^2(0,T;L^2(0,L))} \leq \sqrt{T_\text{max}} \|w_0\|_2 e^{2\left(\beta L \|k(\cdot,0)\|_{\infty}^2 -r\right)T_\text{max}}.
 		\end{equation}
 		Next, we multiply the main equation of \eqref{wp_pde_2} by $2x\overline{w}$,  integrate over $(0,t) \times (0,L)$, consider the imaginary parts and apply several integration by parts to get
 		\begin{equation} \label{wp_mult_2xy}
 			\begin{split}
 			&\int_0^L x |w(x,t)|^2dx + 3\beta \int_0^t \int_0^L |w_x(x,\tau)|^2 dx d\tau + 2r\int_0^L x |w(x,t)|^2 dx \\
 		 	=& \int_0^L x |w_0(x)|^2 dx + \delta \int_0^t \int_0^L |w(x,\tau)|^2 dx d\tau + 2\beta\int_0^t\int_0^L x k_y(x,0) w_x(0,\tau) \overline{w}(x,\tau) dxd\tau
 		 	\end{split}
 		\end{equation}
 		Thanks to Cauchy--Schwarz inequality, the last term at the right hand side of \eqref{wp_mult_2xy} can be estimated as
 		\begin{equation*}
 			\begin{split}
 				&2\beta\int_0^t\int_0^L x k_y(x,0)(x) w_x(0,\tau) \overline{y}(x,\tau) dx d\tau \\
 				\leq& 2\beta L \int_0^t\int_0^L |k_y(x,0)| |w_x(0,\tau)| |w(x,\tau)| dx d\tau \\
 				\leq& \beta L^2 \|k_y(\cdot,0)\|_\infty^2\int_0^t |w_x(0,\tau)|^2 d\tau + \beta L \int_0^t\int_0^L |w(x,\tau)|^2 dx d\tau.
 			\end{split}
 		\end{equation*}
 		Dropping the first and third terms at the left hand side of \eqref{wp_mult_2xy}, and using the above estimate, it follows that
 		\begin{equation*}
 			\begin{split}
	 			&\|w_x\|_{L^2(0,t;L^2(0,L))}^2 \\
	 			\leq& \frac{L}{3\beta} \|w_0\|_2^2 + \frac{\beta L + \delta}{3\beta} \int_0^t \int_0^L |w(x,\tau)|^2 dx d\tau + \frac{L^2 \|k_y(\cdot,0)\|_\infty^2}{3} \int_0^t |y_x(0,\tau)|^2d\tau \\
	 		\leq& \frac{L}{3\beta} \|w_0\|_2^2 + \left(\frac{\beta L + \delta}{6\beta}+\frac{L^2 \|k_y(\cdot,0)\|_\infty^2}{3\beta}\right) \int_0^t E(\tau)d\tau.
	 		\end{split}
 		\end{equation*}
 		Using \eqref{gwp_est_1} we get,
 		 \begin{equation} \label{gwp_est_4}
 		 	\begin{split}
 			&\lim_{T \to T_\text{max}} \|w_x\|_{L^2(0,T;L^2(0,L))} \\ \leq& \sqrt{\frac{L}{3\beta}} \|w_0\|_2 + \left(\sqrt{\frac{\beta L + \delta}{6\beta}}+\frac{L\|k_x(\cdot,0)\|_\infty}{\sqrt{3\beta}}\right) \sqrt{2 T_\text{max}}\|w_0\|_2 e^{2\left(\beta L \|k_y(\cdot,0)\|_{\infty}^2 -r\right)T_\text{max}}.
 			\end{split}
 		\end{equation}
 		Combining \eqref{gwp_est_3} and \eqref{gwp_est_4}, we deduce that
 		\begin{equation} \label{gwp_est_5}
 			\begin{split}
 			&\lim_{T \to T_\text{max}^-} \|w\|_{L^2(0,T;H^1(0,L))} \\
 			\leq& \sqrt{\frac{L}{3\beta}} \|w_0\|_2 + \left(1 + \sqrt{\frac{\beta L + \delta}{6\beta}}+\frac{L\|k_y(\cdot,0)\|_\infty}{\sqrt{3\beta}}\right) \sqrt{2 T_\text{max}}\|w_0\|_2 e^{2\left(\beta L \|k_y(\cdot,0)\|_{\infty}^2 -r\right)T_\text{max}} \\
 			<& \infty.
 			\end{split}
 		\end{equation}
 		
 		From Proposition \ref{ctrl_localwp}, $w$ is the fixed point of \eqref{wp_pde_2}, so it satisfies
 		\begin{equation*}
 			w = S(t)w_0 + \int_0^t S(t - \tau)[Fw](\tau)d\tau
 		\end{equation*}
 		for some $t \in (0,T')$.
 		From  Lemma \ref{wp_lem_1}-(iii) and \ref{wp_lem_3}-(iii), we now that
 		\begin{equation}\label{gwp_est_6}
 			\underset{x\in[0,L]}{\sup} \|\partial_x[S(t)w_0](x)\|_{L^2(0,T)} \lesssim \sqrt{T}\|w_0\|_2
 		\end{equation}
 		and
 		\begin{equation} \label{gwp_est_7}
 			\underset{x\in[0,L]}{\sup} \left\|\partial_x \left[\int_0^t S(t - \tau) [Fw](\tau) d\tau \right] (x)\right\|_{L^2(0,T)} \lesssim \sqrt{T}\int_0^T \|[Fw](\cdot,t)\|_2 dt
 		\end{equation}
 		holds. Using the definition of $Fw$, right hand side of \eqref{gwp_est_7} can be estimated as
 		\begin{equation}
 			\begin{split} \label{gwp_est_8}
 			\sqrt{T}\int_0^T \|[Fw](\cdot,t)\|_{L^2(0,L)} dt &\leq \sqrt{T}r\int_0^T \|w(\cdot,t)\|_2 dt + \sqrt{T} \beta\|k(\cdot,0)\|_2 \int_0^T|w_x(0,t)|dt \\
 			&\leq \frac{\sqrt{T} r}{2} \int_0^T \sqrt{E(t)}dt + \sqrt{T} \|k(\cdot,0)\|_2 \sqrt{E(t)}	\\
 			&\leq \left(\frac{T \sqrt{T} r}{2} +\sqrt{T} \beta \|k(\cdot,0)\|_2^2\right) \sqrt{2} \|w_0\|_2 e^{2\left(\beta L \|k_y(\cdot,0)\|_{\infty}^2 - r\right)T}.
 			\end{split}
 		\end{equation}
 		Finally using \eqref{gwp_est_6}-\eqref{gwp_est_8}
 		\begin{equation*}
 			\begin{split}
 				&\lim_{T \to T_\text{max}^-} \|w_x\|_{C([0,L];L^2(0,T))} \\
 			\lesssim& \sqrt{T_{\text{max}}}\|w_0\|_2 + \left(\frac{T \sqrt{T} r}{2} +\sqrt{T} \beta \|k(\cdot,0)\|_2^2\right) \sqrt{2} \|w_0\|_2 e^{2\left(\beta L \|k_y(\cdot,0)\|_{\infty}^2 - r\right)T} < \infty.
 			\end{split}
 		\end{equation*}
 		This completes the proof.
	\end{proof}
	
	Choosing $w_0 \in H^3(0,L)$ that satisfies compatibility conditions, the global solution enjoys higher order regularity given by the following proposition.
	\begin{prop}[Regularity]
		Let $w_0 \in H^3(0,L)$ satisfy the compatibility conditions. Then $w \in Y_T^3$.
	\end{prop}

	\begin{proof}
		Let $v = w_t$ and consider the following problem
		\begin{eqnarray*}
			\begin{cases}
				i{v}_t + i\beta {v}_{xxx} +\alpha {v}_{xx} +i\delta {v}_x  + ir {v}= i\beta k_y(x,0) {v}_x(0,t) , x\in (0,L), t\in (0,T),\\
				{v}(0,t)= {v}(L,t)= {v}_x(L,t)=0,\\
				{v}(x,0)={v}_0(x),
			\end{cases}
		\end{eqnarray*}
		where ${v}_0(x) \doteq -\beta {w}_0^{\prime\prime\prime}(x) + i\alpha {w}_0^{\prime\prime}(x) -\delta {w}_0^\prime(x) - r{w}_0(x) + \beta k_y(x,0) {w}_0^\prime(0)$. For a given $v_0 \in L^2(0,L)$, we know from Proposition \ref{ctrl_localwp} that $v \in Y_T^0$. Set ${w}(x,t) = {w}_0(x) + \int_0^t {v}(x,\tau)d\tau$. Under the compatibility conditions, one can show that $w$ solves \eqref{wp_pde_2}. From the main equation of \eqref{wp_pde_2}, we have
		\begin{equation*}
		i\beta w_{xxx}(x,t) = (-iv - \alpha w_{xx} - i\delta w_x - irw)(x,t) + i\beta k_y(x,0)w_x(0,t).
		\end{equation*}
		Observe that $w_x(0,t) = - \int_0^L w_{xx}(x,t)dx$. Using this in the above expression and then taking $L^2-$norms of both sides with respect to $x$, we get
		\begin{equation*}
		\beta^2 \|w_{xxx}(\cdot,t)\|_2^2 \leq \|v(\cdot,t)\|_2^2 + \left(\alpha^2 + \beta^2\|k_y(\cdot,0)\|_2^2\right)\|w_{xx}(\cdot,t)\|_2^2 + \delta^2 \|w_{x}(\cdot,t)\|_2^2
		+ r^2\|w(\cdot,t)\|_2^2.
		\end{equation*}
		Similar work as we did on \eqref{xt3-eq3}-\eqref{xt3-eq4} yields
		\begin{equation*}
		\|w_{xxx}(\cdot,t)\|_2^2 \lesssim \|v(\cdot,t) \|_2^2 + \|w(\cdot,t) \|_2^2.
		\end{equation*}
		Taking supremum on both sides, we obtain $w \in C([0,T];H^3(0,L))$. Next, we differentiate the main equation of \eqref{wp_pde_2} with respect to $x$ and take $L^2-$norm of both sides with respect to $x$ to get
		\begin{align*}
		&\beta^2 \|w_{xxxx}(\cdot,t)\|_2^2 \\ \leq& \|v_x(\cdot,t)\|_2^2 + \alpha^2\|w_{xxx}(\cdot,t)\|_2^2
		+ \left(\delta^2+ \beta^2\|k_{yx}(\cdot,0)\|_2^2\right) \|w_{xx}(\cdot,t)\|_2^2 + r^2\|w_x(\cdot,t)\|_2^2.
		\end{align*}
		Proceeding as in \eqref{xt3-eq6}-\eqref{xt3-eq7}, we get
		\begin{equation*}
		\|w_{xxxx}(\cdot,t)\|_2^2 \lesssim \|v_x(\cdot,t) \|_2^2 + \|w_x(\cdot,t) \|_2^2.
		\end{equation*}
		Now the right hand side belongs to $L^2(0,T)$, so $w$ belongs to $L^2(0,T;H^4(0,L))$. Combining with the previous result, we deduce that $w \in X_T^3$ if $w_0 \in H^3(0,L)$.
	\end{proof}

	Now the first part of Theorem \ref{ctrl_thm} follows from the fact that backstepping kernel is a smooth function over a compact set and backstepping transformation is invertible on $L^2(0,L)$ and $H^3(0,L)$.
	
	\subsection{Stability} \label{ctrl_stab}
	In this part, we obtain exponential stability for the original plant. This will be done by first obtaining the exponential stability result for the modified target model \eqref{wp_pde_2}. Thanks to the invertibility of the backstepping transformation, this result will imply the exponential decay of  solutions of the original plant.
	\begin{prop} \label{ctrl_tar_stab}
		Let $\beta > 0$, $\alpha, \delta \in \mathbb{R}$, $k$ be a smooth backstepping kernel solving \eqref{kernel_pk}. Then for sufficiently small $r> 0$, there exists $\lambda = \beta\left(\frac{r}{\beta} -  \frac{\|k_y(\cdot,0;r)\|_2^2}{2}\right) > 0$ such that solution, $w$, of \eqref{p_tar_lin} satisfies the following decay estimate
		\begin{equation*}
			\|w(\cdot,t)\|_2 \leq \|w_0\|_2 e^{-\lambda t}
		\end{equation*}
		for $t \geq 0$.
	\end{prop}
	\begin{proof}	
	Taking the $L^2-$inner product of the main equation of \eqref{p_tar_lin} by $2w$ and proceeding as in \eqref{lin1iden2}-\eqref{lin1iden5}, we get
	\begin{equation} \label{stab_1}
		\frac{d}{dt}\left\|w(\cdot,t)\right\|_2^2 + {\beta}|w_x(0,t)|^2 + 2r\left\|w(\cdot,t)\right\|_2^2
		= 2 \beta \Re\left(w_x(0,t) \int_0^L k_y(x,0) \overline{w}(x,t) dx\right).
	\end{equation}
	Using $\epsilon-$Young's inequality and then the Cauchy--Schwarz inequality, the term at the right hand side can be estimated as
	\begin{equation} \label{rhs_est}
		\begin{split}
		2\beta \Re \int_0^L k_y(x,0) w_x(0,t) \overline{w}(x,t)dx
		&\leq 2\beta \left(\frac{1}{4\epsilon} |w_x(0,t)|^2 + \epsilon\left(\int_0^L |k_y(x,0)| |\overline{w}(x,t)|dx\right)^2 \right) \\
		&\leq \frac{\beta}{2\epsilon} |w_x(0,t)|^2 + 2\beta\epsilon \|k_y(\cdot,0)\|_2^2 \|w(\cdot,t)\|_2^2.
		\end{split}
	\end{equation}
	Combining this estimate with \eqref{stab_1} and choosing $\epsilon = \frac{1}{2}$, we get
	\begin{equation*}
	\frac{d}{dt}\|w(\cdot,t)\|_2^2 + 2\beta\left(\frac{r}{\beta} - \frac{ \|k_y(\cdot,0)\|_2^2}{2}\right) \|w(\cdot,t)\|_2^2 \leq 0,
	\end{equation*}
	which implies
	\begin{equation} \label{w_est}
	\|w(\cdot,t)\|_2 \lesssim \|w_0\|_2 e^{-\lambda t}, \quad \lambda \doteq \beta\left(\frac{r}{\beta} -  \frac{\|k_y(\cdot,0)\|_2^2}{2}\right).
	\end{equation}
	Next we prove that $\lambda > 0$ to show that \eqref{w_est} is indeed a decay estimate. Define $\tilde{\alpha} = \frac{\alpha}{\beta}$, $\tilde{\delta} = \frac{\delta}{\beta}$, $\tilde{r} = \frac{r}{\beta}$ and let $M=\max\{1,\tilde{\alpha},\tilde{\delta},\tilde{r}\}$. Differentiating \eqref{aCauchy} with respect to $t$, taking $i = 0$ and passing to limit as $n \to \infty$, we obtain
	\begin{equation} \label{sum_Gt}
		G_t(s,t) = \frac{\tilde{r}}{3} \sum_{n=0}^\infty H_t^n(s,t).
	\end{equation}
	Then, considering $\tilde{r} < 1$, we first see that $M$ is independent of $\tilde{r}$. Using this, we get from \eqref{Hnest} that the term inside the summation \eqref{sum_Gt} is absolutely less than some constant $c_{L,\tilde{\alpha},\tilde{\delta}}$ which is independent of $\tilde{r}$. So from \eqref{sum_Gt}, we get
	\begin{equation*}
	\|G_t\|_{L^\infty(\Delta_{s,t})} \leq \frac{\tilde{r}c_{L,\tilde{\alpha},\tilde{\delta}}}{3},
	\end{equation*}
	and therefore we have
	\begin{equation*}
	\|k_y(\cdot,0)\|_2^2 \leq L \|k_y(\cdot,0)\|_\infty^2 \leq L\|G_t(\cdot,0)\|_\infty^2  \leq L\|G_t\|_{L^\infty(\Delta_{s,t})}^2 \leq \frac{L\tilde{r}^2c_{L,\tilde{\alpha},\tilde{\delta}}^2}{9}.
	\end{equation*}
	Using this estimate, we get
	\begin{equation*}
	\lambda  =  2\beta\left(\tilde{r} -  \frac{\|k_y(\cdot,0)\|_2^2}{2}\right) = 2\beta\tilde{r}^2\left(\frac{1}{\tilde{r}} -  \frac{\|k_y(\cdot,0)\|_2^2}{2\tilde{r}^2}\right) \geq 2\beta\tilde{r}^2\left(\frac{1}{\tilde{r}} -  \frac{Lc_{L,\tilde{\alpha},\tilde{\delta}}^2}{18} \right),
	\end{equation*}
	which remains positive for sufficiently small $r$.
	\end{proof}
	Now using \eqref{p_bt} and the fact that $k = k(x,y)$ is a smooth function on a compact set $\Delta_{x,y}$, we have
	\begin{equation} \label{w0_est}
	\|w_0\|_2 \leq \left(1 + \|k(\cdot,\cdot)\|_{L^2(\Delta_{x,y})}\right) \|u_0\|_2.
	\end{equation}
	Moreover, using the invertibility of the backstepping transformation given by Lemma \ref{inverselem}, we have
	\begin{equation} \label{u_est}
	\|u(\cdot,t)\|_2 \leq \|I - \Upsilon_{k}\|_{2 \to 2} \|w(\cdot,t)\|_2.
	\end{equation}
	Combining \eqref{w0_est} and \eqref{u_est}, we deduce
	\begin{equation*}
	\|u(\cdot,t)\|_2 \leq \|I - \Upsilon_{k}\|_{2 \to 2}  \left(1 + \|k(\cdot,\cdot)\|_{L^2(\Delta_{x,y})}\right) \|u_0\|_2 e^{-\lambda t}.
	\end{equation*}
	So we conclude the proof of the second part of Therem \ref{ctrl_thm}.
	
	Table \ref{table_decay} below shows some values of $r$ and corresponding decay rates $\lambda$. Results are obtained by choosing $\beta = 1$, $\alpha = 2$, $\delta = 8$, domain length $L = \pi$ and $N = 1001$ spatial node points.
	\begin{table}[h!]
		\centering
		{\tabulinesep=1.2mm
			\begin{tabu} {l c}
				\hline\hline
				$r$ & $\lambda = \beta\left(\displaystyle\frac{r}{\beta} -  \frac{\|k_y(\cdot,0;r)\|_2^2}{2}\right)$ \\
				\hline
				$0.001$ & $0.001981$ \\
				$0.01$  & $0.018054$    \\
				$0.02$  & $0.032221$  \\
				$0.03$  & $0.042507$  \\
				$0.04$  & $0.048918$  \\
				$0.05$  & $0.051463$  \\
				$0.1$   & $0.006407$  \\
				$0.11$  & $-0.014113$ \\
				$0.5$   & $-3.729586$ \\
				$1$   & $-16.379897$ \\
				\hline\hline
			\end{tabu}
			\caption{Some numerical values for the decay rate $\lambda$ corresponding to various values of $r$.} \label{table_decay}
		}
	\end{table}

	\section{Observer design} \label{obs}
	In this section, our aim is to prove the wellposedness and exponential stability of the plant--observer--error system.
	
	\subsection{Wellposedness} \label{obs_wp}
	We start by the wellposedness analysis of the error model \eqref{err_lin}. To this end, we first study the target error model given by \eqref{err_tar_lin} and then use the invertibility of the transformation \eqref{bt_p} and deduce that same results also hold for \eqref{err_lin}. To see that \eqref{bt_p} is invertible with a bounded inverse, we change variables as $s = L - y$ and $t = L - x$ on \eqref{kernel_p}, and obtain that $p = p(x,y)$ solves \eqref{kernel_p} if and only if $p(x,y) \equiv H(s,t)$ solves
	\begin{eqnarray*}
		\begin{cases}
			\beta(H_{sss}+H_{ttt})-i\alpha(H_{ss}-H_{tt})+\delta(H_{s}+H_{t})-rH=0, \\
			H(s,s)= H(s,0)=0,\\
			H_{s}(s,s)=-\frac{rs}{3\beta}.
		\end{cases}
	\end{eqnarray*}
	Observe that this model is exactly the same model given in \eqref{kernel_pk} except that $r$ is replaced by $-r$. Therefore, we obtain the following relation
	\begin{equation} \label{p_eq_k}
	p(x,y) = H(s,t) = k(s,t;-r) = k(L-y,L-x;-r),
	\end{equation}
	where $k$ solves \eqref{kernel_pk}. Consequently existence of smooth kernel $p = p(x,y)$ is guaranteed and \eqref{bt_p} is invertible with a bounded inverse. See Figure \ref{fig:p_kernel_p} for a graph and a contour plot of $p(x,y)$ for $L=\pi$, $\beta = 1$, $\alpha = 2$, $\delta = 8$ and $r = 0.05$.
	\begin{figure}[h]
		\centering
		\begin{subfigure}[b]{0.5\textwidth}
			\includegraphics[width=\textwidth]{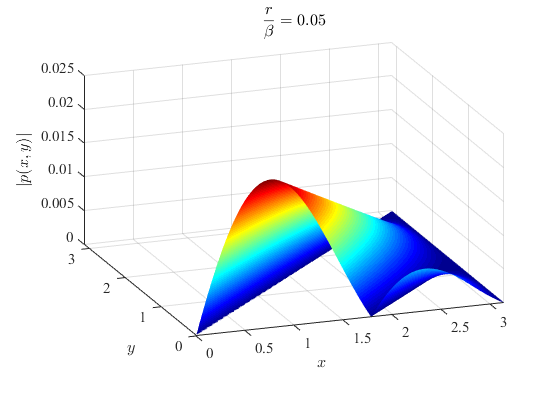}
			\label{fig:kernel_p_3d}
		\end{subfigure}
		~ 
		\begin{subfigure}[b]{0.5\textwidth}
			\includegraphics[width=\textwidth]{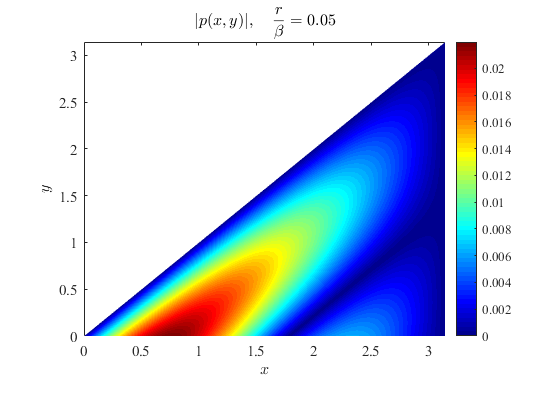}
			\label{fig:kernel_p_contour}
		\end{subfigure}
		\vspace*{-15mm}
		\caption{$p(x,y)$ defined on $\Delta_{x,y}$ for $L=\pi$, $\beta = 1$, $\alpha = 2$, $\delta = 8$ and $r = 0.05$.}
		\label{fig:p_kernel_p}
	\end{figure}

	\subsubsection{Error model} \label{wp_tar_err} Let us prove the wellposedness of the target error model. To this end, let us first consider the following model
	\begin{eqnarray}
		\begin{cases}
			i\tilde{w}_t + i\beta \tilde{w}_{xxx} +\alpha \tilde{w}_{xx} +i\delta \tilde{w}_x + ir\tilde{w} = 0, x\in (0,L), t\in (0,T),\\
			\tilde{w}(0,t)= \tilde{w}(L,t)= 0, \tilde{w}_x(L,t)= \tilde{\psi}(t),\\
			\tilde{w}(x,0)= \tilde{w}_0(x).
		\end{cases}
	\end{eqnarray}
	Note that the function $y$, defined by $y(\cdot,t) \doteq e^{rt} \tilde{w}(\cdot,t)$ together with the initial and boundary conditions $\tilde{w}_0$ and $\psi(t) \doteq e^{rt} \tilde{\psi}(t)$, satisfies the results obtained in Lemma \ref{wp_lem_1} and Lemma \ref{wp_lem_3}. Thus, for given $\tilde{w}_0 \in H^6(0,L)$, $\psi \in H^2(0,T)$ satisfying the higher order compability conditions, Lemma \ref{lem_wp_pde_xt6} implies that
	\begin{equation}\label{estlem25}
		\|y\|_{X_T^6} + \|y_{tt}\|_{X_T^0} \lesssim \|\tilde{w}_0\|_{H^6(0,L)} + e^{cT}\|\psi\|_{H^2(0,T)}.
	\end{equation}
	Notice that the original boundary condition of the problem \eqref{err_tar_lin} is of feedback type, given by
	\begin{equation*}
		\psi(t) = \psi(\tilde{w})(t) = \int_0^L p_x(L,y) \tilde{w}(y,t) dy.
	\end{equation*}
	We will treat the wellposedness of the target error model by using a fixed point argument. To this end, let us define the Banach space $Q_T \equiv \left\{\tilde{w} \in X_T^6: \tilde{w}_{tt} \in X_T^0 \right\}$ and its complete metric subspace $\tilde{Q}_T \equiv \left\{\tilde{w} \in Q_T : \tilde{w}(\cdot,0) = \tilde{w}_0(\cdot) \right\}$ equipped with the metric induced by the norm associated with $Q_T$. Since $p = p(x,y)$ is a smooth solution of \eqref{kernel_p}, for a given $\tilde{w}^* \in \tilde{Q}_T$, we have
	\begin{align*}
		\|\psi(\tilde{w}^*)\|_{H^2(0,T)} &= \left\|\int_0^L p_x(L,y) \tilde{w}^*(y,\cdot) dy\right\|_{H^2(0,T)} \\
		&\leq \sqrt{T} \|p_x(L,\cdot)\|_2 \sum_{j=0}^2 \|\partial_t^j \tilde{w}^*\|_{X_T^0} < \infty.
	\end{align*}
	Thus by the Lemma \ref{lem_wp_pde_xt6}, for $\psi(\tilde{w}^*)(t) \in H^2(0,T)$, the problem \eqref{wp_pde} with $f \equiv 0$ has a unique solution. This naturally defines a map $\Psi :\tilde{Q}_T \to \tilde{Q}_T$, $\Psi\tilde{w}^* = \tilde{w}$. Now let $\tilde{w}_1, \tilde{w}_2 \in \tilde{Q}_T$. Using the estimates \eqref{estlem25}, we get
	\begin{align*}
	\|\Psi\tilde{w}_1 - \Psi\tilde{w}_2\|_{\tilde{Q}_T} &\leq C \|\psi(\tilde{w}_1) - \psi(\tilde{w}_2)\|_{H^2(0,T)} \\
	&\leq C\sqrt{T} \|\tilde{w}_1 - \tilde{w}_2\|_{\tilde{Q}_T}.
	\end{align*}
	For sufficiently small $T$, we can guarantee that the mapping $\Psi :\tilde{Q}_T \to \tilde{Q}_T$ is contraction. Thanks to the Banach fixed point theorem, this yields the existence of a unique local solution of the target error system. As we show in Proposition \ref{tar_err_stab} in the following section, the local solution remains uniformly bounded in time. This yields the unique global solution and we have the following proposition.
	
	\begin{prop}\label{tar_err_wp}
		Let $T, \beta > 0$, $\alpha, \delta \in \mathbb{R}$, $p$ be a smooth backstepping kernel solving \eqref{kernel_p} and $(\tilde{w}_0,\psi) \in H^6(0,L) \times H^2(0,T)$ satisfies higher order compatibility conditions where $\psi = \psi(\tilde{w})  = \int_0^L p_x(L,y) \tilde{w}(y,t)dy$. Then \eqref{err_tar_lin} admits a unique global solution $\tilde{w} \in X_T^6$.
	\end{prop}
	Thanks to the bounded invertibility of the backstepping transformation \eqref{bt_p}, we obtain under the same assumptions that $\tilde{u} \in X_T^6$.
	
	\subsubsection{Observer model}
	Consider the target observer model \ref{tar_p_obs_lin}
	\begin{eqnarray} \label{tar_obs_pde_wp}
		\begin{cases}
		i\hat{w}_t + i\beta \hat{w}_{xxx} +\alpha \hat{w}_{xx} +i\delta \hat{w}_x  + ir 	\hat{w}= i\beta k_y(x,0) \hat{w}_x(0,t) \\
		+ f(x,t), \quad x\in (0,L), t\in (0,T),\\
		\hat{w}(0,t)= \hat{w}(L,t)= \hat{w}_x(L,t)=0,\\
		\hat{w}(x,0)=\hat{w}_0(x).
		\end{cases}
	\end{eqnarray}
	where $f(x,t) \doteq [(I - \Upsilon_{k})p_1](x) \tilde{w}_{x}(0,t) + [(I - \Upsilon_{k})p_2](x) \tilde{w}_{xx}(0,t)$. Recall that the backstepping transformation \eqref{bt_p} transforms \eqref{obs_lin} to \eqref{tar_obs_pde_wp} if $p_1, p_2$ are chosen such that $p_1(x) i\beta p_y(x,0) - \alpha p(x,0)$ and $p_2(x) = -i\beta p(x,0)$ where $p$ is the backstepping kernel that solves \eqref{kernel_p}. An example for the real and imaginary parts of the observer gains for a problem defined on $[0,\pi]$ and  the coefficients $\beta = 1, \alpha = 2, \delta = 8, r = 0.05$ are given in Figure \ref{fig:obs_gain}.
	\begin{figure}[h]
		\centering
		\begin{subfigure}[b]{0.5\textwidth}
			\includegraphics[width=\textwidth]{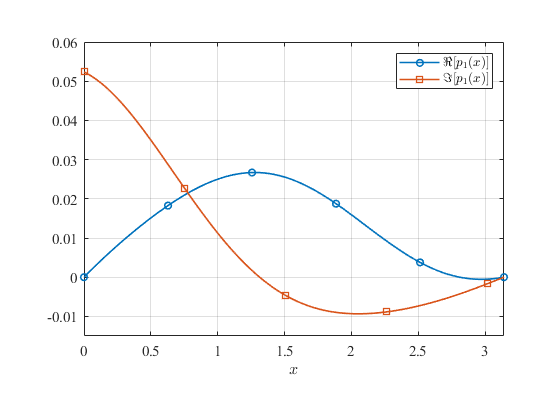}
			\label{fig:p1}
		\end{subfigure}
		~ 
		\begin{subfigure}[b]{0.5\textwidth}
			\includegraphics[width=\textwidth]{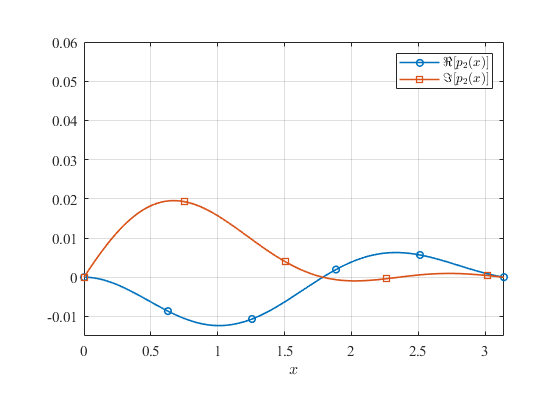}
			\label{fig:p2}
		\end{subfigure}
		\vspace*{-15mm}
		\caption{Observer gains for $L=\pi$, $\beta = 1$, $\alpha = 2$, $\delta = 8$ and $r = 0.05$.}
		\label{fig:obs_gain}
	\end{figure}

	For a given $\hat{w}_0 \in L^2(0,L)$, let us first show that $\hat{w} \in X_T^0$. Thanks to Proposition \ref{ctrl_globalwp}, given $\hat{w}_0 \in L^2(0,L)$, we know that solution, $\hat{w}$, of \ref{tar_obs_pde_wp} with $f\equiv 0$ belongs to the space $X_T^0$. Let us express it as $\hat{w}_i(x,t) \doteq W(t) \hat{w}_0(x)$ and now consider the problem where $\hat{w}_0 \equiv 0$. Let us express its solution as
	\begin{equation*}
	\hat{w}_f(x,t) \doteq \int_0^t W(t-\tau)f(x,\tau)d\tau.
	\end{equation*}
	Recall that $k = k(x,y)$, $p_1(x) = -i\beta p_y(x,0) + \alpha p(x,0)$ and $p_2(x) = i\beta p(x,0)$ are smooth functions. Also, we will see in Proposition \ref{tar_err_stab}-(ii) below that, if $\tilde{w}_0 \in H^3(0,L)$, then $\tilde{w}_x(0,t), \tilde{w}_{xx}(0,t) \in L^1(0,T)$. This implies that $f \in L^1(0,T;H^\infty(0,L))$. Now to see that $\hat{w}_f \in X_T^0$, first observe that
	\begin{equation}\label{obs_nonhom}
	\|\hat{w}_f(\cdot,t)\|_2 \leq \int_0^t \|W(t - \tau) f(\cdot,\tau) \|_2 d\tau \leq \int_0^t \|f(\cdot,\tau)\|_2 d\tau = \|f\|_{L^1(0,T;L^2(0,L))}.
	\end{equation}
	Taking supremum in $t \in [0,T]$ yields
	\begin{equation*}
	\|\hat{w}_f\|_{C([0,T];L^2(0,L))} \leq \|f\|_{L^1(0,T;L^2(0,L))} \leq \|f\|_{L^1(0,T;H^\infty(0,L))}.
	\end{equation*}
	Following from \eqref{obs_nonhom}, we also have
	\begin{equation*}
	\|\hat{w}_f\|_{L^2(0,T;L^2(0,L))} \leq \sqrt{T}\|f\|_{L^1(0,T;L^2(0,L))} \leq \sqrt{T} \|f\|_{L^1(0,T;H^\infty(0,L))}.
	\end{equation*}
	Using similar arguments, one can get
	\begin{equation*}
	\|\partial_x\hat{w}_f\|_{L^2(0,T;L^2(0,L))} \leq \sqrt{T}\|f\|_{L^1(0,T;H^\infty(0,L))}.
	\end{equation*}
	Consequently, $\hat{w}_f \in X_T^0$. As a conclusion given $\hat{w}_0 \in L^2(0,L)$, $\tilde{w}_0 \in H^3(0,L)$ and $f\in L^1(0,T;H^\infty(0,L))$, we obtain that $\hat{w} \in X_T^0$.
	
	Next, we show that $\hat{w} \in X_T^3$. To this end, we set $\tilde{z} = \tilde{w}_t$. Then $\tilde{z}$ satisfies
	\begin{equation*}
	\begin{cases}
	i\tilde{z}_t + i \beta\tilde{z}_{xxx} + \alpha \tilde{z}_{xx} + i\delta \tilde{z}_x + ir \tilde{z} = 0, \quad x\in(0,L), t \in (0,T), \\
	\tilde{z}(0,t) = \tilde{z}(L,t) = 0, \tilde{z}_x(L,t) = \psi^\prime(\tilde{z}), \\
	\tilde{z}(x,0) = \tilde{z}_0(x),
	\end{cases}
	\end{equation*}
	where $\tilde{z}_0 \doteq -\beta \tilde{w}^{\prime\prime\prime}_0 + i\alpha \tilde{w}_0^{\prime\prime} - \delta \tilde{w}_0^\prime - r\tilde{w}_0$.
	Applying the arguments in Section \ref{wp_tar_err}, we see that if $(\tilde{z}_0,\psi^\prime) \in H^3(0,L) \times H^1(0,T)$ satisfies compatibility conditions, then $z \in X_T^3$. Moreover, thanks to the Proposition \ref{tar_err_stab}-(ii) below, we have $\tilde{z}_x(0,t), \tilde{z}_{xx}(0,t) \in L^1(0,T)$. This implies by using $\tilde{z} = \tilde{w}_t$ that, $\tilde{w}_{xt}(0,t), \tilde{w}_{xxt}(0,t) \in L^1(0,T)$ where $\tilde{w}$ solves
	target error model satisfying $(\tilde{w}_0,\psi) \in H^6(0,L) \times H^2(0,T)$ higher order compatibility. Thus $f \in W^{1,1}(0,T;H^\infty(0,L))$.

	Now let us set $\hat{v} = \hat{w}_t$. Then $\hat{v}$ solves
	\begin{eqnarray*}
		\begin{cases}
			i\hat{v}_t + i\beta \hat{v}_{xxx} +\alpha \hat{v}_{xx} +i\delta \hat{v}_x  + ir 	\hat{v}= i\beta k_y(x,0) \hat{v}_x(0,t)
			\\+ f_t(x,t),\quad x\in (0,L), t\in (0,T),\\
			\hat{v}(0,t)= \hat{v}(L,t)= \hat{v}_x(L,t)=0,\\
			\hat{v}(x,0)=\hat{v}_0(x),
		\end{cases}
	\end{eqnarray*}
	where
	\begin{equation*}
		\hat{v}_0(x) \doteq -\beta \hat{w}_0^{\prime\prime\prime}(x) + i\alpha \hat{w}_0^{\prime\prime}(x) -\delta \hat{w}_0^\prime(x) - r\hat{w}_0(x) + \beta k_y(x,0) \hat{w}_0^\prime(0) -i f_t(x,0).
	\end{equation*}
	Assume that $\hat{v}_0 \in L^2(0,L)$. Then, from the above study, we deduce that $\hat{v} \in X_T^0$. If $\hat{w}_0$ satisfies the compatibility conditions, then we can also show that $\hat{w}$, defined by  $\hat{w}(x,t) = \hat{w}_0(x) + \int_0^t \hat{v}(x,\tau)d\tau$ solves \eqref{tar_obs_pde_wp}. Now from the main equation of \eqref{tar_obs_pde_wp}, we have
	\begin{equation*}
		i\beta \hat{w}_{xxx}(x,t) = (-i\hat{v} - \alpha \hat{w}_{xx} - i\delta \hat{w}_x - ir\hat{w})(x,t) + i\beta k_y(x,0)\hat{w}_x(0,t) + f(x,t).
	\end{equation*}
	Using $\hat{w}_x(0,t) = - \int_0^L \hat{w}_{xx}(x,t)dx$ and taking $L^2-$norms of both side we get
	\begin{multline*}
		\beta^2 \|\hat{w}_{xxx}(\cdot,t)\|_2^2 \leq \|\hat{v}(\cdot,t)\|_2^2 + \left(\alpha^2 + \beta^2 \|k_y(\cdot,0)\|_2^2\right)\|\hat{w}_{xx}(\cdot,t)\|_2^2
		\\ + \delta^2 \|\hat{w}_{x}(\cdot,t)\|_2^2
		+ r^2\|\hat{w}(\cdot,t)\|_2^2 + \|f(\cdot,t)\|_2^2,
	\end{multline*}
	Using Gagliardo--Nirenberg's interpolation inequality and $\epsilon-$Young's inequality, second and third terms at the right hand side can be estimated as
	\begin{equation*}
		(\alpha^2 + \beta^2 \|k_y(\cdot,0)\|_2^2)\|\hat{w}_{xx}(\cdot,t)\|_2^2 \leq \epsilon \|\hat{w}_{xxx}\|_2^2 + c_{\beta,\alpha,k,\epsilon} \|\hat{w}(\cdot,t)\|_2^2
	\end{equation*}
	and
	\begin{equation*}
		\delta^2 \|\hat{w}_{x}(\cdot,t)\|_2^2 \leq \epsilon \|\hat{w}_{xxx}\|_2^2 + c_{\delta,\epsilon} \|\hat{w}(\cdot,t)\|_2^2
	\end{equation*}
	respectively. Choosing $\epsilon > 0$ sufficiently small, we obtain
	\begin{equation*}
		\|\hat{w}_{xxx}(\cdot,t)\|_2^2 \lesssim \|\hat{v}(\cdot,t) \|_2^2 + \|\hat{w}(\cdot,t) \|_2^2 + \|f(\cdot,t)\|_2^2.
	\end{equation*}
	Notice from Proposition \ref{tar_err_stab}-(ii) that supremum of the trace terms $\tilde{w}_x(0,t), \tilde{w}_{xx}(0,t)$ exist. Therefore, taking supremum on both sides, we deduce that $\hat{w} \in C([0,T];H^3(0,L))$. Next, again from the main equation, we have
	\begin{equation*}
	i\beta \hat{w}_{xxxx}(x,t) = (-i\hat{v}_x - \alpha \hat{w}_{xxx} - i\delta \hat{w}_{xx} - ir\hat{y}_x)(x,t)
	+ i\beta k_{yx}(x,0)\hat{w}_x(0,t) + f_x(x,t),
	\end{equation*}
	and therefore we get
	\begin{multline*}
	\beta^2 \|\hat{w}_{xxxx}(\cdot,t)\|_2^2 \leq \|\hat{v}_x(\cdot,t)\|_2^2 + \alpha^2\|\hat{w}_{xxx}(\cdot,t)\|_2^2 \\
	+ \left(\delta^2+\beta^2 \|k_{xy}(\cdot,0)\|_2^2\right) \|\hat{w}_{xx}(\cdot,t)\|_2^2 + r^2\|\hat{w}_x(\cdot,t)\|_2^2 + \|f_x(\cdot,t)\|_2^2.
	\end{multline*}
	Similarly, by Gagliardo--Nirenberg's inequality and then $\epsilon-$Young's inequality, we get
	\begin{align*}
		\alpha^2\|\hat{w}_{xxx}(\cdot,t)\|_2^2 &\leq \epsilon \|\hat{w}_{xxxx}(\cdot,t)\|_2^2 + c_{\alpha,\epsilon}\|\hat{w}(\cdot,t)\|_2^2, \\
		\left(\delta^2+\beta^2 \|k_y(\cdot,0)\|_2^2\right) \|\hat{w}_{xx}(\cdot,t)\|_2^2 &\leq \epsilon \|\hat{w}_{xxxx}(\cdot,t)\|_2^2 + c_{\beta,\delta,k,\epsilon}\|\hat{w}(\cdot,t)\|_2^2, \\
		r^2\|\hat{w}_x(\cdot,t)\|_2^2 &\leq \epsilon \|\hat{w}_{xxxx}(\cdot,t)\|_2^2 + c_{r,\epsilon}\|\hat{w}(\cdot,t)\|_2^2.
	\end{align*}
	Using these estimates, we obtain that
	\begin{equation*}
		\|\hat{w}_{xxxx}(\cdot,t)\|_2^2 \lesssim \|\hat{v}_x(\cdot,t) \|_2^2 + \|\hat{w}(\cdot,t) \|_2^2 + \|f_x(\cdot,t)\|_2^2.
	\end{equation*}
	We see that right hand side belongs to $L^2(0,T)$, so $\hat{w}$ belongs to $L^2(0,T;H^4(0,L))$. Combining with the previous result, we obtain that $\hat{w} \in X_T^3$ if $\hat{w}_0 \in H^3(0,L)$. This finishes the proof of the following proposition.
	\begin{prop}\label{tar_obs_wp}
		Let $T,\beta > 0$, $\alpha, \delta \in \mathbb{R}$, $k$ and $p$ be smooth backstepping kernels solving \eqref{kernel_pk} and \eqref{kernel_p} respectively, and $p_1(x) = -i\beta p_y(x,0) + \alpha p(x,0)$, $p_2(x) = i\beta p(x,0)$. Assume that $\hat{w}_0 \in H^3(0,L)$ satisfies the compatibility conditions and the initial--boundary pair of the target error model $(\tilde{w}_0,\psi) \in H^6(0,L) \times H^2(0,T)$ satisfies the higher order compatibility conditions. Then \eqref{tar_obs_pde_wp} admits a unique solution $\hat{w} \in X_T^3$.
	\end{prop}
	
	Finally, thanks to bounded invertibility of the backstepping transformation \eqref{p_bt}, we obtain under the same assumptions that $\hat{u} \in X_T^3$. Combining the wellposedness of $\hat{u}$ and $\tilde{u}$, we obtained the wellposedness of \eqref{plant_lin} and proved the first part of Theorem \ref{obs_thm}.
	
	\subsection{Stability}
	
	In this part, we obtain exponential stability estimates for the plant--observer--error system. This will be done by first considering the target error and target observer models. Then we use the bounded invertibility of the backstepping transformations, which will yield the exponential stability for the error and observer models, consequently for the original plant.
	\subsubsection{Error model}
	
	\begin{prop} \label{tar_err_stab}
		Let $\beta > 0$, $\alpha, \delta \in \mathbb{R}$ and $p$ is the smooth backstepping kernel that solves \eqref{kernel_p}. Then for sufficiently small $r > 0$, it is true that
		\begin{equation*}
			\mu = \beta\left(\frac{r}{\beta} - \frac{\|p_x(L,\cdot;r)\|_2^2}{2}\right) > 0.
		\end{equation*}
		Moreover, the solution $\tilde{w}$ of \eqref{err_tar_lin} satisfies the following estimates
		\begin{enumerate}
			\item [(i)] $\|\tilde{w}(\cdot,t)\|_2 \leq \|\tilde{w}_0\|_2 e^{-\mu t}$,
			\item [(ii)] $|\tilde{w}_{xx}(0,t)| + |\tilde{w}_x(0,t)| + \|\tilde{w}(\cdot,t)\|_{H^3(0,L)} \lesssim \|\tilde{w}_0\|_{H^3(0,L)} e^{-\mu t}$,
		\end{enumerate}
		for $t \geq 0$.
	\end{prop}
	\begin{proof}
		\hfill
		\begin{enumerate}
		\item[(i)] Taking the $L^2-$inner product of the main equation of \eqref{err_tar_lin} with $2\tilde{w}$, following \eqref{lin1iden2}-\eqref{lin1iden5} and applying Cauchy--Schwarz inequality at the right hand side, we get
		\begin{equation*}
			\begin{split}
			\frac{d}{dt} \|\tilde{w}(\cdot,t)\|_2^2 + 2r \|\tilde{w}(\cdot,t)\|_2^2 + \beta|\tilde{w}_x(0,t)|^2 &= \beta |\tilde{w}_x(L,t)|^2 \\
			&= \beta\left|\int_0^Lp_x(L,y) \tilde{w}(y,t)dy\right|^2 \\
			&\leq \beta \|p_x(L,\cdot)\|_2^2 \|\tilde{w}(\cdot,t)\|_2^2.
			\end{split}
		\end{equation*}
		It follows from the last expression that
		\begin{equation*}
			\frac{d}{dt} \|\tilde{w}(\cdot,t)\|_2^2 + 2\beta\left(\frac{r}{\beta} - \frac{\|p_x(L,\cdot)\|_2^2}{2}\right) \|\tilde{w}(\cdot,t)\|_2^2 \leq 0.
		\end{equation*}
		Denoting $\mu \doteq \beta\left(\frac{r}{\beta} - \frac{\|p_x(L,\cdot)\|_2^2}{2}\right)$ and integrating the above estimate yields (i). Recall that that $p_x(L,y) = - k_y(x,0)$. Thus, we can prove that $\mu > 0$ for sufficiently small $r > 0$ as we did
		in Proposition \ref{ctrl_tar_stab}.
		
		\item[(ii)] We differentiate \eqref{err_tar_lin} with respect to $t$, take $L^2-$inner product by $2\tilde{w}_t$ and following similar steps as in part (i), we obtain
		\begin{equation*}
			\frac{d}{dt} \|\tilde{w}_t(\cdot,t)\|_2^2 + 2\beta\left(\frac{r}{\beta} - \frac{\|p_x(L,\cdot)\|_2^2}{2}\right) \|\tilde{w}_t(\cdot,t)\|_2^2 \leq 0,
		\end{equation*}
		which implies
		\begin{equation} \label{h3_est_1}
			\|\tilde{w}_t(\cdot,t)\|_2 \leq\|\tilde{w}_t(\cdot,0)\|_2 e^{-\mu t}.
		\end{equation}
		In particular, from the main equation of \eqref{err_tar_lin} together with \eqref{h3_est_1}, we get
		\begin{equation} \label{h3_est_2}
			\begin{split}
			\|\tilde{w}_t(\cdot,t)\|_2 &\leq \|\tilde{w}_t(\cdot,0)\|_2 e^{-\mu t} \\
			&= \|-\beta \tilde{w}_0^{\prime\prime\prime} + i\alpha \tilde{w}_0^{\prime\prime} - \delta \tilde{w}_0^\prime - r\tilde{w}_0\|_2 e^{-\mu t} \\
			&\lesssim \|\tilde{w}_0\|_{H^3(0,L)} e^{-\mu t}.
			\end{split}
		\end{equation}
		On the other hand, again from \eqref{err_tar_lin}, we also have
		\begin{equation} \label{h3_est_3}
			\beta \|\tilde{w}_{xxx}(\cdot,t)\|_2 ^2 \leq \alpha \|\tilde{w}_{xx}(\cdot,t)\|_2^2 + \delta \|\tilde{w}_{x}(\cdot,t)\|_2^2 + r \|\tilde{w}(\cdot,t)\|_2^2 + \|\tilde{w}_{t}(\cdot,t)\|_2^2.
		\end{equation}
		Applying Gagliardo--Nirenberg interpolation inequality and then $\epsilon-$Young's inequality, the first term at the right hand side can be estimated as
		\begin{align}
			\alpha \|\tilde{w}_{xx}(\cdot,t)\|_2^2 &\leq c_\alpha \|\tilde{w}_{xxx}(\cdot,t)\|_2^\frac{4}{3} \|\tilde{w}(\cdot,t)\|_2^\frac{2}{3} \nonumber\\
			\label{h3_est_4}
			&\leq \epsilon \|\tilde{w}_{xxx}(\cdot,t)\|_2^2 + c_{\alpha,\epsilon } \|\tilde{w}(\cdot,t)\|_2^2.
		\end{align}
		Similarly, the second term can be estimated as
		\begin{equation}
			\label{h3_est_5}
			\delta \|\tilde{w}_{x}(\cdot,t)\|_2^2 \leq \epsilon \|\tilde{w}_{xxx}(\cdot,t)\|_2^2 + c_{\delta,\epsilon } \|w(\cdot,t)\|_2^2.
		\end{equation}
		Combining \eqref{h3_est_4}-\eqref{h3_est_5} with \eqref{h3_est_3}, and then choosing $\epsilon > 0$ sufficiently small, we get
		\begin{equation*}
			\|\tilde{w}(\cdot,t)\|_{H^3(0,L)}^2 \lesssim \|\tilde{w}(\cdot,t)\|_2^2 + \|\tilde{w}_t(\cdot,t)\|_2^2,
		\end{equation*}
		Using \eqref{h3_est_2} and (i) , it follows that
		\begin{equation} \label{h3_est_6}
			\|\tilde{w}(\cdot,t)\|_{H^3(0,L)} \lesssim \|\tilde{w}_0\|_{H^3(0,L)} e^{-\mu t}.
		\end{equation}
		To estimate the trace terms in (ii), we take $L^2-$inner product of \eqref{err_tar_lin} by $2(L - x)\tilde{w}_{xx}$ and consider only the imaginary terms to get
		\begin{multline} \label{h3_est_7}
			2\Re\int_0^L \tilde{w}_t \overline{\tilde{w}}_{xx}(L-x)dx + 2\beta\Re \int_0^L \tilde{w}_{xxx}\overline{\tilde{w}}_{xx}(L-x) dx + 2\alpha \Im \int_0^L \tilde{w}_{xx}\overline{\tilde{w}}_{xx}(L-x) dx \\ +
			2\delta\Re \int_0^L \tilde{w}_x \overline{\tilde{w}}_{xx}(L-x) dx  + 2r\Re \int_0^L \tilde{w} \overline{\tilde{w}}_{xx}(L-x) dx = 0.
		\end{multline}
		Integrating by parts, the second term is equivalent to
		\begin{equation*}
			\begin{split}
			2\beta\Re\int_0^L \tilde{w}_{xxx}\overline{\tilde{w}}_{xx}(L-x) dx &= \beta \Re \int_0^L \frac{d}{dx}|\tilde{w}_{xx}|^2 (L - x) dx \\
			&= \beta\left(-L|\tilde{w}_{xx}(0,t)|^2 + \|\tilde{w}_{xx}(\cdot,t)\|_2^2\right).
			\end{split}
		\end{equation*}
		The third term vanishes since it is pure real. The fourth term, again by integration by parts, can be expressed as
		\begin{equation*}
			\begin{split}
			2\delta\Re \int_0^L \tilde{w}_x \overline{\tilde{w}}_{xx}(L-x) dx &= \delta \Re\int_0^L \frac{d}{dx} |\tilde{w}_x|^2 (L - x) dx \\
			&= \delta \left(-L |\tilde{w}_x(0,t)|^2 + \|\tilde{w}_x(\cdot,t)\|_2^2\right).
			\end{split}
		\end{equation*}
		Using  these estimates in \eqref{h3_est_7}, we obtain that
		\begin{equation*}
			\begin{split}
			L(\beta |\tilde{w}_{xx}(0,t)|^2 + \delta |\tilde{w}_x(0,t)|^2) =& 2 \Re \int_0^L \tilde{w}_t \overline{\tilde{w}}_{xx}(L-x)dx + \beta  \|\tilde{w}_{xx}(\cdot,t)\|_2^2 \\
			&+ \delta \|\tilde{w}_x(\cdot,t)\|_2^2 + 2r \Re \int_0^L \tilde{w} \overline{\tilde{w}}_{xx}(L-x) dx.
			\end{split}
		\end{equation*}
		Applying Cauchy--Schwarz inequality and then Young's inequality on the first and last terms at the right hand side, using \eqref{h3_est_2} and \eqref{h3_est_6}, we get
		\begin{equation*}
			\begin{split}
			|\tilde{w}_{xx}(0,t)|^2 + |\tilde{w}_x(0,t)|^2 &\lesssim \|\tilde{w}_t(\cdot,t)\|_2^2 + \|\tilde{w}(\cdot,t)\|_{H^3(0,L)}^2 \\
			&\lesssim e^{-\mu t}\|\tilde{w}_0\|_{H^3(0,L)}.
			\end{split}
		\end{equation*}
		Combining this result with \eqref{h3_est_6} yields (ii).
		\end{enumerate}
	\end{proof}
	Since $p$ is a smooth function on a compact set $\Delta_{x,y}$ and the backstepping transformation \eqref{kernel_p} is invertible on  $L^2(0,L)$ and  $H^3(0,L)$ with a bounded inverse, we obtain that
	\begin{equation} \label{er_est_1}
		\|\tilde{u}(\cdot,t)\|_2 \leq c_p \|\tilde{u}_0\|_2, \quad c_p = \left(1 + \|p\|_{L^2(\Delta_{x,y})}\right) \|(I - \Upsilon_p)^{-1}\|_{2 \to 2}
	\end{equation}
	and
	\begin{equation}\label{er_est_2}
		\|\tilde{u}(\cdot,t)\|_{H^3(0,L)} \leq c_{p^\prime} \|\tilde{u}_0\|_{H^3(0,L)}, \quad c_{p^\prime} = \left(1 + \|p\|_{H^3(\Delta_{x,y})}\right) \|(I - \Upsilon_p)^{-1}\|_{H^3(0,L) \to H^3(0,L)}.
	\end{equation}

	\subsubsection{Observer model}
	\begin{prop} \label{tar_obs_stab}
		Let $\beta > 0$, $\alpha, \delta \in \mathbb{R}$ and $k, p$ be the smooth backstepping kernels solving \eqref{kernel_pk}, \eqref{kernel_p} respectively. Then for sufficiently small $r,\epsilon > 0$, it is true that
		\begin{equation*}
			\nu \doteq \beta\left(\frac{r}{\beta} - \frac{\|k_y(\cdot,0;r)\|_2^2}{2} - \epsilon \left(\|\Pi_1\|_2^2 + \|\Pi_2\|_2^2\right)\right) > 0,
		\end{equation*}
		where $\|\Pi_j\|_2^2 = \|(I - \Upsilon_k) p_j\|_2^2 $,  $j = 1,2$ with $p_1(x) = -i\beta p_y(x,0) + \alpha p(x,0)$, $p_2(x) = i\beta p(x,0)$. Moreover,  the solution $\hat{w}$ of \eqref{tar_p_obs_lin} satisfies the following estimate
		\begin{equation} \label{obs_obs_stab_est}
			\|\hat{w}(\cdot,t)\|_2 \lesssim e^{-\nu t} \left(\|\hat{w}_0\|_2 + \|\tilde{w}_0\|_{H^3(0,L)}\right)
		\end{equation}
		for $t \geq 0$.
	\end{prop}
	\begin{proof}
		We take $L^2-$inner product of the main equation of \eqref{tar_p_obs_lin} by $2\hat{w}$ and following the steps \eqref{lin1iden2}-\eqref{lin1iden5}, we get
		\begin{equation} \label{tar_obs_stab_1}
			\begin{split}
			\frac{d}{dt} \|\hat{w}(\cdot,t)\|_2^2 + \beta |\hat{w}_x(0,t)|^2 + 2r \|\hat{w}(\cdot,t)\|_2^2
			=& 2\beta \Re \int_0^L k_y(x,0) \hat{w}_x(0,t)  \overline{\hat{w}}(x,t) dx \\
			&+ 2\Im\int_0^L \Pi_1(x) \tilde{w}_{x}(0,t) \overline{\hat{w}}(x,t) dx \\
			&+ 2\Im\int_0^L  \Pi_2(x) \tilde{w}_{xx}(0,t) \overline{\hat{w}}(x,t) dx.
			\end{split}
		\end{equation}
		Using Young's inequality and then Cauchy--Schwarz inequality, the first term at the right hand side can be estimated as
		\begin{equation*}
			2\beta \Re \int_0^L k_y(x,0) \hat{w}_x(0,t) \overline{\hat{w}}(x,t)dx
		 \leq \beta |\hat{w}_x(0,t)|^2 + \beta \|k_y(\cdot,0)\|_2^2 \|\hat{w}(\cdot,t)\|_2^2.
		\end{equation*}
		Applying Cauchy--Schwarz inequality and $\epsilon-$Young's inequality to the second and third terms at the right hand side of \eqref{tar_obs_stab_1}, we get
		\begin{equation*}
			\left|2\Im\int_0^L  \Pi_j(x) \tilde{w}_{x}(0,t) \overline{\hat{w}}(x,t)\right|
			\leq 2\epsilon\beta \bigr\| \Pi_j\bigr\|_2^2 \|\hat{w}(\cdot,t)\|_2^2 + \frac{1}{2\epsilon\beta} |\tilde{w}_{x}(0,t)|^2, \quad j = 1, 2.
		\end{equation*}
		Using these estimates in \eqref{tar_obs_stab_1}, we obtain that
		\begin{multline} \label{tar_obs_stab_2}
			\frac{d}{dt} \|\hat{w}(\cdot,t)\|_2^2 + 2\beta\left(\frac{r}{\beta} - \frac{\|k_y(\cdot,0)\|_2^2}{2} - \epsilon \left(\|\Pi_1\|_2^2 + \|\Pi_2\|_2^2\right)\right)\|\hat{w}(\cdot,t)\|_2^2 \\
			\leq \frac{1}{2\epsilon\beta} \left(|\tilde{w}_{x}(0,t)|^2 + |\tilde{w}_{xx}(0,t)|^2\right).
		\end{multline}
		From Proposition \ref{ctrl_tar_stab} we know that, there exists a sufficiently small $r >0$ such that the term $\left(\frac{r}{\beta} - \frac{\|k_y(\cdot,0)\|_2^2}{2}\right)$ remains positive. So choosing $\epsilon$ sufficiently small, we are able to guarantee that the term
		\begin{equation*}
			\nu \doteq \beta\left(\frac{r}{\beta} - \frac{\|k_y(\cdot,0)\|_2^2}{2} - \epsilon \left(\|\Pi_1\|_2^2 + \|\Pi_2\|_2^2\right)\right)
		\end{equation*}
		remains positive. Now applying Proposition \ref{tar_err_stab}-(ii) to the right hand side of \eqref{tar_obs_stab_2}, we get
		\begin{equation*}
			\frac{d}{dt} \|\hat{w}(\cdot,t)\|_2^2 + 2 \nu \|\hat{w}(\cdot,t)\|_2^2 \lesssim \|\tilde{w}_0\|_{H^3(0,L)}^2 e^{-2\mu t}.
		\end{equation*}
		Also, observing $\|p_x(L,\cdot)\|_2 = \| k_y(\cdot,0)\|_2$ and comparing $\nu$ with $\mu$, we observe that $\mu > \nu$. Thus, integrating the above inequality from $0$ to $t$, we finally obtain
		\begin{equation*}
			\|\hat{w}(\cdot,t)\|_2 \lesssim e^{-\nu t} \left(\|\hat{w}_0\|_2 + \|\tilde{w}_0\|_{H^3(0,L)}\right).
		\end{equation*}
	\end{proof}
	Since $k$, $p$ are smooth backstepping kernels on the triangular domain $\Delta_{x,y}$ and thanks to the invertibility of the corresponding backstepping transformations \eqref{kernel_pk}, \eqref{kernel_p} on $L^2(0,L)$ and $H^3(0,L)$ respectively, with a bounded inverse, we deduce that
	\begin{equation} \label{obs_est_1}
		\|\hat{u}(\cdot,t)\|_2\lesssim  c_{k,p} e^{-\nu t} \left(\|\hat{u}_0\|_2 + \|\tilde{u}_0\|_{H^3(0,L)}\right),
	\end{equation}
	where $c_{k,p}$ is the maximum of
	\begin{equation*}
		c_k = \left(1 + \|k\|_{L^2(\Delta_{x,y})}\right) \|(I - \Upsilon_k)^{-1}\|_{2 \to 2}
	\end{equation*}
	and
	\begin{equation*}
		c_p = \left(1 + \|p\|_{H^3(\Delta_{x,y})}\right) \|(I - \Upsilon_p)^{-1}\|_{H^3(0,L) \to H^3(0,L)}.
	\end{equation*}
	Finally, combining \eqref{er_est_1} and \eqref{obs_est_1}
	\begin{align*}
		\|u(\cdot,t)\|_2 &= \|(\hat{u} + \tilde{u})(\cdot,t)\|_2 \\
		&\leq \|\hat{u}(\cdot,t)\|_2 + \|(u - \hat{u})(\cdot,t)\|_2 \\
		&\lesssim c_{k,p}\left(\|\hat{u}_0\|_2 + \|u_0 - \hat{u}_0\|_{H^3(0,L)} \right)e^{-\nu t} + c_p\|u_0 - \hat{u}_0\|_2 e^{-\mu t}.
	\end{align*}
	This gives us the second part of Theorem \ref{obs_thm}.
	
	\section{Numerical simulations} \label{num_sim}
	In this part, we present our numerical algorithm and numerical simulations for controller and observer designs.
	
	\subsection{Controller design} \label{num_ctrl}
	Our algorithm consists of three steps. We first obtain an approximation for the backstepping kernel $k$ by solving the integral equation \eqref{GepsInt}. Then we solve the modified target equation \eqref{p_tar_lin} numerically. As a third and final step, we use the invertibility of the backstepping transformation and end up with the numerical solution to the original plant. Details are given in the below.
	
	\begin{itemize}
		\item [\textbf{Step i.}] We solve the integral equation
		\begin{eqnarray*}
			G^{j+1}(s,t)=\frac{r}{3\beta}st+\int_0^t\int_0^s\int_0^\omega[DG^j](\xi,\eta)d\xi d\omega d\eta, \quad j = 1, 2, \dotsc
		\end{eqnarray*}
		iteratively, where the iteration is initialized with
		\begin{equation*}
			G^1(s,t) = \frac{r}{3\beta}st.
		\end{equation*}
		As the initial function is a polynomial, the result of the each iteration yields again a polynomial. Thus, here, we use the advantage of the fact that summation and multiplication with a scalar of polynomials, their differentiation and integration can be carried out easily by simple algebraic operations. To perform these operations computationally, we express a given $n-$th degree polynomial with complex coefficients, say
		\begin{equation} \label{PolRep}
			\begin{split}
		P(s,t) =& \alpha_{0,0} + \alpha_{1,0}s + \alpha_{0,1}t + \alpha_{2,0}s^2 + \alpha_{1,1} st + \alpha_{0,2} t^2 + \dotsm \\
		&+ \alpha_{n,0}s^n + \alpha_{n-1,1}s^{n-1}t + \alpha_{n-2,2}s^{n-2}t^2 + \dotsm + \alpha_{0,n}t^n,
			\end{split}
		\end{equation}
		in a more convenient form as
		\begin{equation} \label{MatRep}
		\left[\mathrm{P}\right] =
		\begin{bmatrix}
		\alpha_{0,0} & \alpha_{0,1} & \cdots     & \alpha_{0,n-1}    & \alpha_{0,n} \\
		\alpha_{1,0} & \alpha_{1,1} & \cdots     & \alpha_{1,n-1} & \\
		\vdots  &  \vdots & \reflectbox{$\ddots$} & \\
		\alpha_{n-1,0} & \alpha_{n-1,1} && \mbox{\Huge 0}\\
		\alpha_{n,0} & &  &&
		\end{bmatrix}.
		\end{equation}
		Once we introduce this matrix representation \eqref{MatRep} of $P$ in our algorithm, then it is easy to perform summation and scalar multiplication. Moreover, using the elementary row and column operations, one can perform differentiation and integration. For instance multiplying the $j-$th row of $[P]$ by $j-1$, writing the result to the $(j-1)-$th row and repeating this process for each $j$, $j = 2, 3, \dotsc , n + 1$ yields the matrix representation of $P_s(s,t)$, $[P_s]$. Similarly, multiplying the $j-$th row of $[P]$ by $1/j$, writing the result to the $(j+1)-$th row and repeating this process for each $j$, $j = 1, 2, \dotsc , n + 1$ yields $[\int_0^s P ds]$. Differentiation and integration with respect to $t$ can be done by performing analogous column operations.
		
		\item [\textbf{Step ii.}] Let us consider the uniform discretization of $[0,L]$ with the set of $M > 3$ node points $\{x_m\}_{m=1}^M$ where $x_m = (m - 1)h_x$ and $h_x = \frac{L}{M -1}$ is the the uniform spatial grid spacing. Let us introduce the following finite dimensional vector space
		\begin{eqnarray*}
			\mathrm{X}^M &\doteq& \left\{\mathbf{w} = [w_1 \cdots w_M]^T \in \mathbb{C}^M \right\}
		\end{eqnarray*}
		where each $\mathbf{w} \in \mathrm{X}^M$ satisfies
		\begin{eqnarray}
		\label{num_bc1} w_1(t) = w_M(t) &=& 0, \\
		\label{num_bc2} \frac{w_{M - 2}(t) - 4  w_{M - 1}(t) + 3 w_{M}(t)}{2h_x} &=& 0,
		\end{eqnarray}
		for $t > 0$. Note that $w_m(t)$ is an approximation to $w(x,t)$ at the point $x = x_m$ and \eqref{num_bc1} and \eqref{num_bc2} correspond to Dirichlet and Neumann type boundary conditions respectively. Consider the standard forward and backward difference operators $\mathbf{\Delta}_+ : \mathrm{X}^M \to \mathrm{X}^M$ and $\mathbf{\Delta}_-: \mathrm{X}^M \to \mathrm{X}^M$, respectively and let us introduce the following finite difference operators on $\mathrm{X}^M$:
		\begin{equation}
		\label{fin_dif_op}
		\mathbf{\Delta} \doteq \frac{1}{2} \left(\mathbf{\Delta}_+ + \mathbf{\Delta}_-\right)\quad
		\mathbf{\Delta}^2 \doteq \mathbf{\Delta}_{+} \mathbf{\Delta}_{-} \quad
		\mathbf{\Delta}^3 \doteq \mathbf{\Delta}_+ \mathbf{\Delta}_+ \mathbf{\Delta}_-.
		\end{equation}
		Next assume $N$ to be a positive integer, $T$ be the final time and consider the nodal points in time axis $t_n = (n - 1)k$, where $n = 1, \dotsc , N$ is time index and $h_t = \frac{T}{N - 1}$ is the time step size. Let $\mathbf{w^n} = [w_1^n \cdots w_m^n]^T$ be an approximation of the solution at the $n$-th time step where $w_m^n$ is an approximation to $w(x,t)$ at the point $(x_m,t_n)$. Discretizing \eqref{p_tar_lin} in space by using the finite difference operators \eqref{fin_dif_op} and in time by using Crank--Nicolson time stepping, we end up with the discrete problem: Given $\mathbf{w}^n \in \mathrm{X}^M$, find $\mathbf{w}^{n+1} \in \mathrm{X}^M$ such that
		\begin{equation}
			\left(\mathbf{I}^M + \frac{h_t}{2}\mathbf{A} - \frac{\beta h_t}{2} \mathbf{K}_y^M(\cdot,0) \mathbf{\Gamma}_0^{1,M} \right)\mathbf{w}^{n+1} = \mathbf{F}\mathbf{w}^{n}, \quad n = 1,2, \dotsc, N,
		\end{equation}
		where $\mathbf{I}^M$ is the identity matrix on $\mathrm{X}^M$, $\mathbf{A}$ is defined as
		\begin{equation}
			\mathbf{A} \doteq \beta \mathbf{\Delta}^3 - i\alpha \mathbf{\Delta}^2 + \delta \mathbf{\Delta} + r \mathbf{I}^M,
		\end{equation}
		$\mathbf{K}_y^M(\cdot,0)$ is an $M \times M$ diagonal matrix, where each element on the diagonal consists of the elements of the form $k_y(x_m,0)$, $m = 1, \dotsc, M$ and $k_y(x,0)$ is obtained exactly in the previous step, $\mathbf{\Gamma}_0^{1,M}$ is a discrete counterpart of the trace operator $\mathbf{\Gamma}_0^{1}$ and given by an $M \times M$ matrix
		\begin{equation}\label{traceop}
			\mathbf{\Gamma}_0^{1,M} = \frac{1}{2h_x}
			\begin{bmatrix}
				 -3 & 4 & -1 & 0 & \cdots & 0 \\
				 -3 & 4 & -1 & 0 & \cdots & 0 \\
				\vdots & \vdots & \vdots & \vdots  & \ddots & \vdots  \\
				 -3 & 4 & -1 & 0 & \cdots & 0\\
			\end{bmatrix},
		\end{equation}
		and
		\begin{equation*}
			\mathbf{F} \doteq \mathbf{I}^M - \frac{h_t}{2}\mathbf{A} + \frac{\beta h_t}{2} \mathbf{K}_y^M(\cdot,0) \mathbf{\Gamma}_0^{1,M}.
		\end{equation*}
		Note that the nonzero elements in the matrix $\mathbf{\Gamma}_0^{1,M}$ given in \eqref{traceop} are due to the one--sided second order finite difference approximation to the first order derivative at the point $x = 0$.
		
		\item [\textbf{Step iii.}] Now we find the inverse image, $u$, of $w$ under the backstepping transformation: Given $w$, we find $u$ by using succession method. More precisely, we set $v \doteq \Upsilon_ku$, therefore we obtain $u = v + w$ and substitute $u$ by $v + w$ on \eqref{p_bt} to get
		\begin{equation*}
			v(x,t) = \int_0^x k(x,y) w(y,t) dy + \int_0^x k(x,y) v(y,t) dy.
		\end{equation*}
		Now given $w$ obtained numerically in the previous step, we solve this equation successively for $v$. Using the numerical results for $w$ and $v$ on $u = v + w$, we obtain a numerical solution for $u$.	
	\end{itemize}

	Now, let us present a numerical simulation that verifies our stability results. We take $M = 1001$ spatial nodes, $N = 5001$ time steps. The iteration for the backstepping kernel is performed $j = 27$ times so that the error is around
	\begin{equation*}
		\underset{(s,t) \in \Delta_{s,t}}{\max}|G^{j+1} - G^j| \sim 10^{-14}.
	\end{equation*}
	We consider the following model
	\begin{eqnarray} \label{ctrl_sim}
	\begin{cases}
	iu_t + i u_{xxx} + 2u_{xx} +8i u_x = 0, \quad x\in (0,\pi), t\in (0,T),\\
	u(0,t)=0, u(\pi,t)=h_0(t), u_x(\pi,t)=h_1(t),\\
	u(x,0)= 3 - e^{4ix} - 2e^{-2ix}.
	\end{cases}
	\end{eqnarray}
	In the absence of controllers, i.e. $h_0(t) \equiv h_1(t) \equiv 0$, we have a stationary solution $u(x,t) = 3 - e^{4ix} - 2e^{-2ix}$. Let us choose $r = 0.05$. This choice yields a positive exponent value $\lambda$, defined in Proposition \ref{ctrl_tar_stab}, i.e. solution is decaying exponentially in time (see Figure \ref{table_decay}). Contour plot of the corresponding solution and time evolution of its $L^2-$norm are given by Figures \ref{fig:c_plant}.
	\begin{figure}[h]
		\centering
		\begin{subfigure}[b]{0.5\textwidth}
			\includegraphics[width=\textwidth]{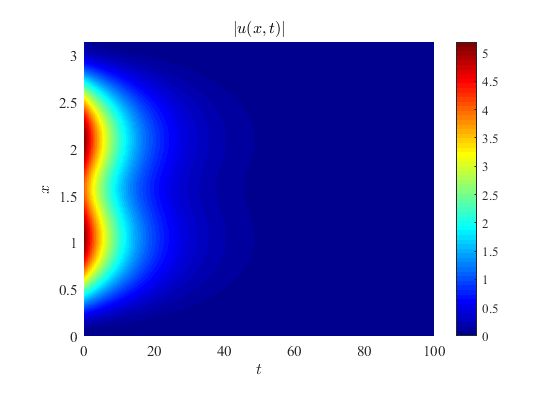}
			\label{fig:c_plant_contour}
		\end{subfigure}
		~ 
		\begin{subfigure}[b]{0.5\textwidth}
			\includegraphics[width=\textwidth]{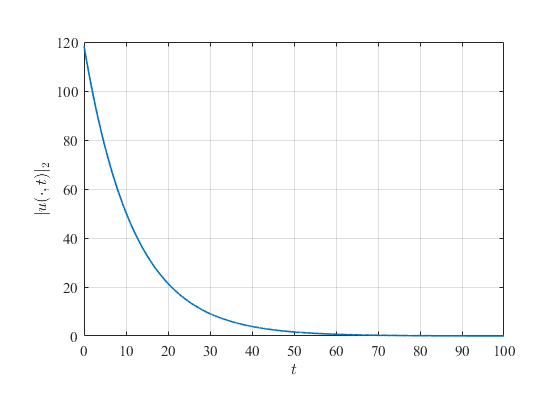}
			\label{fig:c_l2}
		\end{subfigure}
		\vspace*{-15mm}
		\caption{Numerical results in the presence of controllers. Left: Time evolution of $|u(x,t)|$. Right: Time evolution of $\|u(\cdot,t)\|_2$.}
		\label{fig:c_plant}
	\end{figure}

	\subsection{Observer design}
	Our algorithm consists of five steps. First we obtain an approximation for the backstepping kernel $p$. In the second and third steps, we obtain a numerical solution for the error model \eqref{err_lin} and modified target observer model \eqref{tar_p_obs_lin}, respectively. As a fourth step, we get a numerical solution for the observer model by using the invertibility of the backstepping transformation \eqref{kernel_pk}. At the fifth and the last step, we deduce numerical solution of the original plant via $u = \hat{u} + \tilde{u}$.

	\begin{itemize}
		\item [\textbf{Step i.}] Following the same procedure we introduced in the first step of Section \ref{num_ctrl} and then changing the variables first as $\tilde{s} = L - t$, $\tilde{t} = L - s - t$ then as $s = x - y$, $t = y$, we get
		\begin{equation*}
			G(\tilde{s},\tilde{t};-r) = G(L - t,L - s - t;-r) = k(L-y,L-x;-r) = p(x,y).
		\end{equation*}
		Note that using $p$, we also derive $p_1(x) = -i\beta p_y(x,0) + \alpha p(x,0)$ and $p_2(x) = i\beta p(x,0)$.
		
		\item [\textbf{Step ii.}] To solve \eqref{err_lin}, we apply the same discretization procedure as we introduced in the second step of Section \ref{num_ctrl}. The trace terms included in the main equation of \eqref{err_lin} are approximated by the following one sided second order finite differences
		\begin{equation}
			\begin{split}
			\label{trace_fin_dif}	
			\tilde{u}_x(0,t) &\approx \frac{-3\tilde{u}_0(t) + 4\tilde{u}_1(t) - \tilde{u}_2(t)}{2h_x},\\
			\tilde{u}_{xx}(0,t) &\approx \frac{2\tilde{u}_{0}(t) - 5\tilde{u}_{1}(t) + 4\tilde{u}_{2}(t) - \tilde{u}_{3}(t)}{h_x^2}.
			\end{split}
		\end{equation}
		
		\item [\textbf{Step iii.}] Applying the similar discretization procedure, now we solve \eqref{tar_p_obs_lin} numerically. Note that using $\tilde{w}(0,t) = 0$ and $p(x,x) = 0$, one can show  by using the backstepping transformation \eqref{bt_p} that $\tilde{w}_x(0,t) = \tilde{u}_x(0,t)$ and $\tilde{w}_{xx}(0,t) = \tilde{u}_{xx}(0,t)$. Therefore, instead of approximating the first order and second order traces of $\tilde{w}$ at the left end point, we can use \eqref{trace_fin_dif}. Note also that a discrete counterpart, $\mathbf{\Upsilon}_k^M$, of $\Upsilon_k$ can be obtained by applying a suitable numerical integration technique. For instance applying composite trapezoidal rule yields the following representation
		\begin{equation}\label{IntOp}
		\mathbf{\Upsilon}_k^M = h_x
		\begin{bmatrix}
		0 & 0 & \cdots & 0 & 0 \\
		\frac{1}{2}k(x_{2},x_{1}) & \frac{1}{2}k(x_{2},x_{2}) & \cdots & 0 & 0 \\
		\vdots  & \vdots  & \ddots & \vdots & \vdots  \\
		\frac{1}{2}k(x_{M-1},x_1) & k(x_{M-1},x_2) & \cdots & \frac{1}{2}k(x_{M-1},x_{M-1}) & 0 \\
		\frac{1}{2}k(x_M,x_1) & k(x_M,x_2) & \cdots & k(x_M,x_{M-1}) & \frac{1}{2}k(x_M,x_M)
		\end{bmatrix}.
		\end{equation}
		
		\item [\textbf{Step iv.}] Using the invertibility of the backstepping transformation \eqref{p_bt}, we obtain inverse image $\hat{u}$ of $\hat{w}$. This will be done by applying a similar procedure as we introduced in the first step of Section \ref{num_ctrl}.
		
		\item [\textbf{Step v.}] Using the numerical results for the observer and error models and setting $u = \hat{u} + \tilde{u}$, we deduce an approximation for the solution of the original plant.
	\end{itemize}

	Now let us go on with the numerical simulations. We obtain our results by taking $M = 1001$ spatial nodes, $N = 5001$ time steps. We performed the iteration for $p(x,y)$ and $k(x,y)$ several times so that the error is around $\underset{(s,t) \in \Delta_{s,t}}{\max}|G^{j+1} - G^j| \sim 10^{-14}$. We consider the same model
	\begin{eqnarray*}
		\begin{cases}
			iu_t + i u_{xxx} + 2u_{xx} +8i u_x = 0, \quad x\in (0,\pi), t\in (0,T),\\
			u(0,t)=0, u(\pi,t)=h_0(t), u_x(\pi,t)=h_1(t),
		\end{cases}
	\end{eqnarray*}
	where, unlike the controller design case, the feedback controllers use the state of the observer model. We initialize the error model as $\tilde{u}(x,0) = 3 - e^{4ix} - 2e^{-2ix}$ and observer model $\hat{u}(x,0) \equiv 0$. We take $r = 0.05$. Since the problem parameters are same as the previous numerical example, this choice will yield positive exponent vales $\mu > \nu > 0$ where $\mu, \nu$ are defined in Proposition \ref{tar_err_stab} and Proposition \ref{tar_obs_stab}.
	
	Contour plot of the numerical solution of original plant is given at the left side of Figure \ref{fig:o_plant}. At the right, we show time evolution of the $L^2-$norms of solutions of plant-observer-error system.
	
	\begin{figure}[h]
		\centering
		\begin{subfigure}[b]{0.5\textwidth}
			\includegraphics[width=\textwidth]{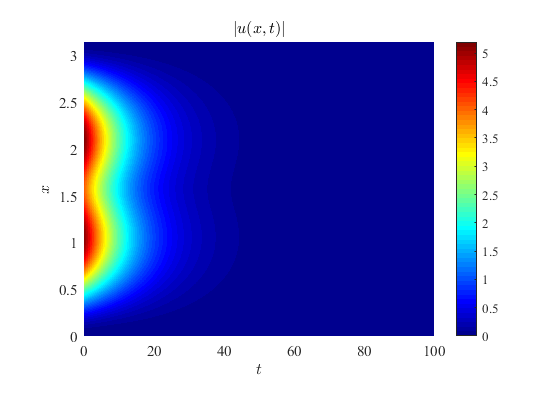}
			\label{fig:o_plant_contour}
		\end{subfigure}
		~ 
		\begin{subfigure}[b]{0.5\textwidth}
			\includegraphics[width=\textwidth]{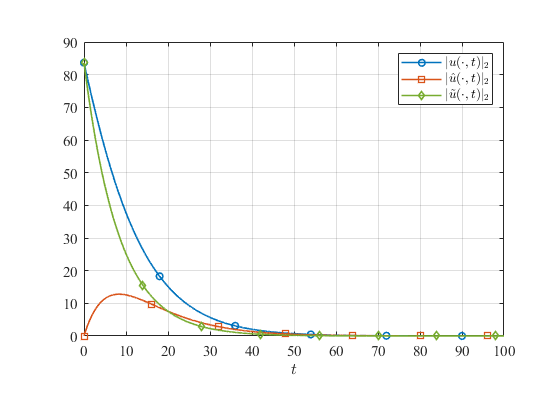}
			\label{fig:o_l2s}
		\end{subfigure}
		\vspace*{-15mm}
		\caption{Numerical results. Left: Time evolution of $|u(x,t)|$. Right: Time evolution of $|u(\cdot,t)|_2$.}
		\label{fig:o_plant}
	\end{figure}
	
	\appendix
	\section{Deduction of the kernel pde model \eqref{kernel_k}} \label{app1}
	In this section, we present the details of the calculations for obtaining the kernel model given in \eqref{kernel_k}. Differentiating both sides of \eqref{bt} with respect to $t$ we get
	\begin{equation*}
		\begin{split}
			iw_t(x,t) =& iu_t(x,t)- \int_0^x i k(x,y)u_t(y,t)dy\\
			=&iu_t(x,t) + \int_0^xk(x,y)(i\beta u_{yyy}(y,t) + \alpha u_{yy}(y,t)+ i\delta u_y(y,t))dy\\
			=& i u_t(x,t) \\
			&+ i\beta \left( k(x,y)u_{yy}(y,t) - k_y(x,y)u_y(y,t) + k_{yy}(x,y) u(y,t)\Biggr|_0^x \right.\\
			&\left.- \int_0^x k_{yyy}(x,y)u(y,t)dy \right) \\
			&+\alpha \left(k(x,y) u_y(y,t) - k_y(x,y) u(y,t) \Biggr|_0^x + \int_0^x k_{yy}(x,y)u(y,t)dy \right) \\
			&+i\delta \left(k(x,y) u(y,t)\Biggr|_0^x - \int_0^x k_y(x,y)u(y,t)dy\right).
		\end{split}
	\end{equation*}
	Using the boundary condition $u(0,t) = 0$, rearrenging the last expression in terms of $u(x,t)$,$u_x(x,t)$, $u_{xx}(x,t)$, $u_x(0,t)$ and $u_{xx}(0,t)$, we obtain
	\begin{equation} \label{wrtt}
		\begin{split}
			iw_t(x,t) =& i u_t(x,t) + \int_0^x \left(-i\beta k_{yyy} + \alpha k_{yy} - i\delta k_y\right)(x,y) u(y,t) dy \\
			&+ \left(i\beta k_{yy} (x,x) - \alpha k_y(x,x) + i\delta k(x,x)\right) u(x,t) \\
			&+\left(-i\beta k_y (x,x) + \alpha k(x,x)\right)u_x(x,t) + i\beta k(x,x) u_{xx}(x,t) \\
			&+ \left(i\beta k_y(x,0) - \alpha k(x,0)\right) u_x(0,t) - i\beta k(x,0) u_{xx}(0,t).
		\end{split}
	\end{equation}
	Next we differentiate both sides of \eqref{p_bt} with respect to $x$ up to the order three and multiply the results by $i\delta$, $\alpha$ and $i\beta$ respectively to obtain
	\begin{equation} \label{wrtx}
		\begin{split}
			i\delta w_x(x,t) &= i\delta u_x(x,t) - i\delta\frac{\partial}{\partial x} \int_0^x  k(x,y)u(y,t)dy \\
			&= i\delta u_x(x,t) - \int_0^x i\delta k_x(x,y) u(y,t)dy - i\delta k(x,x) u(x,t),
		\end{split}
	\end{equation}
	\begin{align} \label{wrtxx}
		\begin{split}
			\alpha w_{xx}(x,t) =& \alpha u_{xx}(x,t) - \alpha\frac{\partial}{\partial x} \int_0^x 		k_x(x,y)u(y,t)dy - \alpha \frac{\partial}{\partial x}\left(k(x,x) u(x,t)\right) \\
			=& \alpha u_{xx}(x,t) - \int_0^x \alpha k_{xx}(x,y)u(y,t)dy \\
			&+\alpha  \left(-k_x(x,x) - \frac{d}{dx} k(x,x)\right) u(x,t) - \alpha k(x,x) u_x(x,t),
		\end{split}
	\end{align}
	and
	\begin{equation} \label{wrtxxx}
		\begin{split}
			i\beta w_{xxx}(x,t) =& i\beta u_{xxx}(x,t) - i\beta \frac{\partial}{\partial x}\int_0^x k_{xx}(x,y)u(y,t)dy \\
			&- i\beta \frac{\partial}{\partial x} \left(\left(k_x(x,x) + \frac{d}{dx} k(x,x)\right)u(x,t) + k(x,x)u_x(x,t)\right) \\
			=& i\beta u_{xxx}(x,t) - \int_0^x i \beta k_{xxx}(x,y)u(y,t)dy \\
			&+ i\beta \left(-k_{xx}(x,x) - \frac{d}{dx}k_x(x,x) - \frac{d^2}{dx^2} k(x,x)\right)u(x,t)  \\
			&+i\beta \left(-k_x(x,x) - 2\frac{d}{dx} k(x,x)\right)u_x(x,t) - i\beta  k(x,x) u_{xx}(x,t).
		\end{split}
	\end{equation}
	Adding \eqref{wrtt}-\eqref{wrtxxx} side by side together with $irw(x,t) = iru(x,t) - ir \int_0^x k(x,y)u(y,t)dy$ and using the main equation of the linear plant, we obtain
	\begin{align}\label{k_cond_1}
		&i w_t + i\beta w_{xxx} + \alpha w_{xx} + i\delta w_x + ir w \nonumber\\
		=& \int_0^x \left(-i\beta(k_{xxx} + k_{yyy}) - \alpha(k_{xx} - k_{yy}) -i\delta (k_x + k_y) - irk\right)(x,y)u(y,t) dy \\
		\label{k_cond_2}
		=& \left(i\beta \left(k_{yy}(x,x) - k_{xx}(x,x) - \frac{d}{dx} k_x(x,x) - \frac{d^2}{dx^2} k(x,x)\right)\right.\nonumber\\
		&\left.+ \alpha \left(-k_y(x,x) - k_x(x,x) - \frac{d}{dx} k(x,x)\right) + ir\right)u(x,t) \\
		\label{k_cond_3}
		&-i\beta \left(k_y(x,x) + k_x(x,x) + 2 \frac{d}{dx }k(x,x)\right)u_x(x,t) \\
		\label{k_cond_4}
		&+\left(i\beta k_y(x,0)- \alpha k(x,0)\right)u_x(0,t) \\
		\label{k_cond_5}
		&-i\beta k(x,0) u_{xx}(0,t).
	\end{align}
	From \eqref{k_cond_5} we have $k(x,0) = 0$ and therefore, from \eqref{k_cond_4} we get $k_y(x,0) = 0$. Using the relation $\frac{d}{dx}k(x,x) = k_x(x,x) + k_y(x,x)$, we obtain from \eqref{k_cond_3} that
	\begin{equation*}
		\frac{d}{dx} k(x,x) = 0
	\end{equation*}
	and thanks to $k(x,0) = 0$, this implies $k(x,x) = 0$. Next, we differentiate $\frac{d}{dx}k(x,x) = k_x(x,x) + k_y(x,x)$ with respect to $x$ and use $\frac{d}{dx} k(x,x) = 0$ to obtain $k_{yy}(x,x) = -2k_{xy}(x,x) - k_{xx}(x,x)$. Using this result on \eqref{k_cond_2}, we deduce that
	\begin{equation*}
		\frac{d}{dx} k_x(x,x) = \frac{r}{3\beta}
	\end{equation*}
	which, by the implications $k_y(x,0) = 0 \Rightarrow k_x(x,0) = 0 \Rightarrow k_x(0,0) = 0$, is equivalent to
	\begin{equation*}
		k_x(x,x) = \frac{rx}{3\beta}.
	\end{equation*}
	Also note that taking $x = 0$ in the backstepping transformation implies $w(0,t) = u(0,t) = 0$ and, taking $x = L$ implies
	\begin{equation*}
		w(L,t) = u(L,t) - \int_0^L k(L,y) u(y,t) dy = 0
	\end{equation*}
	and
	\begin{equation*}
		w_x(L,t) = u_x(L,t) - \int_0^L k_x(L,y) u(y,t) dy - k(0,0) u(0,t) = 0.
	\end{equation*}
	So the boundary conditions are being satisfied without any extra conditions on $k$.
	
	As a conclusion, linear plant is mapped to target model (not modified one) if $k(x,y)$ satisfies the following boundary value problem
	\begin{eqnarray*}
		\begin{cases}
			\beta(k_{xxx}+k_{yyy})-i\alpha(k_{xx}-k_{yy})+\delta(k_x+k_y)+rk=0, \\
			k(x,x)= k_y(x,0) = k(x,0)=0,\\
			k_x(x,x)=\frac{rx}{3\beta},
		\end{cases}
	\end{eqnarray*}
	on $\Delta_{x,y}$.

\section{Deduction of the kernel pde model  \eqref{kernel_p_old}}
\label{app2}
In this section, we present the details of the calculations for obtaining the kernel model given in \eqref{kernel_p}. Differentiating \eqref{btp} with respect to $t$, we get
\begin{equation*}
	\begin{split}
		i\tilde{u}_t(x,t) =& i\tilde{w}_t(x,t)- \int_0^x i p(x,y)\tilde{w}_t(y,t)dy\\
		=&i\tilde{w}_t(x,t) + \int_0^x p(x,y)(i\beta \tilde{w}_{yyy}(y,t) + \alpha \tilde{w}_{yy}(y,t)+ i\delta \tilde{w}_y(y,t) + ir \tilde{w})dy\\
		=& i \tilde{w}_t(x,t) \\
		&+ i\beta \left( p(x,y)\tilde{w}_{yy}(y,t) - p_y(x,y)\tilde{w}_y(y,t) + p_{yy}(x,y) \tilde{w}(y,t)\Biggr|_0^x \right.\\
		&\left.- \int_0^x p_{yyy}(x,y)\tilde{w}(y,t)dy \right) \\
		&+\alpha \left(p(x,y) \tilde{w}_y(y,t) - p_y(x,y) \tilde{w}(y,t) \Biggr|_0^x + \int_0^x p_{yy}(x,y)\tilde{w}(y,t)dy \right) \\
		&+i\delta \left(p(x,y) \tilde{w}(y,t)\Biggr|_0^x - \int_0^x p_y(x,y)\tilde{w}(y,t)dy\right) \\
		&+ir \int_0^L p(x,y)\tilde{w}(y,t)dy.
	\end{split}
\end{equation*}
Using the boundary conditions $w(0,t)= 0$, rearranging the last expression in terms of $\tilde{w}(x,t)$,$\tilde{w}_x(x,t)$, $\tilde{w}_{xx}(x,t)$, $\tilde{w}_x(0,t)$ and $\tilde{w}_{xx}(0,t)$, we obtain
\begin{equation} \label{wrtt2}
	\begin{split}
		i\tilde{u}_t(x,t) =& i \tilde{w}_t(x,t) + \int_0^x \left(-i\beta p_{yyy} + \alpha p_{yy} - i\delta p_y +ir p\right)(x,y) \tilde{w}(y,t) dy \\
		&+ \left(i\beta p_{yy} (x,x) - \alpha p_y(x,x) + i\delta p(x,x)\right) \tilde{w}(x,t) \\
		&+\left(-i\beta p_y (x,x) + \alpha p(x,x)\right)\tilde{w}_x(x,t) + i\beta p(x,x) \tilde{w}_{xx}(x,t) \\
		&+ \left(i\beta p_y(x,0) - \alpha p(x,0)\right) \tilde{w}_x(0,t) - i\beta p(x,0) \tilde{w}_{xx}(0,t).
	\end{split}
\end{equation}
Next we differentiate \eqref{btp} up to order three and multiply the results by $i\delta$, $\alpha$ and $i\beta$ respectively to obtain
\begin{equation} \label{wrtx2}
	\begin{split}
		i\delta \tilde{u}_x(x,t) &= i\delta \tilde{w}_x(x,t) - i\delta\frac{\partial}{\partial x} \int_0^x  p(x,y)\tilde{w}(y,t)dy \\
		&= i\delta \tilde{w}_x(x,t) - \int_0^x i\delta p_x(x,y) \tilde{w}(y,t)dy - i\delta p(x,x) \tilde{w}(x,t),
	\end{split}
\end{equation}
\begin{equation} \label{wrtxx2}
	\begin{split}
		\alpha \tilde{u}_{xx}(x,t) =& \alpha \tilde{w}_{xx}(x,t) - \alpha\frac{\partial}{\partial x} \int_0^x 		p_x(x,y)\tilde{w}(y,t)dy - \alpha \frac{\partial}{\partial x}\left(p(x,x) \tilde{w}(x,t)\right) \\
		=& \alpha \tilde{w}_{xx}(x,t) - \int_0^x \alpha p_{xx}(x,y)\tilde{w}(y,t)dy \\
		&+\alpha  \left(-p_x(x,x) - \frac{d}{dx} p(x,x)\right) \tilde{w}(x,t) - \alpha p(x,x) \tilde{w}_x(x,t),
	\end{split}
\end{equation}
and
\begin{equation} \label{wrtxxx2}
	\begin{split}
		i\beta \tilde{u}_{xxx}(x,t) =& i\beta \tilde{w}_{xxx}(x,t) - i\beta \frac{\partial}{\partial x}\int_0^x p_{xx}(x,y)\tilde{w}(y,t)dy \\
		&- i\beta \frac{\partial}{\partial x} \left(\left(p_x(x,x) + \frac{d}{dx} p(x,x)\right)\tilde{w}(x,t) + p(x,x)\tilde{w}_x(x,t)\right) \\
		=& i\beta \tilde{w}_{xxx}(x,t) - \int_0^x i \beta p_{xxx}(x,y)\tilde{w}(y,t)dy \\
		&+ i\beta \left(-p_{xx}(x,x) - \frac{d}{dx}p_x(x,x) - \frac{d^2}{dx^2} p(x,x)\right)\tilde{w}(x,t) \\
		&+i\beta \left(-p_x(x,x) - 2\frac{d}{dx} p(x,x)\right)\tilde{w}_x(x,t) - i\beta  p(x,x) \tilde{w}_{xx}(x,t).
	\end{split}
\end{equation}
From \eqref{wrtx2} and \eqref{wrtxx2} we also have
\begin{equation} \label{obs_gain1}
	p_1(x)\tilde{u}_x(0,t) = p_1(x)\tilde{w}_x(0,t)
\end{equation}
and
\begin{equation} \label{obs_gain2}
	p_2(x)\tilde{u}_{xx}(0,t) = p_2(x)(\tilde{w}_{xx}(0,t) - p(0,0)\tilde{w}_x(0,t)).
\end{equation}
Adding \eqref{wrtt2}-\eqref{obs_gain2} side by side we obtain
\begin{align}\label{p_cond_1}
	&i \tilde{u}_t + i\beta \tilde{u}_{xxx} + \alpha \tilde{u}_{xx} + i\delta \tilde{u}_x + p_1(x) \tilde{u}_x(0,t) + p_2(x)\tilde{u}_{xx}(0,t)\nonumber\\
	=&i\tilde{w}_t + i\beta \tilde{w}_{xxx} + \alpha \tilde{w}_{xx} + i\delta \tilde{w}_x \\
	\label{p_cond_2}
	&+ \int_0^x \left(-i\beta(p_{xxx} + p_{yyy}) - \alpha(p_{xx} - p_{yy}) -i\delta (p_x + p_y) + irp\right)(x,y)\tilde{w}(y,t) dy \\
	\label{p_cond_3}
	=& \left(i\beta \left(p_{yy}(x,x) - p_{xx}(x,x) - \frac{d}{dx} p_x(x,x) - \frac{d^2}{dx^2} p(x,x)\right)\right.\nonumber\\
	&\left.+ \alpha \left(-p_y(x,x) - p_x(x,x) - \frac{d}{dx} p(x,x)\right)\right)\tilde{w}(x,t) \\
	\label{p_cond_4}
	&-i\beta \left(p_y(x,x) + p_x(x,x) + 2 \frac{d}{dx }p(x,x)\right)\tilde{w}_x(x,t) \\
	\label{p_cond_5}
	&+\left(i\beta p_y(x,0)- \alpha p(x,0) + p_1(x) - p(0,0)p_2(x)\right)\tilde{w}_x(0,t) \\
	\label{p_cond_6}
	&+(-i\beta p(x,0)+p_2(x)) \tilde{w}_{xx}(0,t).
\end{align}
Note that, taking $x = L$ on \eqref{bt_p} and using the boundary condition $\tilde{u}(L,t) = 0$, we must have $p(L,y) = 0$ in order to get $\tilde{w}(L,t) = 0$. On the other hand, using the relation $\frac{d}{dx}p(x,x) = p_x(x,x) + p_y(x,x)$, we get from \eqref{p_cond_4} that
\begin{equation} \label{p_bc_1}
	\frac{d}{dx} p(x,x) = 0,
\end{equation}
which, thanks to $p(L,y) = 0$  implies
\begin{equation*}
	p(x,x) = 0.
\end{equation*}
Next, we differentiate $\frac{d}{dx}p(x,x) = p_x(x,x) + p_y(x,x)$ with respect to $x$ and use $\frac{d}{dx} p(x,x) = 0$ to obtain $p_{yy}(x,x) = -2p_{xy}(x,x) - p_{xx}(x,x)$. Using this result in \eqref{p_cond_3}, we deduce that
\begin{equation*}
	\frac{d}{dx} p_x(x,x) = -\frac{r}{3\beta}.
\end{equation*}
which, due to the implications $p(L,y) = 0 \Rightarrow p_y(L,y) = 0 \Rightarrow p_y(L,L) = 0 \Rightarrow p_x(L,L) = 0$, is equivalent to
\begin{equation*}
	p_x(x,x) = \frac{r}{3\beta}(L - x).
\end{equation*}
On the other hand we obtain from \eqref{p_cond_5}-\eqref{p_cond_6} that
\begin{align*}
	p_1(x) &= -i\beta p_y(x,0) + \alpha p(x,0) \\
	p_2(x)&= i\beta p(x,0).
\end{align*}
Note that for $x=0$ in \eqref{bt_p}, we have $\tilde{u}(0,t) = \tilde{w}(0,t) = 0$. For $x = L$ and thanks to $p(L,y) = 0$, we have $\tilde{u}(L,t) =\tilde{w}(L,t) = 0$. Also for $x = L$ on \eqref{wrtx2}, we see that $\tilde{w}_x(L,t) = 0$ holds if $p_x(L,y) = 0$.

As a conclusion, the error model is mapped to the target error model (not modified one), if $p$ satisfies the following boundary value problem
\begin{eqnarray*}
	\begin{cases}
		\beta(p_{xxx}+p_{yyy})-i\alpha(p_{xx}-p_{yy})+\delta(p_x+p_y)-rp=0, \\
		p(x,x)= p(L,y) = p_x(L,y)=0,\\
		p_x(x,x) = \frac{r}{3\beta}(L - x),
	\end{cases}
\end{eqnarray*}
on $\Delta_{x,y}$.

\section{Roots of the characteristic equation \eqref{char_eqn}.} \label{app_charroots}
	In this part, we investigate the roots $\lambda_j = \lambda_j(s)$, $j = 1,2,3$, of the characteristic equation
	\begin{equation*}
		s + \beta \lambda^3 + i\alpha \lambda^2 + \delta \lambda = 0
	\end{equation*}
	that is obtained by the one parameter family of boudary value problems \eqref{lap_trans}. More precisely, we show that \eqref{char_eqn} has double or possibly triple roots only for finitely many exceptional cases of the values of $s$ in the complex plane. The location of $s$ in the complex plane is directly related with the sign of the quantity $\alpha^2 + 3\beta\delta$. Also in the following calculations, we drop notation for $s$ dependence and simply write $\lambda_j(s) = \lambda_j$, $j = 1,2,3$ for simplicity.
	
	To this end, let us assume that two roots, say $\lambda_2$ and $\lambda_3$, are equal for some $s \in \mathbb{C}$. Then $\lambda_1$ and $\lambda_2$  satisfy
	\begin{align}
		\label{a21}
		\lambda_1+ 2\lambda_2 &= \frac{i\alpha}{\beta}, \\
		\label{a22}
		2\lambda_1\lambda_2 + \lambda_2^2 &= \frac{\delta}{\beta}, \\
		\label{a23}
		\lambda_1\lambda_2^2 &= - \frac{s}{\beta},
	\end{align}
	where $s \in (r -i\infty,r + i\infty)$ for some $r \in \mathbb{R}^+$. Let $\lambda_1 = a + ib$, $\lambda_2 = \lambda_3 = c+id$, $a, b, c, d$ are real functions of $s$. From the real and imaginary parts of \eqref{a21}-\eqref{a22}, we have following equations:
	\begin{align}
		\label{a24}
		a + 2c &= 0, \\
		\label{a25}
		b + 2d &= \frac{\alpha}{\beta}, \\
		\label{a26}
		2ac - 2bd + c^2 - d^2 &= \frac{\delta}{\beta}, \\
		\label{a27}
		ad + bc + cd &= 0.
	\end{align}
	Using \eqref{a24} in \eqref{a27}, we get $c(b-d) = 0$.
	
	\begin{enumerate}
		\item [(i)] Let $b = d$. By \eqref{a25}, $b = d = \frac{\alpha}{3\beta}$. Substituting these into \eqref{a26} yields
		\begin{equation*}
		c(2a + c) = \frac{\alpha^2 + 3\beta\delta}{3\beta^2}.
		\end{equation*}
		Using \eqref{a24}, we get
		\begin{equation*}
		 a^2 = -\frac{4(\alpha^2 + 3\beta\delta)}{9\beta^2}, \quad c^2 = -\frac{\alpha^2 + 3\beta\delta}{9\beta^2}.
		\end{equation*}
		Assuming $\alpha^2 + 3\beta \delta > 0$ yields a contradiction. Assuming $\alpha^2 + 3\beta \delta = 0$ yields $\lambda_1 = \lambda_2 = \lambda_3 = \frac{i\alpha}{3\beta}$. For this case we see from \eqref{a23} that, the only value for $s$ is pure imaginary and given by
		\begin{equation*}
			s^0 = \frac{i\alpha^3}{27\beta^2}.
		\end{equation*}
		Now let $\alpha^2 + 3\beta \delta < 0$. Then we have
		\begin{equation*}
			a_{1,2} = \mp\frac{2\sqrt{-(\alpha^2 + 3\beta\delta)}}{3\beta}, \quad c_{1,2} = \pm\frac{\sqrt{-(\alpha^2 + 3\beta\delta)}}{3\beta}.
		\end{equation*}
		Note that using \eqref{a24} and $b = d$, we obtain from \eqref{a23} there are two possible values of $s$, denoted by $s_1^-$ and $s_2^-$, which are given by
		\begin{equation*}
			\Re(s_1^-) = \frac{2}{27 \beta^2} (-\alpha^2 - 3\beta\delta)^{3/2}, \quad \Re(s_2^-) = - \frac{2}{27 \beta^2} (-\alpha^2 - 3\beta\delta)^{3/2}
		\end{equation*}
		and
		\begin{equation*}
			\Im(s_1^-) = \Im(s_2^-) = -\frac{\alpha}{27\beta^2} (2\alpha^2 + 9\beta\delta).
		\end{equation*}		
		\item[(ii)] Let $c = 0$. Then, by \eqref{a24}, $a = 0$. Using this, direct calculation from \eqref{a25}-\eqref{a26} yields
		\begin{equation*}
			b_{1,2} = \frac{\alpha \mp 2\sqrt{\alpha^2 + 3\beta\delta}}{3\beta}, \quad d_{1,2} = \frac{\alpha \pm \sqrt{\alpha^2 + 3\beta\delta}}{3\beta}.
		\end{equation*}
		Observe that $\alpha^2 + 3\beta\delta<0$ yields a contradiction and $\alpha^2 + 3\beta\delta = 0$ ends up with the triple root case investigated in (i). Now let $\alpha^2 + 3\beta\delta > 0$. From \eqref{a23}, we see that there are two possible values of $s$, denoted by $s_1^+$ and $s_2^+$, given by
		\begin{equation*}
			s_1^+ = \frac{i}{27\beta^2} \left(\alpha^3 - 3\alpha (\alpha^2 + 3\beta\delta) + 2 (\alpha^2 + 3\beta\delta)^{3/2}\right)
		\end{equation*}
		and
		\begin{equation*}
			s_2^+ = \frac{i}{27\beta^2} \left(\alpha^3 - 3\alpha (\alpha^2 + 3\beta\delta) - 2 (\alpha^2 + 3\beta\delta)^{3/2}\right).
		\end{equation*}
	\end{enumerate}

	\section*{Acknowledgements}. We would like to thank Professor Bing-Yu Zhang (University of Cincinnati) and Professor Shu Ming Sun (Virginia Tech) for their fruitful comments regarding the application of the Laplace transform method that we used in Section 2 to prove the boundary smoothing properties.
	\bibliographystyle{amsplain}

\end{document}